\documentclass[a4paper,11pt]{amsart}
\usepackage[margin=1in]{geometry}
\usepackage[utf8]{inputenc}
\usepackage{amsfonts,amssymb,amsmath,amstext,amscd}
\usepackage{amsthm,cases,color,graphics,tikz}
\usepackage{mathtools}
\usepackage[colorlinks=true,pdfborder={0 0 0}]{hyperref}
\hypersetup{urlcolor=blue,citecolor=red,linkcolor=blue}
\numberwithin{equation}{section}
\newtheorem{lemma}{Lemma}[section]
\newtheorem{theorem}[lemma]{Theorem}
\newtheorem{corollary}[lemma]{Corollary}

\newtheorem{proposition}[lemma]{Proposition}
\newtheorem{remark}{Remark}
\newtheorem{hypothesis}{Hypothesis}
\newcommand{\thishypname}{}
\newtheorem*{generichypothesis}{\thishypname}
\newenvironment{namedhyp}[1]
{\renewcommand{\thishypname}{#1}
  \begin{generichypothesis}}
  {\end{generichypothesis}}
\usetikzlibrary{patterns}
\newcommand{\R}{\mathbb{R}}
\newcommand{\C}{\mathbb{C}}

\newcommand{\dd}{\, {\rm d}}
\newcommand{\dt}{\frac{{\rm d}}{{\rm d}t}}
\newcommand{\ds}{\displaystyle}
\newcommand{\cM}{\mathcal{M}}
\newcommand{\cP}{\mathcal{P}}
\newcommand{\cL}{\mathcal{L}}
\newcommand{\wv}{\lfloor v \rceil}
\newcommand{\wdot}{\lfloor \cdot \rceil}
\newcommand{\wvp}{\lfloor v' \rceil}
\newcommand{\wu}{\lfloor u \rceil}
\renewcommand{\Re}{\operatorname{Re}}
\renewcommand{\Im}{\operatorname{Im}}

\title[Quantitative Fluid Approximation in Transport
Theory]{Quantitative Fluid Approximation in Transport Theory:\\ A
  Unified Approach}

\author{\'Emeric Bouin}
\email{bouin@ceremade.dauphine.fr}
\address{Université Paris-Dauphine, France}

\author{Cl\'ement Mouhot}
\email{c.mouhot@dpmms.cam.ac.uk}
\address{University of Cambridge, UK}

\date{\today}



\begin{document}

\begin{abstract}
  We propose a unified method for the large space-time scaling
  limit of \emph{linear} collisional kinetic equations in the
  whole space. The limit is of \emph{fractional} diffusion type
  for heavy tail equilibria with slow enough decay, and of
  diffusive type otherwise. The proof is constructive and the
  fractional/standard diffusion matrix is obtained. 
  The method combines energy estimates and quantitative spectral
  methods to construct a `fluid mode'. The method is applied to
  scattering models (without assuming detailed balance
  conditions), Fokker-Planck operators and Lévy-Fokker-Planck
  operators. It proves a series of new results, including the
  fractional diffusive limit for Fokker-Planck operators in any
  dimension, for which the formulas for the diffusion coefficient
  were not known, for Lévy-Fokker-Planck operators with general
  equilibria, and for scattering operators including some cases
  of infinite mass equilibria. It also unifies and generalises
  the results of previous papers with a quantitative method, and
  our estimates on the fluid approximation error also seem novel.
\end{abstract}

\maketitle

\tableofcontents


\section{Introduction and main results}

The study of \emph{transport processes}, i.e. linear collisional
kinetic equations, is theoretically rooted in the mean-free path
argument of Maxwell~\cite{maxwell1860process} and the kinetic
theory of gases of Maxwell and
Boltzmann~\cite{maxwell1867iv,boltzmann1872}. A linear version of
the Maxwell-Boltzmann equation can be written for the movement of
a tagged particle within a rarefied gas, but the study of such
transport processes was given a crucial new impetus in the
twentieth century with:
\begin{enumerate}
\item the \emph{radiative transfer theory}~\cite{pomraning1973},
  where the kinetic distribution models the flux of photons that
  are transported in the plasma making up the internal layers of
  the sun,
\item the \emph{nuclear reactor theory} (see~\cite{MR0113336},
  the collection~\cite{Birkhoff-Wigner-1961} and in particular
  its fifth chapter~\cite{MR0119519}) where the kinetic
  distribution models the neutrons transported and scattered
  inside the reactor, whose flux is used to initiate and maintain
  the chain reaction,
\item the \emph{semi-conductor theory}~\cite{MR1063852} where the
  kinetic distribution models the flow of charge carriers in
  semiconductors, i.e. the evolution of the position-momentum
  distribution of negatively charged conduction electrons or of
  positively charged holes, which are responsible for the current
  flow in semiconductor crystals.
\end{enumerate}

The main mathematical object of study in \emph{transport theory}
is the linear equation
\begin{equation} \label{eq:kinetic}
  \partial_t f + v\cdot \nabla_x f = \cL f
\end{equation}
on the time-dependent density of particles $f= f(t,x,v) \ge 0$
over $(x,v) \in \R^d \times \R^d$, for $t \ge 0$. The left hand
side accounts for free motion and the right hand side accounts
for the interaction with a background, for instance scatterers,
with an operator $\cL$ that only acts on the kinetic variable
$v$. Several forms are possible. In nuclear reactor, radiative
transfer and semi-conductor theories it is common to consider
\emph{scattering operators}, sometimes also called \emph{linear
  Boltzmann operators}, of the form
\begin{align}
  \label{eq:scattering}
  & \cL f(v) = \left( \int_{\R^d} b(v,v') f(v') \dd v' \right)
    \cM(v) - \nu(v) f(v) \\ \nonumber
  & \text{ for a \emph{collision frequency} } \quad \nu(v)
    := \int_{\R^d} b(v,v') \cM(v') \dd v',
\end{align}
some \emph{collisional kernel} $b=b(v,v')$ and an
\emph{equilibrium distribution} $\cM(v)$. In astrophysics and
sometimes in semi-conductor theory, one also considers
\emph{Fokker-Planck operators},
\begin{equation}
  \label{eq:FP} 
  \mathcal{L} f := \nabla_v \cdot \left( \cM \nabla_v
    \left( \frac{f}{\cM}\right)\right).
\end{equation}
Finally, as a simplified model of long-range collisional
interactions in a gas of charged particles, we also consider
\emph{Lévy-Fokker-Planck operators} (given $s \in (0,1)$):
\begin{equation}
  \label{eq:LFP}
  \begin{cases}
    \mathcal{L}(f)=
    \Delta_v ^s f + \nabla_v\cdot\left(U\,f\right)
    \quad \text{ with $U(v)=U(|v|)$ radially symmetric so that}
    \\[3mm]
    \Delta_v^s \cM + \nabla_v\cdot\left(U\,\cM\right) = 0.
  \end{cases}
\end{equation}
Denoting by $\mathcal{F}$ the Fourier transform, the fractional
Laplacian is defined as
\begin{align}
  \label{eq:fract-Lap}
  \Delta_v^{s} f(v) := - \mathcal{F}^{-1}\left[
  |\cdot|^{2s} \mathcal{F} f(\cdot) \right](v).
\end{align}
These three operators are discussed respectively in
Sections~\ref{sec:scatt}-\ref{sec:kfp}-\ref{sec:lkfp}. Extensions,
such as Fokker-Planck operators with non-gradient force, are
discussed in Section~\ref{sec:appendix}.

The equation~\eqref{eq:kinetic} is too intricate for many
applications. When the relevant time and space scales of
observation are much larger than the mean free time and mean free
path, it is thus natural to search for a simplified regime. The
so-called \emph{diffusion theory} was born out of this endeavour,
and in the words of Wigner~\cite{MR0119519}, `this [diffusion]
theory gives the spatial variation of the [neutron transport]
flux quite accurately in regions well removed from interfaces'.
We also refer to~\cite[Chapter~IX]{MR0113336} for the diffusion
theory of monoenergetic neutrons,
to~\cite[Chapter~III.2]{pomraning1973} for the so-called
\emph{Eddington approximation} in radiative transfer theory, and
to~\cite[Chapter~2]{MR2065070} for a modern mathematical
review. Note that anomalous diffusions and Lévy flights are
observed by biologists and physicists \cite{ariel_swarming_2015,
  TuGrinstein, BarkaiAghion, Marksteiner, SagiBrook}.

We rewrite~\eqref{eq:kinetic} by changing the unknown to $h:= \frac{f}{\cM}$:
\begin{align}
  \label{eq:kinetich}
  \partial_t h + v \cdot \nabla_v h = Lh
  \quad \text{ where } \quad
  Lh:= \cM^{-1} \cL \left( \cM h \right). 
\end{align}
This change of unknown is convenient since asymptotic estimates
compare $f$ with the equilibrium $\cM$. Consider the complex
Hilbert spaces $L^2(\R^d; \cM \dd v) =: L^2_v(\cM)$ and
$L^2(\R^d \times \R^d; \cM \dd x \dd v) =: L^2_{x,v}(\cM)$ and
denote
$\| h \|_ k := \| (1+|\cdot|^2)^{\frac{k}{2}} h \|_{L^2(\cM)}$
(the integration variable(s) will be emphasized when there is
ambiguity). We omit the index when $k=0$. The scalar product
$\langle \cdot, \cdot \rangle$ refers to $L^2_v(\cM)$ or
$L^2_{x,v}(\cM)$ depending on context.

We assume the following hypotheses for some $\alpha, \beta \in \R$ with $\alpha + \beta > 0$, and some $\lambda \in \R_{+}^{*}$:

\begin{hypothesis}[Equilibria]
  \label{hyp:functional}
  The equilibrium $\cM$ takes one of the following two forms.
  \begin{itemize}
  \item[(i)] Either it is given by
    \begin{equation}
      \label{eq:Mpoly}
      \cM(v) = c_{\alpha,\beta} \wv^{-(d+\alpha)} \text{ with }
      c_{\alpha,\beta} := \left( \int_{\R^d} \wv^{-d-\alpha-\beta} \dd v
      \right)^{-1} \text{ and } \wv := \sqrt{1+|v|^2}.
    \end{equation}
  \end{itemize}
  \begin{itemize}
  \item[(ii)] Or it is a smooth positive radially symmetric
    function decaying faster than any polynomial. This case
    is denoted by `$\alpha = +\infty$' in the sequel.
  \end{itemize}
\end{hypothesis}

Note that the normalisation implies the \textbf{generalised mass condition}
\begin{equation}
  \label{eq:gen-mass}
  \int_{\R^d} \cM_\beta(v) \dd v =1 \quad
  \text{ with } \quad \cM_\beta := \wdot^{-\beta}\cM .
\end{equation}

We present our main results assuming that the equilibrium $\cM$
is given by the exact formula~\eqref{eq:Mpoly} in the case of a
polynomial decay because it leads to a neater treatment. However,
as discussed in Section~\ref{sec:appendix}, our results remain
true with an equilibrium $\cM$ that is not an explicit power-law
or even symmetric or centered, but only comparable to
$\wdot^{-(d+\alpha)}$ (see~\eqref{eq:MpolyS} and
Subsections~\ref{subsec:asymp} and~\ref{subsec:non-cent}); this
requires a few technical changes in the proofs that we present
separately in this last section so as not to clutter the paper.

\begin{hypothesis}[Weighted coercivity]
  \label{hyp:coercivity}
  The operator $L$ is linear, independent of time $t$ and space
  $x$, commutes with rotations in $v$, is closed densely defined
  on $\text{{\em Dom}}(L) \subset L^2_v(\cM)$ and satisfies
  $L(1)=L^*(1)=0$, where $L^*$ is the
  $L^2_v(\cM)$-adjoint. Finally
  $\tilde L := \wdot^{\frac{\beta}{2}} L( \wdot^{\frac{\beta}{2}}
  \cdot )$ is closed densely defined on
  $\text{{\em Dom}}(\tilde L) \subset L^2_v(\cM)$, with the
  spectral gap estimate
  \begin{align*}
    \forall \, g \in \mbox{{\em Dom}}(\tilde L), \quad g \, \bot
    \, \wdot^{-\frac{\beta}{2}}, \quad
    - \Re \, \big\langle \tilde L g, g \big\rangle
    \ge \lambda \, \left\| g \right\|^2.
  \end{align*}
  This means, translating back to $L$, 
  \begin{align*}
    & \forall \, h \in \mbox{{\em Dom}}(L), \quad 
      - \Re \, \big\langle L h,h 
      \big\rangle \ge \lambda \, \left\| h - \cP h
      \right\|_{-\beta} ^2 
    & \text{ with  } \quad  \cP h := \left(
      \int_{\R^d} h(v') \cM_\beta(v') \dd v' \right).
  \end{align*}
\end{hypothesis}

The assumption that $\cL$ commutes with rotations in $v$ is
convenient (and satisfied for most physical models), but in fact
only $\cM(v)=\cM(-v)$ is really used in the proof. The latter
could in turn be relaxed at the price of a few technical changes
in the proofs discussed in Section~\ref{sec:appendix}.

\begin{hypothesis}[Amplitude of collisions at large velocities]
  \label{hyp:large-v}
  Given $0 \le \chi \le 1$ a smooth function that is $1$ on
  $B(0,1)$ and $0$ outside $B(0,2)$, and
  $\chi_R=\chi(\frac{\cdot}{R})$ and $\tilde \chi_R=(v\cdot \sigma) \wv^\beta \chi_R$ for $R \ge 1$, 
  \begin{equation*}
     \left\Vert L \left( \chi_R \right) \right\Vert_\beta
     \lesssim R^{-\frac{\alpha+\beta}{2}} \quad \text{and} \quad \left\Vert L \left( \tilde \chi_R \right) \right\Vert_\beta
     \lesssim \begin{cases}
       R^{1+\beta-\frac{\alpha+\beta}{2}} & \text{when } \alpha \in (-\beta,2+\beta) \\
       (\ln R)^{\frac12} & \text{when } \alpha=2+\beta.
     \end{cases}
   \end{equation*}
\end{hypothesis}

Our first result, on the basis of the three previous hypotheses,
is a quantitative construction of a branch of `fluid eigenmode'
in the asymptotic of large time and small spatial frequencies,
i.e. a unique eigenvalue branching from zero for
$\tilde L^*+i\eta \wv^\beta (v\cdot \sigma)$ for small $\eta$
(see Figure~\ref{fig:spectre}):
\begin{lemma}[Construction of the fluid mode]
  \label{lem:existencespectral}
  Given Hypotheses~\ref{hyp:functional}--\ref{hyp:coercivity}--\ref{hyp:large-v},
  there are $\eta_0>0$ and $r_0 \in (0,\lambda)$, explicit in
  terms of the constants in these hypotheses, such that for any
  $\eta \in (0,\eta_0)$ and any $\sigma \in \mathbb{S}^{d-1}$,
  there is a unique solution
  $\phi_\eta =\phi_\eta(v) \in L^2_v(\wdot^{-\beta} \cM)$ and
  $\mu(\eta) \in B(0,r_0)$ to
  \begin{equation*}
    - L^* \phi_{\eta} - i \eta (v \cdot \sigma) \phi_{\eta}=
    \mu(\eta) \wv^{-\beta} \phi_{\eta} \quad \text{ with } \quad
    \int_{\R^d} \phi_{\eta}(v) \, \cM_\beta(v) \dd v = 1.
  \end{equation*}
Moreover, the branch $(\phi_\eta,\mu(\eta))$ connects to
  $\left( 1, 0 \right)$ as $\eta \to 0$, with $\mu(\eta) >0$ and
  the asymptotics
  \begin{equation}
    \label{eq:phi-beta}
    \Vert \phi_\eta -1 \Vert_{-\beta} \lesssim \mu(\eta)^\frac12
    \quad
    \text{ and } \quad \mu(\eta) \in (\textsc{r}_0 \Theta(\eta),
    \textsc{r}_1 \Theta(\eta))
  \end{equation}
  for some $0<\textsc{r}_0<\textsc{r}_1$, where the function
  $\Theta$ is defined by
  \begin{equation}
    \label{eq:def-Theta}
    \Theta(\eta) :=
    \begin{cases}
      \eta^2
      &\text{when } \alpha > 2 + \beta,\\[2mm]
      \eta^2 \vert \ln(\eta) \vert
      &\text{when } \alpha = 2 + \beta,\\[2mm]
      \eta^{\frac{\alpha+\beta}{1+\beta}} &\text{when } -\beta<\alpha
      < 2 + \beta.
    \end{cases}
  \end{equation}
\end{lemma}

\begin{figure}[!ht]
  \label{fig:spectre}
  \includegraphics{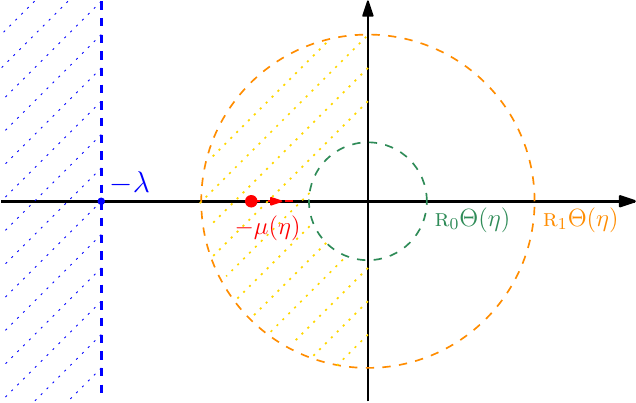}
  \caption{The blue dashed zone on the left of
    $\Re z = - \lambda$ corresponds to the spectral gap estimates
    on $\tilde L^*+i\eta \wv^\beta (v\cdot \sigma)$ for
    $g \, \bot \, \wdot^{-\frac{\beta}{2}}$
    (Hypothesis~\ref{hyp:coercivity}). The yellow dashed zone is
    where
    Lemmas~\ref{lem:existencespectral}-\ref{lem:ratespectral}
    construct a unique real eigenvalue
    $-\mu(\eta) \sim -\mu_0 \Theta(\eta)$ of the operator $\tilde L^*+i\eta \wv^\beta (v\cdot \sigma)$, that goes to zero as $\eta \to 0$.}
\end{figure}

Note that $\Theta$ is well-defined in the case
$\alpha \in (-\beta,2+\beta)$ since
$(1+\beta) > (\alpha+\beta)/2>0$. In this lemma and in the rest
of the paper the dependency in $\sigma$ is kept implicit rather
than explicit in order to lighten notation. In fact, $\phi_\eta$
also depends on $\sigma$, but $\mu(\eta)$ does not if $L$ is
invariant by rotations in $v$.  To identify the macroscopic limit
with quantitative rates and constants, it is necessary to
estimate the leading order of $\mu(\eta)$, and this requires
estimates on the eigenvector, which is our last hypothesis. We
denote
$|u|_\eta := ( \eta^{\frac{2}{1+\beta}} + |u|^2)^{\frac12}$.

\begin{hypothesis}[Scaling of the fluid mode]
  \label{hyp:scalinginfinite}
  We make different assumptions depending on $\alpha$:
  \begin{itemize}
  \item[(i)] \underline{Case $\alpha > 2+\beta$}: The fluid mode
    $\phi_\eta$ constructed in Lemma~\ref{lem:existencespectral}
    satisfies
  \begin{align*}
    \forall \, \ell < \alpha, \quad
    \left\| \phi_\eta  \right\|_{\ell} \lesssim_\ell 1. 
  \end{align*}

\item[(ii)] \underline{Case $\alpha \in (-\beta,2+\beta]$}: The
  rescaled fluid mode
  $\Phi_\eta := \phi_\eta(\eta^{-\frac{1}{1+\beta}} \cdot )$ is
  converging in
  $L_{\text{\tiny {\em loc}}}^2(\R^d \backslash{0})$ as
  $\eta \to 0$ to a limit $\Phi$ and satisfies the pointwise
  controls
    \begin{equation}
      \label{eq:mmt-fract}
      \forall \, \eta \in (0,\eta_1), \ \forall \, u \in \R^d,
      \quad
      \begin{cases}
        \left| \Phi_\eta(u) \right| \lesssim |u|_\eta^{C
          \mu(\eta)}, \\[2mm]
        \left| \Im \Phi_\eta(u) \right| \lesssim
        |u|_\eta^{\beta+\min(\alpha ,1)-\delta}
      \end{cases}
    \end{equation}
    for some $\eta_1 \in (0,\eta_0)$ and $C>0$ and
    $\delta < \beta + \min(2- \alpha,1)$. We also make the
    following additional assumptions in the two following
    subcases:
  \begin{itemize}
  \item[(ii)-(a)] \underline{Case $\alpha =2 + \beta$}: There are
    $\mathfrak{a} : \R^+ \to \R^+$, satisfying $\lim_{\eta \to 0}
    \mathfrak{a}(\eta) = 0$ and $\Omega : \R^d \to \R$ locally integrable
    such that
    \begin{equation*}
      \begin{dcases}
        \left\vert \int_{1 \geq \vert u \vert \geq
            \eta^{\frac{1}{1+\beta}}} (u \cdot \sigma) \Big[
          \Im\Phi_\eta(u) - \Im\Phi(u) \Big] |u|_\eta
          ^{-d-\alpha} \dd
          u \right\vert \leq \mathfrak{a}(\eta) |\ln(\eta)|, \\
        \forall \, \sigma' \in \mathbb{S}^{d-1}, \quad \frac{\Im
          \Phi \left( \lambda \sigma' \right)}
        {\lambda^{1+\beta}} \xrightarrow[\lambda \to 0]{\lambda
          \not = 0} \Omega(\sigma') \quad \text{ in } \quad
        L^1(\mathbb{S}^{d-1}).
      \end{dcases}
    \end{equation*}
    \smallskip

  \item[(ii)-(b)] \underline{Case $\alpha \in (-\beta,\beta]$}: The
    additional following integral control holds:
    \begin{equation}
      \label{eq:mmt-fract-bis}
      \int_{\vert u \vert \geq 1} \left|\Phi_\eta(u)\right|^2
      |u|_\eta ^{-d-\alpha+\beta} \dd u  \lesssim 1.
    \end{equation}
  \end{itemize}
\end{itemize}
\end{hypothesis}
Note that in~\eqref{eq:mmt-fract},
$|u|_\eta^{C \mu(\eta)} \sim 1$ as $\eta \to 0$ in the region
$|u| \lesssim \eta^{\frac{1}{1+\beta}}$. Note also
that~\eqref{eq:mmt-fract} and~\eqref{eq:phi-beta} imply
$\Phi(0)=1$. The case (a) in~(ii) above is subtle and made
necessary by the fact that the case $\alpha = 2 + \beta$ is
borderline between two different regimes (standard diffusion
vs. fractional diffusion) as well as borderline between two
different scalings for obtaining the diffusion coefficient (fluid
mode in variable $v$ vs. fluid mode in the rescaled variable
$u=\eta^{-\frac{1}{1+\beta}}v$).

With these four hypotheses we can characterise the precise
scaling of the fluid eigenvalue:
\begin{lemma}[Rescaled limit of the fluid eigenvalue]
  \label{lem:ratespectral}
  Assume  Hypotheses~\ref{hyp:functional}--\ref{hyp:coercivity}--\ref{hyp:large-v}--\ref{hyp:scalinginfinite}. The eigenvalue
  $\mu(\eta)$ constructed in Lemma \ref{lem:existencespectral}
  satisfies (with convergence rate explicit in terms of the
  constants, error terms and convergence rates in the hypotheses)
  \begin{equation}
    \label{eq:scalingmu}
    \mu(\eta)\sim_{\eta \to 0} \mu_0 \Theta(\eta),
  \end{equation}
  where the constant $\mu_0 \in (\textsc{r}_0,\textsc{r}_1)$ is
  positive and determined as follows:
  \begin{equation*}
    \begin{dcases}
      & \mu_0 := \int_{\R^d} \left(v \cdot \sigma \right)
      F(v) \cM(v) \dd v \quad \text{when } \alpha > 2+\beta, \\
      & \qquad \text{where } \ F = \lim_{\eta \to 0} \frac{\Im
        \phi_\eta}{\eta} \text{ is solution to } L F = - (v\cdot
      \sigma) \ \text{ and } \
      \int_{\R^d} F(v)\, \cM_\beta(v) \dd v =0, \\
      & \mu_0 := \frac{c_{2+\beta,\beta}}{1+\beta}
      \int_{\mathbb{S}^{d-1}} (\sigma \cdot \sigma')
      \Omega(\sigma')
      \dd \sigma' \quad \text{when } \alpha = 2 + \beta, \\
      & \qquad \text{where } \ \Omega(u) = \lim_{\lambda \to 0, \
        \lambda \not = 0} \frac{\Im \Phi \left( \lambda u
        \right)}{\lambda^{1+\beta}} \ \text{ and } \ \Phi =
      \lim_{\eta \to 0} \Phi_\eta = \lim_{\eta \to 0}
      \phi_\eta \left( \eta^{-\frac{1}{1+\beta}} \cdot \right),\\
      & \mu_0 := c_{\alpha,\beta} \int_{\R^d} (u \cdot \sigma)
      \Im\Phi(u)|u|^{- d -\alpha} \dd u \quad \text{when } \alpha
      \in (-\beta,2+\beta).
    \end{dcases}
  \end{equation*}
\end{lemma}

Note how, when $\alpha > 2+\beta$, the function $F$ used in the previous works on standard diffusive limit (usually with $\beta=0$) is recovered here as a limit of our fluid mode; this allows our proof to track the convergence rate.

We now assume $\alpha \ge 0$ and define the \emph{diffusion
  exponent}
\begin{equation}
  \label{eq:exponent}
  \zeta = \zeta(\alpha,\beta):=
  \begin{cases}
    2 & \text{ when } \alpha \ge 2+\beta \\[2mm] \ds
    \frac{\alpha +\beta}{1+\beta}
    & \text{ when } \alpha < 2+\beta, 
  \end{cases}
\end{equation}
and the \emph{scaling
  function}
\begin{equation}
  \label{eq:scaling-function}
  \theta(\varepsilon) := 
  \begin{cases} \ds
    \varepsilon^{\zeta}
    & \text{ when } \alpha \in (0,+\infty]
    \setminus \{2+\beta\},
    \\[3mm] \ds
    \varepsilon^2 \vert \ln \varepsilon \vert
    & \text{ when } \alpha = 2+\beta,\\[3mm] \ds
    \frac{\varepsilon^{\frac{\beta}{1+\beta}}}{\vert \ln \varepsilon \vert}
    & \text{ when } \alpha = 0.
  \end{cases}
\end{equation}
Note that the threshold $\alpha=2+\beta$ between standard and
fractional diffusion corresponds to whether or not $\cM_\beta$
has finite variance. We finally derive the \emph{diffusion
  coefficient}:

\begin{lemma}[Diffusion coefficient]\label{lem:diffcoeff}
  Assume Hypotheses~\ref{hyp:functional}--\ref{hyp:coercivity}--\ref{hyp:large-v}--\ref{hyp:scalinginfinite}
  and $\alpha \ge 0$. Then the following limit holds true with
  convergence rate explicit in terms of the constants, error
  terms and convergence rates in the hypotheses: for any
  $\xi \in \R^d \setminus \{0\}$,
  \begin{align}
    \label{eq:coeff}
    \kappa := \lim_{\varepsilon \to 0} \left( \frac{\mu(\varepsilon |\xi|)
    \vert \xi \vert^{-\zeta}}{\theta(\varepsilon)\left\langle 1,
    \phi_{\varepsilon |\xi|} \right\rangle} \right)
    = \mu_0 \times
    \begin{cases}
      \ds \| \cM \|_{L^1(\R^d)}^{-1}
      &\text{ when } \alpha > 0, \\[4mm]
      \ds
      \frac{1+\beta}{|\mathbb{S}^{d-1}|}
      &\text{ when } \alpha =0.
    \end{cases}
  \end{align}
\end{lemma}

The diffusion coefficient thus emerges from ratios between
(rescaled) integrals as follows:
\begin{align}
  \label{eq:coeff-full}
  \kappa :=
  \begin{cases}
    \ds
    \frac{\ds \int_{\R^d} \left(v \cdot \sigma \right)
    F(v) \cM(v) \dd v}{\ds \| \cM \|_{L^1(\R^d)}}
    \hspace{2.8cm} \boxed{\text{ when } \alpha > 2+\beta}
    \\[9mm]
    \ds \frac{1}{1+\beta} 
    \frac{\ds \int_{\mathbb{S}^{d-1}} (\sigma \cdot \sigma')
    \Omega(\sigma')
    \dd \sigma'}{\ds \int_{\R^d} \wv^{-d-\alpha} \dd v}
    \hspace{2.2cm}
    \boxed{\text{ when } \alpha =
    2+\beta} \\[9mm]
    \ds \frac{\ds \int_{\R^d} (u \cdot \sigma)
    \Im\Phi(u)|u|^{- d -\alpha}
    \dd u}{\ds \int_{\R^d} \wv^{-d-\alpha} \dd v} \hspace{2.1cm}
    \boxed{\text{ when } \alpha \in
    (0,2+\beta)} \\[9mm]
    \ds
    \frac{1+\beta}{|\mathbb{S}^{d-1}|} 
    \frac{\ds
    \int_{\R^d} (u \cdot \sigma) \Im\Phi(u)|u|^{- d -\alpha}
    \dd u}{\ds \int_{\R^d} \wv^{-d-\alpha-\beta} \dd v}
    \hspace{1.05cm}
    \boxed{\text{ when } \alpha =0}
  \end{cases}
\end{align}
where we recall
\begin{align*}
  F = \lim_{\eta \to 0} \frac{\Im
  \phi_\eta}{\eta}, \qquad
  \Phi = \lim_{\eta\to 0} \Phi_\eta = \lim_{\eta \to 0} \phi_\eta
  \left( \eta^{-\frac{1}{1+\beta}} \cdot \right), \qquad
  \Omega(u) = \lim_{\lambda \to
  0, \ \lambda \not = 0} \frac{\Im \Phi \left( \lambda u
  \right)}{\lambda^{1+\beta}}, 
\end{align*}
and (when $\alpha > 2+\beta$) $F$ is also the unique solution to
$L F = - (v\cdot \sigma)$ with
$\int_{\R^d} F(v)\, \wv^{-d-\alpha-\beta} \dd v =0$. For
legibility again, we wrote, in the cases
$\alpha \in [0,2+\beta]$, the formula for $\kappa$ with $\cM$
given by~\eqref{eq:Mpoly}, and we refer to
Section~\ref{sec:appendix} for more general $\cM$'s. The proof
of Lemma~\ref{lem:diffcoeff} is done in
Section~\ref{sec:diff-coef}; it requires the estimating of
$\left\langle 1, \phi_{\eta} \right\rangle$, which is done in
Lemma~\ref{lem:zeromoment}.

Consider a solution $f$ in
$L^\infty_t([0,+\infty); L^2_{x,v}(\cM^{-1}))$ to~\eqref{eq:kinetic} with initial data
$f_{\text{\rm in}} ^{(\varepsilon)}$. Note that the initial data
$f_{\text{\rm in}} ^{(\varepsilon)}$, before the rescaling, is
allowed to depend on $\varepsilon$. Given $\varepsilon >0$ and
$\theta(\varepsilon)$ defined in~\eqref{eq:scaling-function}, we
rescale the solution and define a weighted rescaled spatial
density:
\begin{align*}
  & f_\varepsilon(t,x,v) := \frac{1}{\varepsilon^d} f\left(
    \frac{t}{\theta(\varepsilon)},
  \frac{x}{\varepsilon}, v \right) \in
  L^\infty_t\left([0,+\infty); L^2_{x,v}(\cM^{-1}) \right) \\
  & r_\varepsilon(t,x) := \int_{\R^d} f_\varepsilon(t,x,v)
    \wv^{-\beta} \dd v.
\end{align*}
The equation satisfied by $f_\varepsilon$ is
\begin{equation}
  \label{eq:gvar}
  \theta(\varepsilon) \partial_t f_\varepsilon + \varepsilon v
  \cdot \nabla_x f_\varepsilon = \cL f_\varepsilon.
\end{equation}
The \emph{rescaled} initial data is then
$f_\varepsilon(0,x,v) = \varepsilon^{-d} f_{\text{\rm in}}
^{(\varepsilon)}(\varepsilon^{-1} x, v)$, and in the following
theorem we assume the original initial data
$f^{(\varepsilon)} _{\text{\rm in}}$ to be \emph{well-prepared}
(see~\eqref{eq:initial-data}-\eqref{eq:initial-layer-micro}-\eqref{eq:initial-layer-fluid}):
this means that the fluid limit holds at time zero with
$f_\varepsilon(0,\cdot) \sim r_\varepsilon(0,\cdot) \cM$ and
$r_\varepsilon \sim r(0,\cdot)$ which provides the initial data
for the limit equation; this is standard in the literature. We
however note that when~\eqref{eq:initial-data} is satisfied
but~\eqref{eq:initial-layer-micro}-\eqref{eq:initial-layer-fluid}
are not imposed at $t=0$, the energy estimate and compactness
arguments on $r_\varepsilon$ would imply
that~\eqref{eq:initial-layer-micro}-\eqref{eq:initial-layer-fluid}
are satisfied at any later positive time $\tau >0$ (without
information on the rate though), and our method would prove the
fluid approximation for $t \ge \tau$. This would allow us for
instance to choose $f^{(\varepsilon)} _{\text{\rm in}} = r \cM$
independent of $\varepsilon$. We however kept the
assumptions~\eqref{eq:initial-layer-micro}-\eqref{eq:initial-layer-fluid} in order to precisely track the rate of convergence and the
initial data of the limit equation.

\begin{theorem}[Unified second fluid approximation, see
  Figure~\ref{fig:schema}]
  \label{theo:main}
  Assume Hypotheses~\ref{hyp:functional}--\ref{hyp:coercivity}--\ref{hyp:large-v}--\ref{hyp:scalinginfinite},
  $\alpha \ge 0$, and consider
  $f_\varepsilon \in L^\infty_t([0,+\infty);L^2_{x,v}(\cM^{-1}))$
  solving~\eqref{eq:kinetic} in the weak sense with initially
    \begin{equation}
    \label{eq:initial-data}
    \left\| \frac{f_\varepsilon(0,\cdot,\cdot)}{\cM} \right\|_0 = o\left(
      \begin{cases} \ds
        \theta(\varepsilon)^{-\frac12} ,
        & \text{ when } \alpha > \beta,
        \\[3mm]
        \varepsilon^{-\frac{\beta}{1+\beta}} \left| \ln \left( \varepsilon
          \right) \right|^{-\frac12} 
        & \text{ when } \alpha = \beta,
        \\[3mm] \ds
        \varepsilon^{ -\frac{\alpha}{1+\beta}}
        & \text{ when } \alpha \in (0,\beta),
        \\[3mm] \ds
        | \ln(\varepsilon)|^\frac32
        & \text{ when } \alpha = 0,
      \end{cases}
  \right)
  \end{equation}
  \\ \, \text{ and }
  \begin{equation}
    \label{eq:initial-layer-micro}
    \left\| \frac{f_\varepsilon}{\cM}(0,\cdot,\cdot)
      - r_\varepsilon(0,\cdot)
    \right\|_{-\beta} = o \left(
      \begin{cases} \ds
        1,
        & \text{ when } \alpha > \beta,
        \\[3mm]
        \left| \ln \left( \varepsilon
          \right) \right|^{-\frac12}
        & \text{ when } \alpha = \beta,
        \\[3mm] \ds
        \varepsilon^{ \frac{\beta - \alpha}{2(1+\beta)}} 
        & \text{ when } \alpha \in (0,\beta),
        \\[3mm] \ds
        \varepsilon^{\frac{\beta}{2(1+\beta)}}| \ln(\varepsilon)| 
        & \text{ when } \alpha = 0,
      \end{cases}
    \right),
  \end{equation}
  and (recalling the definition of $\zeta$ in~\eqref{eq:exponent})
  \begin{equation}
    \label{eq:initial-layer-fluid}
    r_\varepsilon(0,\cdot) \xrightarrow[\varepsilon \to
    0]{H^{-\zeta}(\R^d)} r(0,\cdot).    
  \end{equation}
  Then for any $T>0$ 
  \begin{equation}
    \label{eq:main-conv}
    \left\| \frac{f_\varepsilon}{\cM} - r
    \right\|_{L^2_t([0,T];H^{-\zeta}_x L^2_v(\cM_\beta))}
    \xrightarrow[\varepsilon \to 0]{} 0 
  \end{equation}
  when $\alpha > \beta$ and
  \begin{equation*}
    \left\| \left|\ln \frac{2|\nabla_x|}{1+|\nabla_x|}
      \right| \left( \frac{f_\varepsilon}{\cM} - r \right)
    \right\|_{L^2_t([0,T];H^{-\zeta}_x L^2_v(\cM_{\beta}))}
    \xrightarrow[\varepsilon \to 0]{} 0 
  \end{equation*}
  when $\alpha =\beta$ and
  \begin{equation*}
    \left\|
      |\nabla_x|^{\frac{\beta-|\alpha|}{2(1+\beta)}}
      \lfloor \nabla_x \rceil^{-\frac{\beta-|\alpha|}{2(1+\beta)}}
      \left( \frac{f_\varepsilon}{\cM} - r \right)
    \right\|_{L^2_t([0,T];H^{-\zeta}_x L^2_v(\cM_{\beta}))}
    \xrightarrow[\varepsilon \to 0]{} 0 
  \end{equation*}
  when $\alpha \in [0,\beta)$, where $r=r(t,x)$ solves
  \begin{equation*}
    \partial_t r = \kappa \, \Delta_x^{\frac{\zeta}{2}} r, \quad
    t>0, \quad \text{ with initial data $r(0,\cdot)$ defined
      in~\eqref{eq:initial-layer-fluid}}.
  \end{equation*}  
  The rates of convergence are estimated in terms of $T$, the
  constants, error terms and convergence rates in Hypotheses~  \ref{hyp:functional}--\ref{hyp:coercivity}--\ref{hyp:large-v}--\ref{hyp:scalinginfinite},
  and the initial convergence rates~\eqref{eq:initial-data}-\eqref{eq:initial-layer-micro}-\eqref{eq:initial-layer-fluid}. Apart
  from~\eqref{eq:initial-layer-fluid}, the errors we obtain are polynomial in $\varepsilon$ for $\alpha \in (-\beta,+\infty) \setminus \{0,2+\beta\}$ and logarithmic for $\alpha \in \{0,2+\beta\}$.
\end{theorem}

This theorem is the core contribution of the paper, and is used
to obtain results on concrete models in the corollaries
below. Together with
Lemmas~\ref{lem:existencespectral}--\ref{lem:ratespectral}--\ref{lem:diffcoeff},
it reveals the relevant macroscopic scales for a large class of
operators in any dimension and provides a unified theoretical
framework to answer questions of the last decades on the
topic. The diffusive limit is reduced to a spectral problem --the
construction of the fluid mode-- that we solve in a general
setting. The proof is constructive and the key constants
governing the macroscopic behaviours are derived. The fractional
Laplacian in the space variable is defined as
in~\eqref{eq:fract-Lap}, and $r(t,x)$ is the limit (in the
topology of the above theorem) of the weighted velocity average
\begin{equation*}
  r_\varepsilon(t,x) = 
  \int_{\R^d} f \left( \frac{t}{\theta(\varepsilon)},
    \frac{x}{\varepsilon}, v \right) \wv^{-\beta}  \dd v.
\end{equation*}
When $\alpha > 0$, the density
$\rho_\varepsilon (t,x) := \int_{\R^d} f \left(
  \frac{t}{\theta(\varepsilon)}, \frac{x}{\varepsilon}, v \right)
\dd v$ exists. When $\alpha>\beta$, it is straightforward that
$\rho_\varepsilon \to \| \cM \|_{L^1_v} r$ by Hölder's inequality
and the convergence~\eqref{eq:main-conv}. When $\alpha \in (0,\beta)$, the latter still holds under the slightly stronger assumption $|f_\varepsilon(0,x,\cdot)| \le C \cM$ for some $C>0$ on the
initial data. Indeed then $|f_\varepsilon(t,x,\cdot)| \le C \cM$
by the comparison principle, and
\begin{equation*}
  \forall \, t \ge 0, \ x \in \R^d, \quad
  \left\| \frac{f_\varepsilon(t,x,\cdot)}{\mathcal M(\cdot)} - r
  \right\|_{L^2_v(\mathcal M)} \lesssim 1
\end{equation*}
so we have by Hölder's inequality
\begin{align*}
  \left| \rho_\varepsilon(t,x) - \| \cM \|_{L^1_v} r(t,x) \right|
  & = \left| \int_{\R^d} \left( f_\varepsilon(t,x,v) - r \mathcal
    M \right) \dd v \right|\\ 
  & \le \left\| \mathcal M(\cdot) \wdot^{\frac{\alpha}{2}}
    \right\| \times \left\|
    \frac{f_\varepsilon(t,x,\cdot)}{\mathcal M(\cdot)} - r
    \right\|_{L^2_v(\mathcal M_{\frac{\alpha}{2}})} \\
  & \le \left\| \mathcal M(\cdot) \wdot^{\frac{\alpha}{2}}
    \right\| \left\|
    \frac{f_\varepsilon(t,x,\cdot)}{\mathcal
    M(\cdot)} - r \right\|_{L^2_v(\mathcal M)} ^{1-\frac{2
    \beta}{\alpha}} \left\| \frac{f_\varepsilon(t,x,\cdot)}{\mathcal
    M(\cdot)} - r \right\|_{L^2_v(\mathcal M_\beta)}
    ^{\frac{2 \beta}{\alpha}} \\
  & \lesssim \left\|
    \frac{f_\varepsilon(t,x,\cdot)}{\mathcal
    M(\cdot)} - r \right\|_{L^2_v(\mathcal M_\beta)}
    ^{\frac{2 \beta}{\alpha}}
\end{align*}
which implies $\rho_\varepsilon \to \| \cM \|_{L^1_v} r$ by
integrating against decaying test functions in $x$ and
using~\eqref{eq:main-conv}.

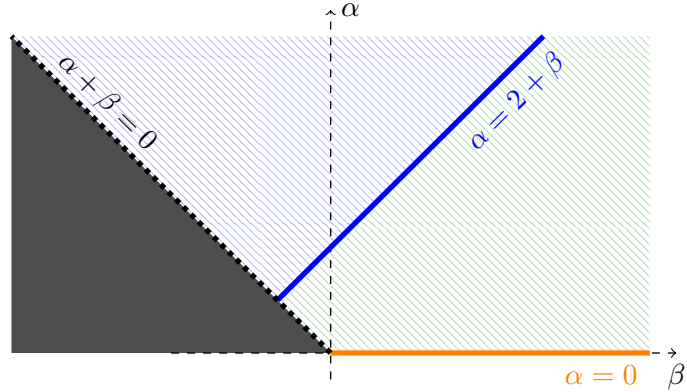
\begin{figure}[!ht]
  \begin{tikzpicture}[scale=0.7]
    \fill[pattern=north west lines,pattern color=black!50!green,
    opacity=0.5] 
    (0,0)--(6,0)--(6,6)--(4,6)--(-1,1) -- cycle;
    \fill[pattern=north west lines, pattern color=blue,
    opacity=0.5] 
    (-6,6)--(-1,1)--(4,6) -- cycle;
    \fill[fill=black, opacity=0.7] 
    (-6,6)--(0,0)--(6,0)--(-6,0) -- cycle;
    \draw[->,line width=0.5, dashed](-3,0)--(6.5,0)
    node[below]{$\beta$};
    \draw[->,line width=0.5, dashed](0,-0.5)--(0,6.5)
    node[right]{$\alpha$};
    \draw [line width=2,domain=-1:4, samples=150, blue] plot
    (\x, 2+\x)
    node[below]{\rotatebox{45}{$\hspace{-1cm}\alpha=2+\beta$}};
    \draw [line width=2,orange] (0,0)--(6,0) node[anchor=north
    east]{$\alpha = 0$}; \draw [line width=2,black,dotted]
    (0,0)--(-6,6) node[anchor = north west]{\quad
      \rotatebox{-45}{$\alpha+\beta=0$}};
  \end{tikzpicture}
  \caption{Summary of the results in the $(\alpha,\beta)$ plane. Admissible parameters are in the half-plane $\alpha+\beta > 0$. The blue hatched area leads to $\theta(\varepsilon) = \varepsilon^2$ and a standard diffusive limit with symbol $\kappa \vert \xi \vert^{2}$. The blue line is the set of parameters yielding the anomalous scaling $\theta(\varepsilon) = \varepsilon^2 \vert \ln(\varepsilon)\vert$ but still a standard diffusive limit with symbol $\kappa \vert \xi \vert^{2}$. The green hatched area results into the fractional scaling $\theta(\varepsilon) =   \varepsilon^{\frac{\alpha+\beta}{1+\beta}}$ and a fractional   diffusive limit with symbol $\kappa \vert \xi \vert^{\frac{\alpha+\beta}{1+\beta}}$. The orange bold line yields the fractional scaling $\theta(\varepsilon) = \varepsilon^{\frac{\beta}{1+\beta}} \vert \ln(\varepsilon) \vert^{-1}$ and a fractional diffusive limit with symbol $\kappa \vert \xi \vert^{\frac{\beta}{1+\beta}}$.}
  \label{fig:schema}
\end{figure}

We now apply the abstract Theorem~\ref{theo:main} to particular models:
\begin{corollary}[Scattering equation]
  \label{cor:scat}
  Assume that $\cL$ is the scattering
  operator~\eqref{eq:scattering} with $b \in \mathcal{C}^1$ and
  $\cM$ satisfying Hypothesis~\ref{hyp:functional} and that, for
  some constant $\nu_0 >0$ and $\alpha \ge 0$ and
  $\alpha + \beta >0$,
  \begin{equation}
    \label{eq:hyp-nu} 
    \begin{cases}
      \forall \, v \in \R^d, \quad
      \wv^{-\beta} \lesssim \nu(v) \lesssim
      \wv^{-\beta} \\[3mm]
      \forall \, v \in \R^d \setminus \{0\}, \quad
      \lambda^\beta \nu (\lambda v) \sim_{\lambda \to \infty} \nu_0
      |v|^{-\beta} \\[3mm]
      \forall \, v \in \R^d, \quad
      \Vert b(v,\cdot) \Vert_\beta + \Vert b(\cdot,v) \Vert_\beta
      \lesssim \wv^{-\beta}.
    \end{cases}
  \end{equation}
  This includes $b(v,v') = \wv^{-\beta}\wvp^{-\beta}$, and
  $b(v,v') = \lfloor v -v' \rceil^{-\beta}$ when $\beta <0$ and
  $\alpha + \beta >0$ or when $\beta \ge 0$ and
  $\alpha > 3 \beta$. Then Theorem~\ref{theo:main} applies with
  $\alpha,\beta$ given in Hypotheses~\ref{hyp:functional}
  and~\eqref{eq:hyp-nu}. This proves the diffusive limit for
  solutions to~\eqref{eq:gvar} satisfying~~\eqref{eq:initial-data}-\eqref{eq:initial-layer-micro}-\eqref{eq:initial-layer-fluid} with quantitative rate, diffusion
  exponent $\zeta = \frac{\alpha+\beta}{1+\beta}$, scaling
  function~\eqref{eq:scaling-function} and diffusion
  coefficient~\eqref{eq:coeff-full}. Apart from~\eqref{eq:initial-layer-fluid},
  the errors we obtain are polynomial in $\varepsilon$ for
  $\alpha \in (0,+\infty) \setminus \{2+\beta\}$ and logarithmic
  for $\alpha \in \{0,2+\beta\}$. Moreover the diffusion
  coefficient can be computed explicitly with
  $F(u) = \nu(v)^{-1}(v\cdot \sigma)$ when $\alpha > 2+\beta$,
  $\Omega(u) = \nu_0^{-1} \vert u \vert^{\beta} (u \cdot \sigma)$
  when $\alpha = 2+\beta$ and
  \begin{equation*}
    \Phi(u) = \frac{\nu_0}{\nu_0 - i \vert u \vert^{\beta} (u
      \cdot \sigma)}
  \end{equation*}
  when $\alpha \in (-\beta,2+\beta)$, resulting in
  \begin{equation*}
    \label{eq:coeff-scattering}
    \kappa :=
    \begin{cases}
      \ds \frac{\ds \int_{\R^d} \left(v \cdot \sigma \right)^2 
      \nu(v)^{-1} \wv^{-d-\alpha} \dd v}{\ds \int_{\R^d}
      \wv^{-d-\alpha} \dd v}
      \hspace{2cm} \boxed{\text{ when } \alpha \in
      (2+\beta,+\infty)} \\[10mm]
      \ds \frac{1}{\nu_0(1+\beta)} 
      \frac{\int_{\mathbb{S}^{d-1}} (\sigma \cdot \sigma')^2 
      \dd \sigma'}{\ds \int_{\R^d} \wv^{-d-\alpha} \dd v}
      \hspace{2.45cm}
      \boxed{\text{ when } \alpha =
      2+\beta} \\[10mm]
      \frac{\ds \int_{\R^d} \frac{\nu_0 \vert u \vert^{\beta}
      (u \cdot \sigma)^2}{\nu_0^2
      + \vert u \vert^{2\beta} (u \cdot \sigma)^2}
      \frac{\dd u}{|u|^{d+\alpha}}}{\ds \int_{\R^d} \wv^{-d-\alpha}
      \dd v} \hspace{2.1cm}
      \boxed{\text{ when } \alpha \in
      (0,2+\beta)} \\[10mm]
      \ds
      \frac{(1+\beta)}{|\mathbb{S}^{d-1}|} 
      \frac{\ds \int_{\R^d} \frac{\nu_0 \vert u \vert^{\beta}
      (u \cdot \sigma)^2}{\nu_0^2
      + \vert u \vert^{2\beta} (u \cdot \sigma)^2}
      \frac{\dd u}{|u|^{d}}}{\ds \int_{\R^d}
      \wv^{-d-\beta} \dd v} \hspace{1.2cm}
      \boxed{\text{ when } \alpha =0}
    \end{cases}
  \end{equation*}
  as well as $\kappa := \| \cM \|^{-1} _{L^1(\R^d)} \int_{\R^d} (v\cdot
  \sigma)^2 \nu^{-1} \cM(v) \dd v$ for the case ``$\alpha = +\infty$''.
\end{corollary}

This recovers and unifies the results in~\cite{MR2588245,M3,MR2861578,MR1803225,Mel} (except for the
case of space-dependent collision kernels in~\cite{MR1803225})
and extends them to new cases such as $\alpha =0$
(infinite mass). The convergence rate is also new. Our approach
shares common points with, but differs from the probabilistic
method in~\cite{MR2588245}, the Hilbert expansions
in~\cite{MR2861578,MR1803225}, the moment method in~\cite{Mel}
and the Fourier-Laplace method in~\cite{M3}.

\begin{corollary}[Kinetic Fokker-Planck equation]
  \label{cor:FP}
  Assume that $\cL$ is the Fokker-Planck operator~\eqref{eq:FP}
  with $\cM$ satisfying Hypothesis~\ref{hyp:functional} with
  $\alpha \geq 0$. Then Theorem~\ref{theo:main} applies with
  $\alpha$ given in Hypothesis~\ref{hyp:functional} and
  $\beta = 2$. This proves the diffusive limit for solutions
  to~\eqref{eq:gvar} satisying~\eqref{eq:initial-data}-\eqref{eq:initial-layer-micro}-\eqref{eq:initial-layer-fluid} with quantitative rate, diffusion exponent
  $\zeta = \min \left( 2, \frac{\alpha+2}{3} \right)$, scaling
  function~\eqref{eq:scaling-function} and diffusion
  coefficient~\eqref{eq:coeff-full}. Apart from~\eqref{eq:initial-layer-fluid},
  the errors we obtain are polynomial in $\varepsilon$ for
  $\alpha \in (0,+\infty) \setminus \{4\}$ and logarithmic for
  $\alpha=4$. The diffusion coefficient can be made precise when
  $\alpha \in (0,4]$ using that $\Phi$ solves the
  Schrödinger-type equation
  \begin{equation*}
    -|u|^2 \Delta_u \Phi + (d+\alpha)u \cdot  \nabla_u \Phi
    - i (u \cdot \sigma) |u|^2 \Phi =0 \quad \text{ with the
      normalisation } \Phi(0)=1.
  \end{equation*}
  In particular when $\alpha=4$, the function $\Omega$ solves
  \begin{equation*}
    -|u|^2 \Delta_u \Omega + (d+\alpha)u \cdot  \nabla_u \Omega =
    (u \cdot \sigma) |u|^2 \quad \text{ with } \quad \Omega(0)=0 \quad
    \Longrightarrow \quad
    \Omega(u) := \frac{|u|^2 (u \cdot \sigma)}{d+8}.
  \end{equation*}
\end{corollary}

This recovers and unifies the fractional diffusive limit results
in~\cite{MR3984748,MR3342190,fournier2018dimensional,fournier2018anomalous,lebeau2017diffusion}. Novel
contributions include formulas for the diffusion coefficient in
dimension higher than $1$, and the quantitative argument
providing a convergence rate.
  
\begin{corollary}[Kinetic Lévy-Fokker-Planck equation]
  \label{cor:LFP}
  Assume that $\cL$ is the Lévy-Fokker-Planck
  operator~\eqref{eq:LFP} with parameter $s \in (\frac12,1)$ and
  with $\cM$ satisfying Hypothesis~\ref{hyp:functional} with
  $\alpha > s$.  Then Theorem~\ref{theo:main} applies with
  $\beta := 2s-\alpha$. This proves the diffusive limit for
  solutions to~\eqref{eq:gvar} satisfying~\eqref{eq:initial-data}-\eqref{eq:initial-layer-micro}-\eqref{eq:initial-layer-fluid} with quantitative rate and
  diffusion exponent
  \begin{equation*}
    \zeta = 
    \begin{cases}
      2 & \text{ when } \alpha \ge 1+s\\[2mm] \ds
      \frac{2s}{1+2s- \alpha} & \text{ when } \alpha \in
      (s,1+s),    
    \end{cases}
 \end{equation*}
 and scaling function~\eqref{eq:scaling-function} and diffusion
 coefficient~\eqref{eq:coeff-full}. Apart from~\eqref{eq:initial-layer-fluid}, the errors we obtain
 are polynomial in $\varepsilon$ when
 $\alpha \in (s,1+s) \cup (1+s,+\infty)$ and logarithmic for
 $\alpha=1+s$. Moreover
 \begin{equation*}
   \Phi(u) := \exp\left(i\frac{2s c_{\alpha,0}}{c_{\alpha,\beta}}
     \frac{|u|^\beta (u \cdot \sigma)}{1+\beta}\right)
 \end{equation*}
 when $\alpha \in (s,1+s]$ and
 \begin{equation*}
   \Omega(u) := \frac{2s c_{\alpha,0}} {c_{\alpha,\beta}}  \frac{|u|^\beta (u \cdot \sigma)}{1+\beta}
 \end{equation*}
 when $\alpha=1+s$, which yields for the diffusion coefficient
 \begin{equation*}
   \kappa :=
   \begin{cases}
     \ds \frac{2s c_{\alpha,0}^2} {c_{\alpha,\beta}(1+\beta)^2} \ds
     \int_{\mathbb{S}^{d-1}} (\sigma' \cdot \sigma)^2 \dd \sigma'
     \hspace{2.2cm}
     &\boxed{\text{ when } \alpha =
       1+s} \\[9mm]     
     \ds \frac{c_{\alpha,0}}{1+\beta}
     \left(\frac{2s c_{\alpha,0}}{c_{\alpha,\beta} (1+\beta)}
     \right)^{\frac{\alpha-1}{1+\beta}}
     \int_{\R^d} (w \cdot \sigma) \sin(w \cdot \sigma) \frac{dw}{
     |w|^{d + \frac{\alpha + \beta}{1+\beta}}}
     \hspace{0.5cm}
     &\boxed{\text{ when } \alpha \in
       (s,1+s)}
   \end{cases}
 \end{equation*}
\end{corollary}

This recovers and extends the qualitative results
in~\cite{MR3916950,CesbronMellet} to general equilibria, with
quantitative error estimates and formulas for the diffusion
coefficient. In~\cite{MR3916950,CesbronMellet}, the moment method initiated by Mellet is used to derive a fractional limit in the case
$\beta=0$. It raises several interesting questions: (1) can our
approach be extended to $s \in (0,\frac12)$? (this seems to be a
technical difficulty), (2) is the fractional diffusive limit
possible for infinite mass equilibria?  (i.e. $\alpha <0$), (3)
can the connexion between the kinetic Lévy-Fokker-Planck equation
with $\alpha=2s$ (for which the $\cL$ is the generator of a Lévy
process) and the standard kinetic Fokker-Planck equation with
Gaussian equilibrium be clarified as $s \to 1$?  (our diffusion
constant $\kappa$ above diverges as $s\to 1$ so the two limits in
$\varepsilon \to 0$ and $s \to 1$ do not commute which calls for
further investigation).

Let us summarise our contributions. Theorem~\ref{theo:main} and
Corollaries~\ref{cor:scat}--\ref{cor:FP}--\ref{cor:LFP} recover
the results of
\cite{MR3916950,MR2861578,MR3984748,MR1803225,fournier2018dimensional,fournier2018anomalous,lebeau2017diffusion,Mel,M3,MR3342190}
with a shorter and unified constructive method and prove new
results for (1) Lévy-Fokker-Planck operators, (2) scattering
operators with decaying collision kernel and infinite mass
equilibria and importantly (3) Fokker-Planck operators in any
dimension (for which formulas for the diffusion coefficient were
not known). The quantitative error in this fluid approximation
seems to also be novel for all equations considered. Note finally
that like the abstract Theorem~\ref{theo:main},
Corollaries~\ref{cor:scat}--\ref{cor:FP}--\ref{cor:LFP} are
stated with the exact equilibrium of Hypothesis
\ref{hyp:functional}, but can be extended to more general
equilibria, see Section~\ref{sec:appendix}. Moreover, it would be
interesting to try and apply this method in other settings such
as~\cite{BardosGolseMoyano,frisch1977non} (radiative transfer
theory), \cite{borgers1992diffusion,MR1788479,MR1632712}
(rarefied gas in a region between two parallel plates),
\cite{MR4152643,MR3861298,MR4064198} (domains with interface or
boundaries), \cite{MR3636614} (scattering with external
acceleration field), \cite{perthame_fractional_2018} (models for
chemotaxis) and~\cite{MR3370610,MR3422647,MR3465990} (adding a
local conservation of momentum).

The method of the present paper extends to the fractional
diffusive limit the approach pioneered
in~\cite{Nicolaenko,Ellis-Pinsky} of constructing exact
dispersion laws in the regime of parabolic time-space scaling and
small eigenvalues; this extension is inspired by the recent
one-dimensional result~\cite{lebeau2017diffusion} and in
particular we use and generalise the idea of rescaling velocities
to obtain a non-trivial dispersion law in~\cite{lebeau2017diffusion}. In comparison with~\cite{lebeau2017diffusion}, the main novelty of
the present paper is a quantitative spectral method for
constructing the branch of fluid eigenvalue:
in~\cite{lebeau2017diffusion} it was done by a one-dimensional
argument connecting two infinite series on $\R_-$ and $\R_+$ (and
it was done by fixed points in the simpler case of classical
diffusive limit in the older
works~\cite{Nicolaenko,Ellis-Pinsky}).

Let us now compare our paper with the previous recent works by
probabilists~\cite{fournier2018dimensional,fournier2018anomalous}. In
probabilistic terms, we try to describe particles moving in the
full $d$-dimensional space along ${\rm d} X_t = V_t \dd t$ with
velocities $V_t$ following a reversible process with invariant
measure of the form given in \eqref{hyp:functional}. The velocity
process is typically of scattering type, or Langevin type with
drift and Brownian or non-Gaussian L\'evy-type noises. We show in
Theorem~\ref{theo:main} that the rescaled process
$\varepsilon X_{\theta(\varepsilon)^{-1}t}$ converges, with
explicit rates and multiplicative constants, towards a Brownian
motion when $\alpha \geq 2 + \beta$, and towards a radially
symmetric $\zeta$-stable process when
$\alpha \in [0,2+\beta)$. In spite of using quite different
languages, the common point
between~\cite{fournier2018dimensional,fournier2018anomalous} and
the present paper is the use of a scaling in velocity, which
corresponds to applying some power function to the random
variable in the probabilistic viewpoint and corresponds to the
study of the rescaled fluid mode
$\Phi_\eta := \phi_\eta(\eta^{-\frac{1}{1+\beta}} \cdot )$ in our
study. Note finally that the eigenvalue problem we study to
compute the limit diffusion coefficient does not seem to have a
counterpart
in~\cite{fournier2018dimensional,fournier2018anomalous}.

The rest of the paper is structured as follows. Section
\ref{sec:duality} is devoted to the proof of
Theorem~\ref{theo:main} assuming
Lemmas~\ref{lem:existencespectral}, \ref{lem:ratespectral} and
\ref{lem:diffcoeff}. We then prove
Lemma~\ref{lem:existencespectral} (construction of the fluid
mode) in Section~\ref{sec:branch}, Lemma~\ref{lem:ratespectral}
(scaling of the fluid mode) in Section~\ref{sec:ratespectral},
and Lemma~\ref{lem:diffcoeff} (derivation of the diffusion
coefficient) in
Section~\ref{sec:diffcoeff}. Sections~\ref{sec:scatt}-\ref{sec:kfp}-\ref{sec:lkfp}
prove the abstract hypotheses on the three concrete models; one
argument of independent interest is a tightness estimate for the
Schrödinger-type equation satisfied by the rescaled fluid mode in
the cases of Fokker-Planck operators, see
Lemma~\ref{lem:gain-mmt-rescaled}. Finally,
Section~\ref{sec:appendix} briefly discusses extensions of our
results to more general equilibrium distributions and operators.

\subsection*{Acknowledgments}

The authors wish to thank Jean Dolbeault for an enlightening discussion on Hardy-Poincaré inequalities, Marc Briant for a careful reading and feedback leading to an improvement of the statement of the main abstract result, and Christian Schmeiser and Sara Merino-Aceituno for discussions in this initial phase of this work. Finally they thank Nicolas Fournier and Camille Tardif for crucial comments, helpful discussions and pointing out a mistake in a previous version.
    

\section{Proof of Theorem~\ref{theo:main} (convergence)}
\label{sec:duality}

In this section we assume the Lemmas~\ref{lem:existencespectral},
\ref{lem:ratespectral} and \ref{lem:diffcoeff} to hold,
$\alpha \ge 0$, and prove Theorem~\ref{theo:main}. Consider~\eqref{eq:kinetich} and the rescaling
\begin{align*}
  h_\varepsilon(t,x,v) := h
  \left(\frac{t}{\theta(\varepsilon)},\frac{x}{\varepsilon},v
  \right) =
  \frac{f_\varepsilon(t,x,v)}{\cM(v)} =
  \frac{f\left(\frac{t}{\theta(\varepsilon)},\frac{x}{\varepsilon},v
  \right)}{\cM(v)}. 
\end{align*}
It satisfies the equation
\begin{equation}
  \label{eq:hvar}
  \theta(\varepsilon) \partial_t h_\varepsilon +
  \varepsilon v \cdot \nabla_x h_\varepsilon = L h_\varepsilon.
\end{equation}

\subsection{The energy estimate}

Integrate~\eqref{eq:hvar} against $h_\varepsilon \cM$ in $t,x,v$,
and take the real part:
\begin{align*}
  \frac{\theta(\varepsilon)}{2} \| h_\varepsilon(t) \| ^2
  &  = \frac{\theta(\varepsilon)}{2} \| h_\varepsilon(0) \| ^2
    + \int_0 ^t \Re \, \big \langle L h_\varepsilon(\tau), 
    h_\varepsilon (\tau) \big\rangle \dd \tau \\
  & \le \frac{\theta(\varepsilon)}{2} \| h_\varepsilon(0) \| ^2
    - \lambda \int_0 ^t \left\| h_\varepsilon(\tau) -
    r_\varepsilon(\tau) \right\|_{-\beta} ^2 \dd \tau
\end{align*}
where we have used Hypothesis~\ref{hyp:coercivity} and 
\begin{equation*}
  r_\varepsilon(t,x) := \int_{\R^d} h_\varepsilon(t,x,v)
  \cM_\beta(v) \dd v.
\end{equation*}
This proves
\begin{align}
  \label{eq:bound-energy}
  \forall \, t \ge 0, \quad
  \| h_\varepsilon(t) \| ^2 \le \| h_\varepsilon(0) \| ^2  \
  \text{ and } \ \int_0 ^t \left\| h_\varepsilon(\tau) -
  r_\varepsilon(\tau) \right\|_{-\beta} ^2 \dd \tau \le
  \frac{\theta(\varepsilon)}{2 \lambda} \|
  h_\varepsilon(0) \| ^2.
\end{align}

\subsection{Framework of the calculations}

Denote $\xi$ to be the Fourier variable of $x$, and Fourier-transform~\eqref{eq:hvar} in $x$ to get on
$\hat h_\varepsilon(t,\xi,v)$
\begin{equation}
  \label{eq:g}
  \theta(\varepsilon) \partial_t \hat h_\varepsilon = L \hat
  h_\varepsilon - i \varepsilon (v \cdot \xi) \hat h_\varepsilon.
\end{equation}
Note that \eqref{eq:bound-energy} and the Plancherel theorem
imply $\hat h_\varepsilon \in L^\infty_t(\R^+;L^2_{\xi,v}(\cM))$ and
\begin{equation}\label{eq:bound-energy-fourier}
  \big\| \hat h_\varepsilon - \hat r_\varepsilon 
  \big\|_{L^2_t(\R^+;L^2_{\xi,v}(\cM_\beta))} \lesssim
  \theta(\varepsilon)^\frac12 \|
  h_\varepsilon(0) \|. 
\end{equation}

Denote $\xi =: |\xi| \sigma$ and $\eta := \varepsilon
|\xi|$. Test~\eqref{eq:g} against $\cM \phi_{\eta}$ with
$\phi_\eta$ constructed in Lemma~\ref{lem:existencespectral}:
\begin{align}
  \label{eq:hphi-energy}
  \theta(\varepsilon)\dt \left\langle \hat h_\varepsilon,
  \phi_{\eta} \right\rangle
  & = \left\langle L \hat h_\varepsilon - i \varepsilon (v \cdot
    \xi) \hat h_\varepsilon , \phi_{\eta} \right\rangle
   = \left\langle \hat h_\varepsilon , L^*\left(
    \phi_{\eta} \right) + i \varepsilon (v \cdot \xi) \phi_{\eta} 
    \right\rangle \\
  \nonumber
  & = - \mu(\eta) \left\langle \hat h_\varepsilon , \wv^{-\beta}
    \phi_{\eta} \right\rangle.
\end{align}
We then split the integrals as follows:
\begin{align*}
  \left\langle \hat h_\varepsilon, \phi_{\eta} \right\rangle
  & = \hat r_\varepsilon \left\langle 1, 
    \phi_{\eta} \right\rangle  + \left\langle \hat h_\varepsilon
    - \hat r_\varepsilon , 
    \phi_{\eta} \right\rangle =: \left\langle 1, \phi_{\eta}
    \right\rangle
    \left[ \hat r_\varepsilon - E_1 \right] \\[3mm]
  \left\langle \hat h_\varepsilon , \wv^{-\beta} \phi_{\eta} 
  \right\rangle
  &= \hat r_\varepsilon \left\langle 1 , \wv^{-\beta} \phi_{\eta}  
    \right\rangle + \left\langle \hat h_\varepsilon - \hat r_\varepsilon,
    \wv^{-\beta} \phi_{\eta} \right\rangle =:
    \left\langle 1, \phi_{\eta} \right\rangle
    \frac{\theta(\varepsilon)}{\mu(\eta)}
    \left[  \kappa_\eta \hat r_\varepsilon - E_2\right]
\end{align*}
with the definitions (using the normalisation
$\langle 1,\wv^{-\beta} \phi_{\eta} \rangle =1$)
\begin{align*}
  & \kappa_\eta := \frac{\mu(\eta)\left\langle 1,
    \wv^{-\beta} \phi_{\eta}  
    \right\rangle}{\theta(\varepsilon)\left\langle 1,  \phi_{\eta}
    \right\rangle} =
    \frac{\mu(\eta)}{\theta(\varepsilon)\left\langle 1,  \phi_{\eta}
    \right\rangle},  \\
  & E_1 := -\frac{\left\langle \hat h_\varepsilon - \hat r_\varepsilon, 
    \phi_{\eta} \right\rangle}{\left\langle 1, 
    \phi_{\eta} \right\rangle} \quad \text{ and } \quad 
   E_2 := -\frac{\mu(\eta)\left\langle \hat h_\varepsilon
    - \hat r_\varepsilon, \wv^{-\beta} \phi_{\eta} 
    \right\rangle}
    {\theta(\varepsilon)\left\langle 1, \phi_{\eta} \right\rangle }.
\end{align*} 
Consequently, \eqref{eq:hphi-energy} gives
\begin{equation*}
  \partial_t \hat r_\varepsilon + \kappa_\eta \hat r_\varepsilon  =
  \partial_t E_1 + E_2.
\end{equation*}
We then want to pass to the limit $\varepsilon \to 0$ (hence
$\eta \to 0$ for each frequency $\xi$).

\subsection{Estimating $\kappa_\eta$, $\partial_t E_1$ and
  $E_2$.}

Lemma~\ref{lem:existencespectral}-\ref{lem:ratespectral} yield
\begin{equation*}
  \lim_{\varepsilon \to 0}  \frac{\mu(\eta)}{ \theta(\varepsilon)} =
  \mu_0 \vert \xi\vert^{\zeta} \quad \text{ with } \quad \zeta :=
  \frac{\alpha + \beta}{1+\beta}, 
\end{equation*}
with constructive rate, for each frequency $\xi \in \R^d$ (note
that in the cases $\alpha=0$ or $\alpha = 2+\beta$, the error in
the convergence includes a loss of frequency weight
$|\ln |\xi||$). Lemma~\ref{lem:diffcoeff} implies
\begin{equation}
  \lim_{\eta \to 0} \; \kappa_\eta= \kappa |\xi|^\zeta =
  \mu_0 |\xi|^\zeta \times 
  \begin{cases}
    \ds \| \cM \|_{L^1(\R^d)}^{-1}
    & \text{ when } \alpha >0,
    \\[4mm] \ds
    \frac{1+\beta}{|\mathbb{S}^{d-1}|}
    & \text{ when } \alpha = 0,
  \end{cases}
\end{equation}
with constructive convergence rate and 
\begin{equation*}
\frac{\mu(\eta)}{\theta(\varepsilon)
    \vert \left\langle 1,\phi_{\eta} \right\rangle \vert}
    \lesssim |\xi|^{\zeta}.
\end{equation*} 

To estimate $E_2$, write
\begin{align*}
  \left| \left\langle \hat h_\varepsilon  -  \hat r_\varepsilon,
  \wdot^{-\beta} \phi_{\eta} \right\rangle \right|
  \lesssim \left\| \hat h_\varepsilon  -  \hat r_\varepsilon
  \right\|_{-\beta}
\end{align*}
where we have used $\|\phi_\eta\|_{-\beta} =1$. All in all, we
get, using again Lemmas~\ref{lem:diffcoeff}
and~\ref{lem:zeromoment},
\begin{equation}
  \label{eq:bound-E2}
  \left\vert E_2 \right\vert \lesssim
  \frac{\mu(\eta)}{\theta(\varepsilon)|\left\langle 1, 
      \phi_{\eta} \right\rangle|}
  \left\| \hat h_\varepsilon  -  \hat r_\varepsilon
  \right\|_{-\beta} \lesssim \left\| \hat h_\varepsilon  -  \hat
    r_\varepsilon \right\|_{-\beta} |\xi|^{\zeta}.
\end{equation}
To estimate $E_1$, compute first
\begin{equation*}
  \left| \left\langle \hat h_\varepsilon - \hat r_\varepsilon , \phi_{\eta} 
    \right\rangle \right| \leq \left\| \hat h_\varepsilon - \hat
    r_\varepsilon \right\|_{-\beta} \left\| \phi_{\eta} \right\|_{\beta},
\end{equation*}
to get 
\begin{equation*}
  \left\vert E_1 \right\vert \lesssim
  \frac{\left\| \phi_{\eta} \right\|_{\beta}}{
    \left|\left\langle 1,\phi_{\eta} \right\rangle \right|}
  \left\| \hat h_\varepsilon - \hat r_\varepsilon 
  \right\|_{-\beta}.
\end{equation*}
One then estimates $\left\| \phi_\eta \right\|_\beta$. When
$\alpha > \beta$, it is bounded by construction, and when
$\alpha \leq \beta$,
\begin{equation*}
 \left\| \phi_\eta \right\|_\beta^2 =
  \eta^{\frac{\alpha-\beta}{1+\beta}}
  \int_{\R^d} \left|\Phi_\eta(u)\right|^2
  |u|_\eta^{-d-\alpha+\beta} \dd u.
\end{equation*}
Using the pointwise bound \eqref{eq:mmt-fract} and the moment
bound \eqref{eq:mmt-fract-bis} from Hypothesis
\ref{hyp:scalinginfinite}, the latter integral exists and is
uniformly bounded in $\eta$ for $\alpha \in [0,\beta)$ and is
bounded by $|\ln \eta|$ when $\alpha = \beta$. Thus we get, using
Lemma \ref{lem:zeromoment} to estimate
$\langle 1, \phi_\eta \rangle$ again,
\begin{equation}
  \label{eq:bound-E1}
  \left\vert E_1 \right\vert \lesssim
  \left\| \hat h_\varepsilon -
    \hat r_\varepsilon  \right\|_{-\beta} \times
  \begin{cases}
    \ds 1,
    & \text{ when } \alpha > \beta,
    \\[3mm]
    \left|
      \ln \left( \varepsilon |\xi|
      \right) \right|^{\frac12}
    & \text{ when } \alpha = \beta,
    \\[3mm] \ds
    (\varepsilon 
    |\xi|)^{\frac{\alpha-\beta}{2(1+\beta)}} 
    & \text{ when } \alpha \in (0,\beta),
    \\[3mm] \ds
    | \ln(\varepsilon |\xi|)|^{-1}
    \left( \varepsilon |\xi|
    \right)^{-\frac{\beta}{2(1+\beta)}} 
    & \text{ when } \alpha = 0.
  \end{cases}
\end{equation}

Set $r:=r(t,x)$ to be solution to $\partial_t r + \kappa |\xi|^\zeta r =0$ with initial data $r(0,\cdot)$ defined in~\eqref{eq:initial-layer-fluid} and deduce that $\omega_\varepsilon := \hat r_\varepsilon - \hat r$ satisfies
\begin{equation*}
  \partial_t \omega_\varepsilon + \kappa |\xi|^\zeta
  \omega_\varepsilon  = \partial_t E_1 + E_2 + \left( \kappa -
    \kappa_\eta \right) \hat r_\varepsilon.
\end{equation*}

We assume
that~\eqref{eq:initial-data}-\eqref{eq:initial-layer-micro}
hold. By Duhamel's formula,
\begin{align*}
  \omega_\varepsilon(t,\xi)
  & =  \omega_\varepsilon(0,\xi) e^{-\kappa |\xi|^\zeta
    t} + \int_0 ^t e^{-\kappa |\xi|^\zeta
    (t-s)} \left[
    \partial_t E_1(s,\xi)+ E_2(s,\xi) + \left( \kappa -
    \kappa_\eta \right) \hat r_\varepsilon (s,\xi) \right] \dd s \\
  & = \omega_\varepsilon(0,\xi) e^{-\kappa |\xi|^\zeta
    t} + E_1(t,\xi) - e^{-\kappa |\xi|^\zeta
    t} E_1(0,\xi) \\
  & \qquad + \int_0 ^t e^{-\kappa |\xi|^\zeta (t-s)} \left[\kappa
    |\xi|^\zeta E_1(s,\xi)+ E_2(s,\xi) + \left( \kappa -
    \kappa_\eta \right) \hat r_\varepsilon (s,\xi) \right] \dd s.
\end{align*}
Define the weight
\begin{equation*}
  W(\xi) :=  \lfloor \xi \rceil^{-\zeta} \times
  \begin{cases}
    1 \quad \text{ when } \quad \alpha
    > \beta, \\[3mm]
    \left| \ln \frac{2|\xi|}{1+|\xi|} \right|^{-1}
    \quad
    \text{ when } \quad \alpha = \beta, \\[3mm]
    |\xi|^{\frac{\beta-|\alpha|}{2(1+\beta)}}
    \lfloor \xi \rceil^{-\frac{\beta-|\alpha|}{2(1+\beta)}}
    \quad
    \text{ when } \quad \alpha \in [0,\beta).
  \end{cases}
\end{equation*}
Then we integrate in $L^2_\xi(W)$ and then in $L^2_t([0,T])$: the estimates~\eqref{eq:bound-energy}-\eqref{eq:bound-E2}-\eqref{eq:bound-E1} imply that the errors $E_1$ and $E_2$ go to zero in $L^2_t([0,T];L^2_\xi(W))$, and~\eqref{eq:initial-layer-fluid} ensures that $\omega_\varepsilon(0,\xi)$ goes to zero in $L^2_\xi(W)$, which concludes the proof.

\section{Proof of Lemma~\ref{lem:existencespectral} (construction
  of the fluid mode)}
\label{sec:branch}

In this section we prove Lemma~\ref{lem:existencespectral}, assuming Hypotheses~\ref{hyp:functional}--\ref{hyp:coercivity}--\ref{hyp:large-v} with $\alpha+\beta>0$.  Denote
\begin{equation*}
  \tilde{L}_{\eta} ^* \psi := \wv^{\frac{\beta}{2}} L_\eta^*
  \left( \wdot^{\frac{\beta}{2}} \psi \right) = 
  \wv^{\frac{\beta}{2}} L^* \left( \wdot^{\frac{\beta}{2}}\psi \right)  
  + i \eta \wv^\beta (v \cdot \sigma) \psi.
\end{equation*}
As before, the dependency in $\sigma$ is omitted from the
subscripts for readability.

\subsection{Existence of the resolvent}
\label{subsec:resolvent}

\subsubsection{Near zero}
\label{sss:near}

We first prove that when $z \in B(0,\textsc{r}_0 \Theta(\eta))$
with $\textsc{r}_0$ small enough and $\eta \in (0,\eta_0)$ with
$\eta_0$ small enough, the operator $\tilde{L}_{\eta} ^* - z$ has
a bounded inverse in $L^2_v(\cM)$. Given
$G \in L^2(\wdot^{-\beta} \cM)$ and
$z \in B(0,\textsc{r}_0 \Theta(\eta))$, consider an \textit{a
  priori} solution $F \in L^2(\wdot^{-\beta}\cM)$ to
\begin{align}
  \label{eq:FG-0}
  - L^* F - i \eta (v \cdot \sigma) F - z \wv^{-\beta} F =
  \wv^{-\beta} G.
\end{align}
Recall the decomposition
\begin{equation}
  \label{eq:decomp-bot-0}
  F = \cP F + \cP^\bot F \quad \text{ with } \quad
  \cP F := \int_{\R^d} F(v) \cM_\beta(v) \dd v,
\end{equation}
which is orthogonal for the scalar product associated with
$\| \cdot \|_{-\beta}$. Integrate~\eqref{eq:FG-0} against
$\bar F \cM$ and take the real part to get, using
Hypothesis~\ref{hyp:coercivity} and denoting $r:=|z|$,
\begin{align*}
  & \lambda \left\| \cP^\bot F \right\|^2 _{-\beta} - r \| F
    \|_{-\beta}^2 \le \| G \|_{-\beta} \| F \|_{-\beta} \\
  \Longrightarrow \qquad 
  & \left( \lambda -r \right)
    \left\| \cP^\bot F \right\|^2 _{-\beta}
    \le \| G \|_{-\beta} \| F \|_{-\beta} + r \left| \cP F
    \right|^2 \\
  \Longrightarrow \qquad 
  & \left( \lambda -r \right)
    \left\| \cP^\bot F \right\|^2 _{-\beta}
    \le \frac{\textsc{r}_0 \Theta(\eta)}{2} \| F \|^2 _{-\beta} +
    \frac{1}{2\textsc{r}_0 \Theta(\eta)} \| G \|^2_{-\beta}
    + r \left| \cP F \right|^2 \\
  \Longrightarrow \qquad 
  & \left( \lambda -r \right)
    \left\| \cP^\bot F \right\|^2 _{-\beta}
    \le \frac{\textsc{r}_0 \Theta(\eta)}{2} \| \cP^\bot F \|^2 _{-\beta} +
    \frac{1}{2\textsc{r}_0 \Theta(\eta)} \| G \|^2_{-\beta} +  \frac{3\textsc{r}_0 \Theta(\eta)}{2} 
    \left| \cP F \right|^2 \\
  \Longrightarrow \qquad 
  & \left(  \frac{\lambda}{2} -r \right)
    \left\| \cP^\bot F \right\|^2 _{-\beta}
    \le 
    \frac{1}{2\textsc{r}_0 \Theta(\eta)} \| G \|^2_{-\beta} +\frac{3\textsc{r}_0 \Theta(\eta)}{2} 
    \left| \cP F \right|^2 
\end{align*}
which implies finally for $r$ small enough (say for instance $r < \frac{\lambda}{4}$),
\begin{equation}
  \label{eq:FG-L2-0}
\left\| \cP^\bot F \right\| _{-\beta}
    \lesssim \textsc{r}_0^{\frac12} \Theta(\eta)^{\frac12} \left| \cP F
    \right| + \textsc{r}_0^{-\frac12} \Theta(\eta)^{-\frac12}\| G \|_{-\beta}.
\end{equation}
Consider then a smooth function $0 \le \mathfrak{K} \le 1$,
radially symmetric, and such that $\mathfrak{K} \equiv 1$ on
$B(0,3) \setminus B(0,2)$ and $\mathfrak{K} \equiv 0$ on $B(0,1)$
and outside $B(0,4)$, and denote
$\mathfrak{K}_R(v) := \mathfrak{K}(\frac{v}{R})$ for
$R>0$. Denote
$\tilde{\mathfrak{K}}_R := (v \cdot \sigma) \wv^\beta
\mathfrak{K}_R$ and integrate~\eqref{eq:FG-0} against
$\tilde{\mathfrak{K}}_R \cM$:
\begin{align}
  \nonumber
  - \left\langle L^* F, \tilde{\mathfrak{K}}_R  \right\rangle 
  - i \eta \int_{\R^d} (v \cdot \sigma)^2 F(v)
  \mathfrak{K}_R (v) \wv^\beta \cM(v) \dd v - z \int_{\R^d} F(v)
  \tilde{\mathfrak{K}}_R(v) \wv^{-\beta} \cM(v) \dd v \\
   \label{eq:trunc1-0}
  = \int_{\R^d} G(v) \tilde{\mathfrak{K}}_R(v) \wv^{-\beta}
  \cM(v) \dd v.
\end{align}
Using the decomposition~\eqref{eq:decomp-bot-0}, $L^* 1=0$ and Hypothesis~\ref{hyp:large-v}, we have, for $\alpha \in (-\beta,2+\beta)$,
\begin{equation}
  \label{eq:trunc2-0}
  \left| \left\langle L^* F, \tilde{\mathfrak{K}}_R \right\rangle \right| =
  \left| \left\langle L^* \left( \cP ^\bot F \right),
      \tilde{\mathfrak{K}}_R \right\rangle \right| 
  \le \left\Vert L \left( \tilde{\mathfrak{K}}_R \right)
  \right\Vert_\beta \left\Vert  \cP ^\bot F  \right\Vert_{-\beta}
  \lesssim  R^{1+\beta-\frac{\alpha+\beta}{2}} \left\Vert  \cP ^\bot F
  \right\Vert_{-\beta}. 
\end{equation}
Observe also that
\begin{align}
  \nonumber
  & \left| \int_{\R^d} (v \cdot \sigma)^2
    F (v) \mathfrak{K}_R (v) \wv^\beta \cM(v) \dd v \right| \\
  \nonumber
  & \gtrsim \left|\cP F\right| \left|
    \int_{\R^d} (v \cdot \sigma)^2 \mathfrak{K}_R(v) \wv^\beta
    \cM(v) \dd v \right| - \left| \int_{\R^d} (v \cdot \sigma)^2
    \cP^\bot F (v) \mathfrak{K}_R (v) \wv^\beta \cM(v) \dd v \right|\\
  \nonumber
  & \gtrsim R^{2+\beta-\alpha} \left|\cP F\right|
    -  \left( \int_{\R^d} (v
    \cdot \sigma)^4  \mathfrak{K}_R (v)^2  \wv^{3\beta} \cM(v) \dd v
    \right)^{\frac12} \left\| \cP^\bot F \right\|_{-\beta} \\
  \label{eq:trunc3-0}
  & \geq R^{2+\beta-\alpha} \left|\cP F\right|
    - R^{2+\frac{3\beta}{2}-\frac{\alpha}{2}}
    \left\| \cP^\bot F \right\|_{-\beta}.
\end{align}
Then, we have
\begin{align}
  \nonumber
  \left| \int_{\R^d} F(v) \tilde{\mathfrak{K}}_R(v)
  \wv^{-\beta} \cM(v) \dd v \right|
  & = \left| \int_{\R^d} (v \cdot \sigma)^2 \cP^\bot
    F(v) \mathfrak{K}_R(v) \cM(v) \dd v \right| \\
  \nonumber
  & \lesssim \left\| \cP^\bot F \right\|_{-\beta} \left(
    \int_{\R^d}  (v \cdot \sigma)^2 \mathfrak{K}_R(v)^2
    \wv^{\beta} \cM(v) \dd v \right)^{\frac12} \\
  \label{eq:trunc4-0}
  & \lesssim R^{1+\beta-\frac{\alpha+\beta}{2}} \left\| \cP^\bot F
    \right\|_{-\beta}.
\end{align}
Finally, we have
\begin{align}
  \nonumber
  \left| \int_{\R^d} G(v) \tilde{\mathfrak{K}}_R(v) \wv^{-\beta}
  \cM(v) \dd v \right|
  &\lesssim \left\| G
    \right\|_{-\beta} \left(
    \int_{\R^d}  (v \cdot \sigma)^2 \mathfrak{K}_R(v)^2
    \wv^{\beta} \cM(v) \dd v \right)^{\frac12}\\
   \label{eq:trunc4-0}
  & \lesssim R^{1+\beta-\frac{\alpha+\beta}{2}} \left\| G
    \right\|_{-\beta}.
\end{align}

Combining~\eqref{eq:trunc1-0}--\eqref{eq:trunc2-0}--\eqref{eq:trunc3-0}--\eqref{eq:trunc4-0} yields
\begin{multline*}
\eta R^{2+\beta-\alpha} \left|\cP F\right| \leq \eta R^{2+\frac{3\beta}{2}-\frac{\alpha}{2}}
    \left\| \cP^\bot F \right\|_{-\beta} + R^{1+\beta-\frac{\alpha+\beta}{2}} \left\| G
    \right\|_{-\beta} \\ + r R^{1+\beta-\frac{\alpha+\beta}{2}} \left\| \cP^\bot F
    \right\|_{-\beta} + R^{1+\beta-\frac{\alpha+\beta}{2}} \left\Vert  \cP ^\bot F
  \right\Vert_{-\beta}.
\end{multline*}
Observe that the last but one term is negligible in front of the last one when $r$ is small, giving
\begin{equation*}
\left|\cP F\right| \lesssim \left( R^{\frac{\alpha + \beta}{2}}  + \eta^{-1} R^{-1+\frac{\alpha-\beta}{2}} \right) 
    \left\| \cP^\bot F \right\|_{-\beta} + \eta^{-1}R^{-1 + \frac{\alpha-\beta}{2}} \left\| G
    \right\|_{-\beta}.
\end{equation*}
Taking $R = \eta^{-\frac{1}{1+\beta}}$, we have
\begin{equation}
\left|\cP F\right| \lesssim \Theta(\eta)^{-\frac12} \left(  \left\| \cP^\bot F \right\|_{-\beta} + \left\| G   \right\|_{-\beta} \right).
\end{equation}
Combining the latter with \eqref{eq:FG-L2-0}, and for
$r \le \textsc{r}_0 \Theta(\eta)$ with $\Theta(\eta)$ defined
in~\eqref{eq:def-Theta}) and $\textsc{r}_0>0$ small enough, we have,
\begin{align}
  \label{eq:apriori-near}
  \left\vert \cP F \right\vert
  \lesssim \Theta(\eta)^{-1} \left\| G
  \right\|_{-\beta} \quad \text{ and } \quad
  \left\| \cP^\bot F \right\|_{-\beta} 
  \lesssim \Theta(\eta)^{-\frac12}\left\| G
    \right\|_{-\beta}.
\end{align}

When $\alpha=2+\beta$, the calculation is slightly modified as follows: we recall $\chi_R$ and $\tilde{\chi}_R := (v \cdot \sigma) \wv^\beta \chi_R$ defined in Hypothesis~\ref{hyp:large-v}, and integrate~\eqref{eq:FG-0} against $\tilde{\chi}_R \cM$ to get
\begin{align*}
  & \left| \left\langle L^* F, \tilde{\chi}_R \right\rangle \right|
    \lesssim (\ln R)^{\frac12} \left\Vert  \cP ^\bot F
    \right\Vert_{-\beta} \\
  & \left| \int_{\R^d} (v \cdot \sigma)^2
    F (v) \chi_R (v) \wv^\beta \cM(v) \dd v \right| 
    \geq (\ln R) \left|\cP[F]\right|
    - R^{1+\beta} \left\| \cP^\bot F \right\|_{-\beta}, \\
  & \left| \int_{\R^d} F(v) \tilde{\chi}_R(v)
  \wv^{-\beta} \cM(v) \dd v \right|
    \lesssim (\ln R)^{\frac12} \left\| \cP^\bot F \right\|_{-\beta}, \\
  & \left| \int_{\R^d} G(v) \tilde{\chi}_R(v)
  \wv^{-\beta} \cM(v) \dd v \right|
    \lesssim (\ln R)^{\frac12} \left\| G \right\|_{-\beta},
\end{align*}
which results in the same estimate~\eqref{eq:apriori-near} (with the choices $R = \eta^{-\frac{1}{1+\beta}}$, $r \le \textsc{r}_0 \Theta(\eta)$ and $\Theta(\eta) = \eta^2 |\ln \eta|$ with $\textsc{r}_0$ small enough).

Given $z \in B(0,\textsc{r}_0 \Theta(\eta))$ with $\textsc{r}_0$
small enough and $\eta \in (0,\eta_0)$ with $\eta_0$ small
enough, we deduce from~\eqref{eq:apriori-near} on the {\it a priori} inverse of $\tilde L^*_\eta -z$ the existence
of a solution to~\eqref{eq:FG-0} with the uniform estimate
$\| F \|_{-\beta} \lesssim \Theta(\eta)^{-1} \| G
\|_{-\beta}$. The equation~\eqref{eq:FG-0} re-writes
\begin{align}
  \label{eq:FGtilde-0}
  - \tilde L^* \tilde F - i \eta (v \cdot \sigma) \wv^\beta
  \tilde F - z \tilde F = \wv^{-\frac{\beta}{2}} G  \in
  L^2_v(\cM), 
\end{align}
with $\tilde F := \wdot^{-\frac{\beta}{2}} F$. Since
(by Hypothesis~\ref{hyp:coercivity}) $\tilde L^*$ generates a
contraction semigroup in $L^2_v(\cM)$, it is a standard result
(see~\cite[Theorem~II.3.15]{MR1721989}) that $\tilde L^*$ is
maximal dissipative. Therefore, given any $M\ge 1$, the operator
$\tilde L^* _{\eta,M} := \tilde L^* + i \eta (v \cdot \sigma)
\wv^\beta \chi_M(v)$ is maximal dissipative (perturbation by a
bounded purely imaginary multiplicative operator). Observe that
the previous a priori estimate~\eqref{eq:apriori-near} holds for
$\tilde L^*_{\eta,M}$ by the same calculation, and uniformly as
$M \to +\infty$. This implies that for each $M \ge 1$ and
$z \in S(0,r)$, there is $\tilde F_M \in L^2_v(\cM)$ that solves
$- \tilde L^*_M \tilde F_M - z \tilde F_M =
\wv^{-\frac{\beta}{2}} G$, and that $\tilde F_M$ is uniformly
bounded in $L^2_v(\cM)$ as $M \to \infty$. Taking a subsequence
weakly converging to some $\tilde F \in L^2_v(\cM)$ as
$M \to \infty$ gives a solution to~\eqref{eq:FGtilde-0} and thus
to~\eqref{eq:FG-0}.

\subsubsection{Away from zero}

We now prove that when $z \in \mathbb C$ with
$|z| \in (\textsc{r}_1 \Theta(\eta),r_0)$ with $\textsc{r}_1$
\emph{large} enough and $r_0$ \emph{small} enough and $\eta$
\emph{small} enough, the operator $\tilde{L}_{\eta} ^* - z$ has a
bounded inverse in $L^2_v(\cM)$, and the bound is uniform in
$|z| \in (\textsc{r}_1 \Theta(\eta),r_0)$. Given
$G \in L^2(\wdot^{-\beta} \cM)$ and $z \in S(0,r)$, consider an
\textit{a priori} solution $F \in L^2(\wdot^{-\beta}\cM)$
to~\eqref{eq:FG-0}. Integrating~\eqref{eq:FG-0} against
$\bar F \cM$, taking the real part and using
Hypothesis~\ref{hyp:coercivity} yields~\eqref{eq:FG-L2-0} again.
Consider then $0 \le \chi \le 1$ smooth radially symmetric and
such that $\chi \equiv 1$ on $B(0,1)$ and $\chi \equiv 0$ outside
$B(0,2)$, and denote $\chi_R(v) := \chi(\frac{v}{R})$ for $R
>0$. Integrate~\eqref{eq:FG-0} against $\chi_R \cM$:
\begin{align}
  \label{eq:trunc1}
  - \left\langle L^* F, \chi_R  \right\rangle 
  - i \eta \int_{\R^d} (v \cdot \sigma) F(v)
  \chi_R (v) \cM(v) \dd v - z \cP_R F = \cP_R G
\end{align}
where we denote the truncated average
\begin{align*}
  \cP_R F := \int_{\R^d} F(v) \chi_R(v) \cM_\beta(v) \dd v.
\end{align*}
Using the decomposition~\eqref{eq:decomp-bot-0}, $L^* 1=0$ and
Hypothesis~\ref{hyp:large-v}: 
\begin{equation}
  \label{eq:trunc2}
  \left| \left\langle L^* F, \chi_R \right\rangle \right| =
  \left| \left\langle L^* \left( \cP ^\bot F \right),
      \chi_R \right\rangle \right| 
  \le \left\Vert L \left( \chi_R \right) \right\Vert_\beta
  \left\Vert  \cP ^\bot F  \right\Vert_{-\beta}
  \lesssim
  R^{-\frac{\alpha+\beta}{2}} \left\Vert  \cP ^\bot F
  \right\Vert_{-\beta}. 
\end{equation}
Observe also that
\begin{align}
  \nonumber
  \left| \int_{\R^d} (v \cdot \sigma)
  F (v) \chi_R (v) \cM(v) \dd v \right|
  & = \left| \int_{\R^d} (v \cdot \sigma)
    \left[ \cP ^\bot F (v) \right] \chi_R (v) \cM(v) \dd v
    \right| \\
  \label{eq:trunc3}
  & \leq \left( \int_{|v| \le 2R} (v \cdot \sigma)^2 \wv^\beta \cM(v)
    \, dv \right)^{\frac12} \left\| \cP ^\bot F \right\|_{-\beta} 
    \lesssim \ell(R)\left\| \cP ^\bot F \right\|_{-\beta}
\end{align}
with
\begin{equation}
  \label{eq:lR}
  \ell(R) = \left( \int_{|v| \le 2R} (v \cdot \sigma)^2 \wv^\beta
    \cM (v) \dd v \right)^{\frac12} \lesssim
  \begin{cases}
    1  &\text{when } \alpha > 2 + \beta,\\[2mm]
    \sqrt{\ln(R)}  &\text{when } \alpha = 2 + \beta,\\[2mm]
    R^{1-\frac{\alpha-\beta}{2}} &\text{when } \alpha < 2 +
    \beta.
  \end{cases}
\end{equation}
Combining~\eqref{eq:trunc1}--\eqref{eq:trunc2}--\eqref{eq:trunc3}
yields the following estimate on the truncated average:
\begin{align}
  \label{eq:estim1-FG}
  \left\vert \cP_R F \right\vert \leq \frac{1}{r}
  \left[ \eta \ell(R) +
  R^{-\frac{\alpha+\beta}{2}} \right]
  \left\|  \cP ^\bot F  \right\|_{-\beta}
  + \frac1r \left\| G \right\|_{-\beta}. 
\end{align}
We next estimate the difference between $\cP F$ and $\cP_R F$:
\begin{align*}
  \nonumber
  \left| \cP F - \cP_R F \right|
  & \le \int_{\R^d} |F| \left| 1 - \chi_R
    \right| \cM_\beta(v) \dd v \\
  & \le  \left(
    \int_{\R^d} \left| 1 - \chi_R
    \right|^2 \cM_\beta  \dd v \right)^{\frac12} \| F \|_{-\beta}
    \lesssim R^{-\frac{\alpha+\beta}{2}} \| F \|_{-\beta},
\end{align*}
which implies for $R$ large enough
\begin{equation}
  \label{eq:diffmR}
  |\cP_R F| \lesssim |\cP_R F|  +
  R^{-\frac{\alpha+\beta}{2}}\left\| \cP^\bot F
  \right\|_{-\beta}.
\end{equation}  
Combining~\eqref{eq:estim1-FG} and~\eqref{eq:diffmR}, we deduce
\begin{align}
  \nonumber
  |\cP F|
  &\lesssim \left( \frac{\eta \ell(R) +
    R^{-\frac{\alpha+\beta}{2}}}{r} \right)
    \left\|  \cP ^\bot F  \right\|_{-\beta}
    + R^{-\frac{\alpha+\beta}{2}}\left\|  \cP ^\bot F
    \right\|_{-\beta} + \frac{1}{r}
    \left\| G \right\|_{-\beta},\\
  \label{eq:FG-average}
  &\lesssim  \left( \frac{\eta \ell(R) +
    R^{-\frac{\alpha+\beta}{2}}}{r} \right)
    \left\|  \cP ^\bot F  \right\|_{-\beta} +
    \frac{1}{r} \left\| G \right\|_{-\beta}.
\end{align}
Optimising $R$ so that the two terms in the parenthesis are equal
yields again $R = \eta^{-\frac{1}{1+\beta}}$ (with $\eta$ small
enough so that $R$ is large enough to obtain~\eqref{eq:diffmR})
and therefore
$\eta \ell(R) + R^{-\frac{\alpha+\beta}{2}} \sim
\Theta(\eta)^{\frac12}$ where $\Theta$ was defined
in~\eqref{eq:def-Theta}. Combining~\eqref{eq:FG-L2-0} and
\eqref{eq:FG-average} we get then
\begin{align*}
  \left| \cP F \right| \lesssim
  \frac{\Theta(\eta)^{\frac12}}{r^{\frac12}} \left| \cP F \right|
  + \left( 1 + \frac{\Theta(\eta)^{\frac12}}{r^{\frac12}}
  \right) \frac{1}{r} \left\| G \right\|_{-\beta}. 
\end{align*}
When $r > \textsc{r}_1 \Theta(\eta)$ with $\textsc{r}_1$
\emph{large} enough, $\frac{\Theta(\eta)^{\frac12}}{r^{\frac12}}$
is small and we deduce
\begin{equation}
  \label{eq:apriori-again}
  \left| \cP F \right| \lesssim \frac{1}{r} \left\| G
  \right\|_{-\beta} \quad \text{ and } \quad \left\| \cP^\bot F
  \right\|_{-\beta} \lesssim r^{-\frac12} \left\| G
  \right\|_{-\beta}.
\end{equation}
Arguing as before, given
$|z| \in (\textsc{r}_1 \Theta(\eta),r_0)$ with $\textsc{r}_1$
\emph{large} enough and $r_0$ \emph{small} enough and $\eta$
\emph{small} enough, we deduce from~\eqref{eq:apriori-again} the construction of a solution $F$ to~\eqref{eq:FG-0}
with the uniform estimate
$\| F \|_{-\beta} \lesssim \Theta(\eta)^{-1} \| G
\|_{-\beta}$. Together with Subsubsection~\ref{sss:near}, we have
thus proved that $\tilde L^*$ has no eigenvalues in
$|z| < \textsc{r}_0\Theta(\eta)$ and
$\textsc{r}_1 \Theta(\eta) < |z| < r_0$. We now prove the
existence of a unique eigenvalue in
$|z| \in (\textsc{r}_0 \Theta(\eta),\textsc{r}_1 \Theta(\eta) )$.

\subsection{The spectral projections}

We define the spectral projections
\begin{equation*}
  \mathsf\Pi_{r,\eta} := \frac{1}{2i\pi} \int_{S(0,r)} \left[
    \tilde{L}_{\eta} ^* - z \right]^{-1} \dd z
\end{equation*}
for $r \in (\textsc{r}_1 \Theta(\eta),r_0)$ for $r_0$,
$\textsc{r}_1$ and $\eta$ as above; it is well-defined since we
proved above that $(\tilde{L}_{\eta} ^* - z)^{-1}$ then
exists. In this section, we first estimate the difference between
the projections $\mathsf\Pi_{r,\eta}$ and $\mathsf\Pi_{r,0}$ when acting on $\psi_0 := \wv^{-\frac{\beta}{2}}$ (the kernel of $\tilde{L}_0$) and projected on $\text{Span}(\psi_0)$ and prove that it goes to zero as $\eta \to 0$; this implies that $\Pi_{r,\eta}$ is non-zero for $r$ and $\eta$ small enough and thus proves the existence of an eigenvalue. Second, we amplify the previous estimate and prove that $||| \mathsf\Pi_{r,\eta} - \mathsf\Pi_{r,0} ||| \to 0$ as $\eta \to 0$, which implies that the dimensions of these two projections are the same for $\eta$ small enough. This implies the existence and uniqueness of the eigenvalue in $|z| \in (\textsc{r}_0 \Theta(\eta),\textsc{r}_1 \Theta(\eta) )$ and quantitative convergence estimates as $\eta \to 0$.

\subsection{Preparation for the first scalar estimate}

Recall $\psi_0 = \wdot^{-\frac{\beta}{2}}$, then
\begin{align*}
  {\sf\Pi}_{r,\eta} \psi_0 - {\sf\Pi}_{r,0} \psi_0
  &= \frac{1}{2i\pi}
    \int_{S(0,r)} \left[ \tilde{L}_{\eta}^* - z \right]^{-1}
    \left[ \tilde{L}_{0}^* -  \tilde{L}_{\eta} \right]
    \left[ \tilde{L}_{0}^* - z \right]^{-1} \psi_0 \dd z\\
  &= -\frac{\eta}{2\pi}
    \int_{S(0,r)}  \left[ \tilde{L}_{\eta}^* - z \right]^{-1}
    \left\{ (v\cdot \sigma) \wv^\beta  \left[
    \tilde L^*_0  - z \right]^{-1} \psi_0 \right\}  \dd z\\
  &= \frac{\eta}{2\pi}
    \int_{S(0,r)} \wv^{-\frac{\beta}{2}} F \, \frac{{\rm d}z}{z} 
\end{align*}
where we have used
\begin{align*}
  \left( \tilde L^* _0 -z \right)^{-1} \psi_0 =
  \left[ \wdot^{\frac{\beta}{2}} L \left(
  \wdot^{\frac{\beta}{2}} \cdot \right) -
  z \right]^{-1} \psi_0 = - \frac{1}{z} \psi_0
\end{align*}
and we have defined $F$ through 
\begin{align*}
  \left[ \tilde{L}_{\eta} ^* - z \right]^{-1}\left[ v' \mapsto
  (v' \cdot \sigma) \lfloor v' \rceil^{\frac{\beta}{2}}\right](v)
  =: \wv^{-\frac{\beta}{2}} F(v),
\end{align*}
that is
\begin{equation}
  \label{eq:F}
  - L^* F - i \eta (v \cdot \sigma) F - z \wv^{-\beta} F =
  (v \cdot \sigma)
\end{equation}
(the dependency of $F$ on $\eta$, $z$ and $\sigma$ is omitted for
readability).

Since ${\sf\Pi}_{r,0} \psi_0 = \psi_0$ and
\begin{equation*}
  \int_{\R^d} {\mathsf\Pi}_{r,0} \psi_0(v) \wv^{-\frac{\beta}2}
  \cM(v) \dd v = \int_{\R^d} \wv^{-\beta} \cM(v) \dd v = \int_{\R^d}
  \cM_\beta(v) \dd v = 1,
\end{equation*}
to prove the existence of an eigenvalue, it is enough to prove
that for $r$ and $\eta$ small enough
\begin{equation*}
  A_{r,\eta} := \left| \int_{\R^d} \left( {\sf\Pi}_{r,\eta}
      \psi_0 - {\sf\Pi}_{r,0} \psi_0 \right) \wv^{-\frac{\beta}2}
    \cM(v) \dd v \right| < 1.
\end{equation*}
  
Using the decomposition~\eqref{eq:decomp-bot-0} one gets
\begin{equation}
  \label{eq:A}
  A_{r,\eta} = \left| \frac{\eta }{2\pi} \int_{S(0,r)} \frac{\cP
      F}{z} \,\dd z \right|.
\end{equation}
The next three steps are devoted to estimating $\cP F$.

\subsection{Localised average estimate}
\label{subsec:average}

Integrate~\eqref{eq:F} against $\chi_R \cM$: the right hand side
vanishes since $\cM$ and $\chi_R$ are even and one gets
\begin{align*}
  - \left\langle L^* F, \chi_R  \right\rangle 
  - i \eta \int_{\R^d} (v \cdot \sigma) F(v) \chi_R (v) \cM(v)
  \dd v - z \cP_R F =0.
\end{align*}
Using the same argument as for~\eqref{eq:trunc2}
and~\eqref{eq:trunc3}, we get
\begin{align}
  \label{eq:estim1}
  \left\vert \cP_R F \right\vert \leq \frac{1}{r} \left[ \eta
  \ell(R) + R^{-\frac{\alpha+\beta}{2}} \right]
  \left\|  \cP ^\bot F  \right\|_{-\beta}
\end{align}
and using~\eqref{eq:diffmR} we deduce, for
$R = \eta^{-\frac{1}{1+\beta}}$ large enough,
\begin{align*}
  |\cP F| \lesssim  \frac{\left(\eta \ell(R) + 
  R^{-\frac{\alpha+\beta}{2}} \right)}{r}
  \left\|  \cP ^\bot F  \right\|_{-\beta}
  \lesssim \frac{\Theta(\eta)^{\frac12}}{r}
  \left\|  \cP ^\bot F  \right\|_{-\beta}.
\end{align*}

\subsection{$L^2$ estimate}

Re-organise~\eqref{eq:F} as
\begin{align*}
  - L^* F - i \eta (v \cdot \sigma)
  \left( F - \frac{1}{i\eta} \right) = z \wv^{-\beta} F,
\end{align*}
integrate it against
\[
  \overline{\left(F - \frac{1}{i\eta} \right)} \cM
\]
and take the real part to obtain
\begin{equation*}
  - \Re \left\langle L^* F,
    \left( F - \frac{1}{i \eta} \right) \right\rangle
  = \Re \left( z \int_{\R^d} \wv^{-\beta} F
    \overline{\left(F -\frac{1}{i\eta} \right)} \cM \dd v
  \right).
\end{equation*}
The left hand side satisfies (using $L 1 =0$ and
Hypothesis~\ref{hyp:coercivity})
\begin{align*}
  - \Re \left\langle L^* F,
  \left( F - \frac{1}{i \eta} \right) \right\rangle =
  - \Re \left\langle L^* F, F \right\rangle \ge \lambda
  \left\| \cP^\bot F \right\|^2_{-\beta},
\end{align*}
and the right hand side is bounded by
\begin{align*}
  \Re \left( z \int_{\R^d} \wv^{-\beta} F  \overline{\left(F -
  \frac{1}{i\eta} \right)} \cM \dd v \right) \le r \| F
  \|_{-\beta}^2 + \frac{r}{\eta} \left|\cP F\right|.
\end{align*}
This results in the estimate (using again the orthogonal
decomposition)
\begin{align*}
  \lambda \left\| \cP^\bot F \right\|^2_{-\beta} 
  & \le r \| F \|^2_{-\beta}  + \frac{r}{\eta} \left|\cP
    F\right|\\
  & \le r \left|\cP F\right|^2 + r \left\| \cP^\bot F
    \right\|_{-\beta} ^2 +  \frac{r}{\eta} \left|\cP F\right|,
\end{align*}
which implies, when $r$ is small, 
\begin{align}
  \label{eq:estim2}
  \left\| \cP^\bot F \right\|^2_{-\beta} \lesssim r
  \left|\cP F \right|^2
  + \frac{r}{\eta} \left|\cP F\right|.
\end{align}

\subsection{Synthesis and the first scalar estimate}

The two previous steps lead to
\begin{equation*}
  \begin{cases}
    \ds |\cP F|^2 \lesssim \frac{\Theta(\eta)}{r^2}
    \left\|  \cP ^\bot F  \right\|_{-\beta} ^2,
    \\[3mm] \ds
    \left\| \cP^\bot F \right\|^2_{-\beta} \lesssim r
    \left|\cP F\right|^2
    +  \frac{r}{\eta} \left|\cP F\right|.
  \end{cases}
\end{equation*}  
Plugging the second estimate into the first one, we obtain
\begin{equation}\label{eq:estimate-m}
  |\cP F| \lesssim \frac{\Theta(\eta)}{r} \left|\cP F\right|
  + \frac{\Theta(\eta)}{\eta r}.
\end{equation}
Given $r \in [\textsc{r}_1 \Theta(\eta),r_0)$ with $\textsc{r}_1$
\emph{large} enough and $\eta$ \emph{small} enough so that
$\frac{\Theta(\eta)}{r}$ is small we get
\begin{align}
  \label{eq:estim-m}
  \vert \cP F \vert \lesssim \frac{\Theta(\eta)}{\eta r}.
\end{align}
Plugging this into~\eqref{eq:A} finally yields
\begin{align*}
  A_{r,\eta} \lesssim
  \frac{\eta}{2\pi} \int_{S(0,r)}
  \frac{\Theta(\eta)}{\eta r^2} \,\dd 
  z  \lesssim \frac{\Theta(\eta)}{r},
\end{align*}
which is as small as wanted for
$r \in (\textsc{r}_1 \Theta(\eta),r_0)$ with $\textsc{r}_1$
\emph{large} enough and $\eta$ \emph{small} enough.

\subsection{Estimating the full norm of the difference of the
  projections at $\psi_0$}

Combining~\eqref{eq:estim2} and~\eqref{eq:estim-m} yields
\begin{align*}
  \left\| \cP^\bot F \right\|_{-\beta} \lesssim
  \frac{\Theta(\eta)}{r^{\frac12} \eta} +
  \frac{\Theta(\eta)^{\frac12}}{\eta}.
\end{align*}
This implies
\begin{align}
  \nonumber
  \Vert {\sf\Pi}_{r,\eta} \psi_0 - {\sf\Pi}_{r,0}
  \psi_0 \Vert
  & \lesssim \frac{\eta}{2\pi}
    \int_{S(0,r)} \frac{1}{r} \Vert F \Vert_{-\beta}\dd z \\
  \nonumber
  & \lesssim \frac{\eta}{2\pi}
    \int_{S(0,r)} \frac{1}{r} \left( \left|\cP F\right| + \left\|
    \cP^\bot F \right\|_{-\beta} \right) \dd z \\
  \label{eq:full1}
  & \lesssim \frac{1}{r} \Theta(\eta) + \frac{1}{r^{\frac12}}
    \Theta(\eta) + \Theta(\eta)^{\frac12} \lesssim
    \frac{\Theta(\eta)}{r}
\end{align}
which is as small as wanted for
$r \in (\textsc{r}_1\Theta(\eta),r_0)$ with $\textsc{r}_1$
\emph{large} enough and $\eta$ \emph{small} enough.

\subsection{Estimating the full norm of the difference of
  projections}

Take now any $\psi \in L^2(\cM)$. Then
$\wdot^{\frac{\beta}{2}} \psi \in L^2(\wdot^{-\beta}\cM)$ and the
following decomposition holds
\begin{align*}
  \psi = \wv^{-\frac{\beta}{2}} \cP
  \left[\wdot^{\frac{\beta}{2}} \psi\right] +
  \wv^{-\frac{\beta}{2}} \cP^\bot
  \left(\wdot^{\frac{\beta}{2}} \psi \right).
\end{align*}
As a consequence,
\begin{align*}
  \left\|  ({\sf\Pi}_{r,\eta}- {\sf\Pi}_{r,0}) \psi \right\| \leq
  \left| \cP \left[\wdot^{\frac{\beta}{2}} \psi \right] \right|
  \left\| ({\sf\Pi}_{r,\eta}- {\sf\Pi}_{r,0}) \psi_0 \right\|
  + \left\| ({\sf\Pi}_{r,\eta}- {\sf\Pi}_{r,0})
  \left[ \wdot^{-\frac{\beta}{2}}
  \cP^\bot \left( \wdot^{\frac{\beta}{2}} \psi
  \right) \right] \right\|.
\end{align*}
The first term in the right hand side is estimated
by~\eqref{eq:full1}. We estimate the second term in the right
hand side by the triangle inequality:
\begin{align*}
  & \left\| ({\sf\Pi}_{r,\eta}- {\sf\Pi}_{r,0})
    \left[\wdot^{-\frac{\beta}{2}}
    \cP^\bot\left(\wdot^{\frac{\beta}{2}} \psi \right)
    \right] \right\| \\ 
  & \qquad \leq
    \left\| {\sf\Pi}_{r,\eta} \left[ \wdot^{-\frac{\beta}{2}}
    \cP^\bot \left( \wdot^{\frac{\beta}{2}} \psi \right) \right]
    \right\|
    + \left\| {\sf\Pi}_{r,0} \left[ \wdot^{-\frac{\beta}{2}}
    \cP^\bot \left( \wdot^{\frac{\beta}{2}} \psi \right) \right]
    \right\|
\end{align*}
and now consider each term separately. Start with
\begin{align}
  \nonumber
  {\sf\Pi}_{r,\eta} \left[ \wdot^{-\frac{\beta}{2}}
  \cP^\bot \left( \wdot^{\frac{\beta}{2}} \psi \right) \right]
  &= \frac{1}{2i\pi}
    \int_{S(0,r)} \left[ \tilde{L}_{\eta} - z  \right]^{-1}
    \left[ \wdot^{-\frac{\beta}{2}}
    \cP^\bot \left( \wdot^{\frac{\beta}{2}} \psi \right) \right]
    \dd z\\
  \label{eq:int-G}
  &= \frac{1}{2i\pi}
    \int_{S(0,r)} \wv^{-\frac{\beta}{2}} F \dd z 
\end{align}
where this time  $F$ satisfies (as before we omit writing the dependency in $\eta,z,\sigma$)
\begin{equation}
  \label{eq:G}
  - L^*F - i \eta (v \cdot \sigma) F - z\wv^{-\beta} F
  =  \wdot^{- \beta} \cP^\bot\left(\wdot^{\frac{\beta}{2}} \psi
  \right).
\end{equation}  

First, test~\eqref{eq:G} on $\overline{F}\cM$, take the real part
and use $\cP[\cP^\bot(\wdot^{\frac{\beta}{2}} \psi)]=0$:
\begin{align*}
  \lambda \left\| \cP^\bot F \right\|_{-\beta} ^2
  &\le (\Re \,z ) \, \| F\|_{-\beta} ^2
    + \Re \, \left\langle \wdot^{-
    \beta} \cP^\bot(\wdot^{\frac{\beta}{2}} \psi), F
    \right\rangle\\
  &= (\Re \, z) \, \| F\|_{-\beta} ^2
    + \Re \, \left\langle \wdot^{-\beta}
    \cP^\bot(\wdot^{\frac{\beta}{2}} \psi), \cP^\bot
    F \right\rangle\\
  &\leq r \vert \cP F \vert^2 + r \| \cP^\bot F\|_{-\beta} ^2
    + \left\| \cP^\bot \left( \wdot^{\frac{\beta}{2}} \psi \right)
    \right\|_{-\beta} \left\| \cP^\bot F \right\|_{-\beta},
\end{align*}   
which implies, since $r <r_0<\lambda$ stays away from $\lambda$,
\begin{equation}
  \label{eq:FF1}
  \left\| \cP^\bot F \right\|^2_{-\beta} \lesssim
  r \, \left| \cP F \right|^2
  + \left\| \cP^\bot \left( \wdot^{\frac{\beta}{2}}
      \psi \right) \right\|_{-\beta}^2
  \lesssim r \, \left| \cP F \right|^2 + \| \psi \|^2.
\end{equation} 
We now estimate $\cP F$. Integrate~\eqref{eq:G} against
$\chi_R \cM$ with $R = \eta^{-\frac{1}{1+\beta}}$
\begin{align*}
  - \left\langle L^* F, \chi_R \right\rangle
  - i \eta \int_{\R^d} (v \cdot \sigma) F(v) \chi_R (v) \cM(v)
  \dd v - z \cP_R F = \cP_R \left[ \cP^\bot \left(
  \wdot^{\frac{\beta}{2}} \psi \right) \right].
\end{align*}
Using the same arguments as in Subsections~\ref{subsec:resolvent}
and~\ref{subsec:average} we obtain
\begin{align*}
  \left\vert \cP_R F \right\vert 
  \lesssim \frac{\Theta(\eta)^{\frac12}}{r} 
  \left\|  \cP ^\bot F  \right\|_{-\beta} + \frac{1}{r}
  \left| \cP_R \left[ \cP^\bot \left( \wdot^{\frac{\beta}{2}} \psi
  \right) \right] \right|.
\end{align*}
Since $\cP[\cP^\bot(\wdot^{\frac{\beta}{2}} \psi)]=0$, we can
estimate $|\cP_R[\cP^\bot(\wdot^{\frac{\beta}{2}} \psi)]|$ as
follows:
\begin{align*}
  \left| \cP_R \left[ \cP^\bot \left( \wdot^{\frac{\beta}{2}} \psi
  \right) \right] \right|
  & = \left|  \cP \left[ \cP^\bot \left( \wdot^{\frac{\beta}{2}} \psi
    \right) \right] -  \cP_R
    \left[ \cP^\bot \left( \wdot^{\frac{\beta}{2}} \psi
    \right) \right] \right| \\
  & \lesssim R^{-\frac{\alpha+\beta}{2}}
    \left\| \cP^\bot \left( \wdot^{\frac{\beta}{2}} \psi \right)
    \right\|_{-\beta}
    \lesssim R^{-\frac{\alpha+\beta}{2}}\| \psi \| \lesssim
    \Theta(\eta)^{\frac12} \| \psi \|.
\end{align*}
We deduce
\begin{align*}
  \left\vert \cP_R F \right\vert 
  \lesssim \frac{\Theta(\eta)^{\frac12}}{r} \left( 
  \left\|  \cP ^\bot F  \right\|_{-\beta} + \| \psi \| \right),
\end{align*}
and using
\begin{equation*}
  |\cP F|^2 \lesssim |\cP_R F|^2  + R^{-(\alpha+\beta)}\left\|  \cP
    ^\bot F  \right\|_{-\beta}^2 \le |\cP_RF|^2  + \Theta(\eta)
  \left\|  \cP^\bot G \right\|_{-\beta}^2
\end{equation*}
we finally get
\begin{equation}
  \label{eq:FF2}
  |\cP F|^2
  \lesssim \frac{\Theta(\eta)}{r^2} \left( 
    \left\|  \cP ^\bot F  \right\|_{-\beta}^2
    + \| \psi \|^2 \right) 
    + \Theta(\eta) \left\|  \cP ^\bot F \right\|_{-\beta}^2
  \lesssim \frac{\Theta(\eta)}{r^2} \left( 
    \left\|  \cP ^\bot F  \right\|_{-\beta}^2
    + \| \psi \|^2 \right).
\end{equation} 
Combining~\eqref{eq:FF1}--\eqref{eq:FF2} implies for
$r \in [\textsc{r}_1 \Theta(\eta),r_0)$ with $\textsc{r}_1$
\emph{large} enough and $\eta$ \emph{small} enough
\begin{align*}
  |\cP F|^2 \lesssim \frac{\Theta(\eta)}{r^2} \|\psi \|^2 \quad
  \text{ and thus } \quad  \left\|  \cP ^\bot F
  \right\|_{-\beta}^2 \lesssim \frac{\Theta(\eta)}{r} \|\psi
  \|^2 + \|\psi \|^2 \lesssim \| \psi\|^2.
\end{align*}  
Plugging these estimates into~\eqref{eq:int-G} yields
\begin{equation*}  
  \left\| {\sf\Pi}_{r,\eta} \left[ \wdot^{-\frac{\beta}{2}}
      \cP^\bot \left( \wdot^{\frac{\beta}{2}} \psi
      \right) \right] \right\|
  \leq r \Vert F \Vert_{-\beta} \lesssim r |\cP F|
  + r  \| \cP^\bot F \|_{-\beta} \lesssim \Theta(\eta)^{\frac12} \|
  \psi \| + r \| \psi \|.
\end{equation*}
 
We now come to the estimate of
\begin{align*}
  {\sf\Pi}_{r,0}\left[ \wdot^{-\frac{\beta}{2}}
  \cP^\bot \left( \wdot^{\frac{\beta}{2}} \psi \right) \right]
  &= \frac{1}{2i\pi}
    \int_{S(0,r)} \left[ \tilde{L}_{0} - z  \right]^{-1}
    \left[ \wdot^{-\frac{\beta}{2}}
    \cP^\bot \left( \wdot^{\frac{\beta}{2}} \psi
    \right) \right] \dd z\\
  &= \frac{1}{2i\pi}
    \int_{S(0,r)} \wv^{-\frac{\beta}{2}} F \dd z 
\end{align*}
where $F$ satisfies this time
\begin{equation}\label{eq:G0}
  - L^* F - z \wv^{-\beta} F = \wv^{- \beta} \cP^\bot \left(
    \wdot^{\frac{\beta}{2}} \psi \right).
\end{equation}
Integrating this equation against $\cM$ implies $\cP F=0$ since
$\langle L^* F,1 \rangle =\cP[\cP^\bot(\wdot^{\frac{\beta}{2}}
\psi)]=0$ and $z \not =0$. Hypothesis~\ref{hyp:coercivity} then
implies, since $r <r_0<\lambda$ is away from $\lambda$:
\begin{align*}
  \left\| F \right\|_{-\beta} =
  \left\| \cP^\bot F \right\|_{-\beta} \lesssim \| \psi \|,
\end{align*}
and thus
\begin{align*}
  \left\| {\sf\Pi}_{r,0} \left[ \wdot^{-\frac{\beta}{2}}
  \cP^\bot \left( \wdot^{\frac{\beta}{2}} \psi \right) \right]
  \right\| \lesssim r \| \psi \|.
\end{align*}

The conclusion is that for any $\psi \in L^2(\cM)$,
\begin{equation*}
  \left\| \left({\sf\Pi}_{r,\eta}- {\sf\Pi}_{r,0}\right)
    \psi \right\|
  \lesssim \Theta(\eta)^{\frac12} \|
  \psi \| + r \| \psi \|
\end{equation*}
which implies (combining all the previous conditions), for
$r \in (\textsc{r}_1 \Theta(\eta),r_0)$ with $r_0>0$ \emph{small}
enough \emph{independently} of $\eta$ and $\textsc{r}_1$
\emph{large} enough \emph{independently} of $\eta$ and $\eta$
\emph{small} enough,
\begin{equation*}
  \left\| {\sf\Pi}_{r,\eta}- {\sf\Pi}_{r,0} \right\|_{L^2(\cM) \to
    L^2(\cM)} < 1. 
\end{equation*}

It implies that, for $r \in (\textsc{r}_1 \Theta(\eta),r_0)$ and
$\eta$ small enough, the projections ${\sf\Pi}_{r,\eta}$ and
${\sf\Pi}_{r,0}$ both exist thanks to
Subsection~\ref{subsec:resolvent} and their dimensions are the
same, i.e. $1$, which proves existence and uniqueness of an
eigenvalue $\mu(\eta) \in B(0,r_0)$ with
$|\mu(\eta)| \in (\textsc{r}_0 \Theta(\eta),\textsc{r}_1
\Theta(\eta))$. This implies that this eigenvalue is real: if
$(\psi_{\eta},-\mu(\eta))$ is an eigenpair of $\tilde{L}^* _\eta$
with $\mu(\eta) \in B(0,r_0)$, then so is
$(\overline{\psi_{\eta}(- \cdot)},-\overline{\mu(\eta)})$. Since
$\tilde L^* _\eta \le 0$ and $0$ is not an eigenvalue for
$\eta \not=0$, this proves that $\mu(\eta) >0$.

\subsection{Estimate on the branch as $\eta \to 0$}

Denote $\phi_\eta := \wdot^{\frac{\beta}{2}} \psi_\eta$ and
normalize 
\begin{align*}
  \int_{\R^d} \psi_\eta(v) \wv^{-\frac{\beta}{2}} \cM(v) \dd v =
  \int_{\R^d} \phi_\eta(v) \wv^{-\beta} \cM(v) \dd v =
  \int_{\R^d} \phi_\eta(v) \cM_\beta(v) \dd v = 
  m\left[\phi_\eta \right] =1.
\end{align*}
Then integrating the equation against $\overline{\psi_\eta}\cM$,
taking the real part and using Hypothesis~\ref{hyp:coercivity} yields
\begin{align*}
  \lambda \left\| \psi_\eta - \psi_0 \right\|^2
  \le \mu(\eta) \left\|
  \psi_\eta \right\|^2 \lesssim \mu(\eta) \left\|
  \psi_\eta - \psi_0 \right\|^2 + \mu(\eta)
\end{align*}
where we have used $\| \psi_0 \|=1$. Hence for $\eta$ small
enough we deduce
\begin{align*}
  \left\| \phi_\eta - 1 \right\|_{-\beta} =
  \left\| \psi_\eta - \psi_0
  \right\| \lesssim \mu(\eta)^{\frac12}.
\end{align*}
This concludes the proof of Lemma~\ref{lem:existencespectral}.

\section{Proof of Lemma~\ref{lem:ratespectral}
  (scaling of the eigenvalue)} \label{sec:ratespectral}

In this section we prove Lemma~\ref{lem:ratespectral}, assuming all Hypotheses~\ref{hyp:functional}--\ref{hyp:coercivity}--\ref{hyp:large-v}--\ref{hyp:scalinginfinite}. Consider the unique eigenpair $(\phi_{\eta},\mu(\eta))$ that satisfies $\mu(\eta) \in B(0,r_1)$ and
\begin{equation}
  \label{eq:unweighted}
  - L^* \phi_{\eta} -i \eta (v \cdot \sigma) \phi_{\eta}
  = \mu(\eta) \wv^{-\beta} \phi_{\eta} \quad
  \text{ and } \quad \int_{\R^d} \phi_{\eta}(v) \cM_\beta(v) \dd v =1.
\end{equation}

\subsection{Proof in the case $\alpha > 2+\beta$}

The function $F_\eta := \frac{\Im \phi_\eta}{\eta}$ satisfies
\begin{equation*}
  -L^*F_\eta - \mu(\eta) \wv^{-\beta} F_\eta
  = (v \cdot \sigma) \Re \phi_\eta \quad \text{ and } \quad
  \int_{\R^d} F_{\eta}(v) \cM_\beta \dd v = 0.
\end{equation*}

Since (by Hypothesis~\ref{hyp:coercivity}) $\tilde L^*$ is
invertible on the $L^2_v(\cM)$-orthogonal of
$\wdot^{-\frac{\beta}{2}}$, and
$v \mapsto \wv^{\frac{\beta}{2}} (v \cdot \sigma)$ belongs to
$L^2_v(\cM)$ when $\alpha >2+\beta$, we can then define a solution
$F \in L^2_v(\wdot^{-\beta} \cM)$ to
\begin{align*}
  - L^* F = (v \cdot \sigma)
  \quad \text{ with } \quad  \int_{\R^d}
  F(v) \cM_\beta \dd v = 0.
\end{align*}
The difference $F_\eta-F$ satisfies
\begin{equation*}
  -L^*\left( F_\eta - F \right) - \mu(\eta) \wv^{-\beta} \left(
    F_\eta - F \right) 
  = (v \cdot \sigma) \left[ \Re \phi_\eta -1 \right] + \mu(\eta)
  \wv^{-\beta} F.
\end{equation*}
Integrate this against $( F_\eta - F ) \cM$ and use Hypothesis~\ref{hyp:coercivity}:
\begin{align*}
  \left[ \lambda - \mu \left(\eta\right) \right]
  \left\| F_\eta -F \right\|_{-\beta}^2
  &\leq  \int_{\R^d} (v \cdot \sigma)
    \left( \Re \phi_\eta - 1 \right)
    \left(F_\eta - F\right) \cM \dd v 
    + \mu(\eta) \int_{\R^d} F \left(F_\eta - F\right) 
    \cM_\beta \dd v\\
  &\leq \Vert \Re \phi_\eta - 1 \Vert_{2+\beta}\Vert F_\eta -
    F\Vert_{-\beta}
    + \mu(\eta)  \Vert F \Vert_{-\beta}\Vert F_\eta -F\Vert_{-\beta}.
\end{align*}
Write for any $\ell\in (2+\beta,\alpha)$
\begin{align*}
  \left\| \Re \phi_\eta - 1 \right\|_{2+\beta}
  \leq \left\| \Re \phi_\eta - 1 \right\|_{-\beta}^{\zeta}
  \left\| \Re \phi_\eta - 1 \right\|_{\ell}^{1-\zeta}
  \leq \mu(\eta)^{\frac{\zeta}{2}}
  \left\| \Re \phi_\eta - 1 \right\|_{\ell}^{1-\zeta}
\end{align*}
with $\mathsf{z} = \frac{\ell-(2+\beta)}{\ell+\beta}\in (0,1)$,
then Hypothesis~\ref{hyp:scalinginfinite}-(i) implies
\begin{align*}
  \| \Re \phi_\eta - 1 \|_{2+\beta} \lesssim
  \mu(\eta)^{\frac{\mathsf{z}}{2}} \to 0
  \quad \text{ as } \quad \eta \to 0,
\end{align*}
and thus, since $\alpha >1$ (combining $\alpha>2+\beta$ and
$\alpha+\beta>0$) and 
\begin{align*}
  \| F \|_{-\beta} \lesssim \left\| (v \cdot \sigma) 
  \right\| \lesssim 1,
\end{align*}
we deduce 
\begin{align*}
  \left\| F_\eta -F \right\|_{-\beta} \lesssim
  \mu(\eta)^{\frac{\mathsf{z}}{2}} \to 0 
  \quad \text{ as } \quad \eta \to 0. 
\end{align*}

Finally,
\begin{align*}
  \left\vert \frac{\mu(\eta)}{\eta^2} -
  \int_{\R^d} (v\cdot \sigma) F(v) \cM(v) \dd v \right\vert
  \leq \left\vert \int_{\R^d} (v\cdot \sigma) (F_\eta (v) - F(v))
  \cM(v)
  \dd v \right\vert \\
  \lesssim  \Vert 1 \Vert_{2+\beta} \Vert F_\eta -F\Vert_{-\beta}
  \lesssim  \mu(\eta)^{\frac{\mathsf{z}}{2}} \to 0 \text{ as }
  \eta \to 0, 
\end{align*}
which identifies the limit of $\mu(\eta)$ and provides a rate.

\subsection{Proof in the case $\alpha \in (-\beta, 2+\beta)$}

Take $0 \le \chi \le 1$ smooth that is $1$ on $B(0,R_0)$ and $0$  outside $B(0,2R_0)$. Integrate \eqref{eq:unweighted} against $\Theta(\eta)^{-1} \chi(\cdot\eta^{\frac{1}{1+\beta}}) \cM$ and take the real part:
\begin{align}
  \label{eq:mu-chi}
  &\frac{\mu(\eta)}{\Theta(\eta)} + \frac{1}{\Theta(\eta)} \eta
    \left\langle  (v \cdot \sigma) \Im\phi_\eta ,
    \chi\left(\cdot\eta^{\frac{1}{1+\beta}}\right) \right\rangle
  \\ \nonumber
  &= - \frac{\mu(\eta)}{\Theta(\eta)}
    \left( \left\langle \wv^{-\beta} \Re\phi_\eta,
    \chi\left(\cdot\eta^{\frac{1}{1+\beta}}\right) \right\rangle
    - 1 \right) -
    \frac{1}{\Theta(\eta)} \left\langle L^*(\Re\phi_\eta-1),
    \chi\left(\cdot\eta^{\frac{1}{1+\beta}}\right) \right\rangle
  \\ \nonumber
  &= - \frac{\mu(\eta)}{\Theta(\eta)}
    \left\langle \wv^{-\beta} \Re\phi_\eta,
    \chi\left(\cdot\eta^{\frac{1}{1+\beta}}\right) -1 \right\rangle
    -\frac{1}{\Theta(\eta)} \left\langle \Re\phi_\eta-1,
    L\left( \chi\left(\cdot\eta^{\frac{1}{1+\beta}}\right)
    \right)\right\rangle. 
\end{align}
The first term in the right hand side is controlled by  
\begin{align*}
  \left\vert \frac{\mu(\eta)}{\Theta(\eta)}
  \left\langle \wv^{-\beta} \Re\phi_\eta,
  \chi\left(\cdot\eta^{\frac{1}{1+\beta}}\right) -1 \right\rangle
  \right\vert
  \lesssim  \left\vert  \left\langle \wv^{-\beta} \Re\phi_\eta,
  \chi\left(\cdot\eta^{\frac{1}{1+\beta}}\right) -1 \right\rangle
  \right\vert
  \lesssim R_0 ^{-\frac{\alpha+\beta}{2(1+\beta)}}
  \eta^{\frac{\alpha+\beta}{2(1+\beta)}}
\end{align*}
and the second term is controlled by
\begin{align*}
  \left\vert\frac{1}{\Theta(\eta)} \left\langle \Re\phi_\eta-1,
  L\left( \chi(\cdot\eta^{\frac{1}{1+\beta}})
  \right)\right\rangle \right\vert
  &\leq \frac{1}{\Theta(\eta)}
    \left\| \phi_\eta - 1 \right\|_{-\beta} \left\|
    L\left[ \chi\left( \cdot \eta^{\frac{1}{1+\beta}}
    \right) \right] \right\|_\beta\\
  &\lesssim \Theta(\eta)^{-\frac12} \left\|
    L\left[ \chi\left( \cdot \eta^{\frac{1}{1+\beta}}
    \right) \right] \right\|_\beta \\
  &\lesssim \Theta(\eta)^{-\frac12}
    \eta^{\frac{\alpha+\beta}{2(1+\beta)}}
    R_0 ^{-\frac{\alpha+\beta}{2(1+\beta)}} \lesssim
    R_0 ^{-\frac{\alpha+\beta}{2(1+\beta)}}.
\end{align*}
The second term in the left hand side satisfies (changing
variable to $u=v \eta^{\frac{1}{1+\beta}}$)
\begin{align*}
\frac{\eta}{\Theta(\eta)} 
  \left\langle  (v \cdot \sigma) \Im\phi_\eta ,
  \chi(\cdot\eta^{\frac{1}{1+\beta}}) \right\rangle
  = c_{\alpha,\beta}\int_{\R^d} (u \cdot \sigma)
  \Im\Phi_\eta(u)  |u|_\eta ^{-d-\alpha}  \chi(u) \dd u
\end{align*}
and we deduce 
\begin{align*}
  \left\vert \frac{\mu(\eta)}{\Theta(\eta)}
  + c_{\alpha,\beta}\int_{\R^d} (u \cdot \sigma)  \Im\Phi_\eta(u)
  |u|_\eta ^{-d-\alpha}  \chi(u) \dd u \right\vert
  \lesssim R_0^{-\frac{\alpha+\beta}{2}}.
\end{align*}
Then observe that assumption~\eqref{eq:mmt-fract} in
Hypothesis~\ref{hyp:scalinginfinite}-(ii) implies the uniform
integrability of the integrand on the support of $\chi$ and the
convergence of the integral as $\eta \to 0$ for a given $\chi$.

All in all we have the double limit
\begin{align*}
  \int_{\R^d} (u \cdot \sigma)  \Im\Phi_\eta(u)
  |u|_\eta ^{-d-\alpha}  \chi(u)
  \dd u \xrightarrow[R_0 \to \infty]{\eta \to
  0}  \int_{\R^d} (u \cdot \sigma)  \Im\Phi(u)
  |u|^{-d-\alpha}  \dd u.
\end{align*}
This double limit thus proves that
$\frac{\mu(\eta)}{\Theta(\eta)}$ converges and
\begin{equation*}  
  \lim_{\eta \to 0} \frac{\mu(\eta)}{\Theta(\eta)}
  = c_{\alpha,\beta} \int_{\R^d} (u \cdot \sigma)  \Im\Phi(u)
  |u|^{-d-\alpha} \dd u. 
\end{equation*}
This limit then belongs to $(\textsc{r}_0,\textsc{r}_1)$
because of estimates on $\mu(\eta)$ already established.

\subsection{Proof in the case $\alpha = 2+\beta$}

Take $0 \le \chi \le 1$ smooth that is $1$ on $B(0,1)$ and $0$ outside $B(0,2)$. Consider again~\eqref{eq:mu-chi} (with now $\Theta(\eta)=\eta^2|\ln \eta|$) and estimate
\begin{align*}
  \left\vert \frac{\mu(\eta)}{\Theta(\eta)} 
  \left\langle \wv^{-\beta} \Re\phi_\eta,
  \chi\left(\cdot\eta^{\frac{1}{1+\beta}}\right ) -1
  \right\rangle \right\vert
  \lesssim  \left\vert  \left\langle \wv^{-\beta} \Re\phi_\eta,
  \chi\left(\cdot\eta^{\frac{1}{1+\beta}}\right) -1 \right\rangle
  \right\vert \lesssim \eta  
\end{align*}
and 
\begin{align*}
  \left\vert\frac{1}{\eta^2|\ln \eta|} \left\langle
  \Re\phi_\eta-1, L\left( \chi(\cdot\eta^{\frac{1}{1+\beta}})
  \right) \right\rangle \right\vert
  &\leq \frac{1}{\eta^2|\ln \eta|} 
    \left\| \phi_\eta - 1 \right\|_{-\beta} \left\|
    L\left[ \chi\left( \cdot \eta^{\frac{1}{1+\beta}}
    \right) \right] \right\|_\beta\\
  &\lesssim \frac{1}{\eta |\ln(\eta)|^{\frac12}} \left\|
    L\left[ \chi\left( \cdot \eta^{\frac{1}{1+\beta}}
    \right) \right] \right\|_\beta
    \lesssim \frac{1}{|\ln(\eta)|^{\frac12}}.  
\end{align*}
We have also
\begin{align*}
  \frac{1}{\eta |\ln \eta|}
  \left\langle  (v \cdot \sigma) \Im\phi_\eta ,
  \chi \left(\cdot\eta^{-\frac{1}{1+\beta}}\right) \right\rangle
  = \frac{c_{\alpha,\beta}}{|\ln \eta|} \int_{\R^d} (u \cdot
  \sigma) \Im\Phi_\eta  |u|_\eta ^{-d-\alpha}  \chi(u) \dd u
\end{align*}
which gives 
\begin{equation}
  \label{eq:musurtheta}
  \left\vert \frac{\mu(\eta)}{\Theta(\eta)} +
    \frac{c_{\alpha,\beta}}{|\ln(\eta)|}
    \int_{\R^d} (u \cdot \sigma)  \Im\Phi_\eta  |u|_\eta
    ^{-d-\alpha}  \chi(u) \dd u \right\vert \lesssim
  \frac{1}{|\ln(\eta)|^{\frac12}}    + \eta.
\end{equation}  
Let us decompose 
\begin{align*}
  & \frac{c_{\alpha,\beta}}{|\ln \eta|} \int_{\R^d} (u \cdot
    \sigma) \Im\Phi_\eta  |u|_\eta ^{-d-\alpha}  \chi(u) \dd u\\
  &= \frac{c_{\alpha,\beta}}{|\ln \eta|} \int_{\vert u \vert \leq
    \eta^{\frac{1}{1+\beta}}}
    (u \cdot \sigma)  \Im\Phi_\eta  |u|_\eta ^{-d-\alpha}
    \chi(u) \dd u
    + \frac{c_{\alpha,\beta}}{|\ln \eta|} \int_{\vert u \vert
    \geq \eta^{\frac{1}{1+\beta}}} (u \cdot \sigma)  \Im\Phi_\eta
    |u|_\eta ^{-d-\alpha}  \chi(u) \dd u.
\end{align*}
The first term is bounded by
\begin{align*}
  \frac{c_{\alpha,\beta}}{|\ln \eta|}\int_{\vert u \vert \leq 
  \eta^{\frac{1}{1+\beta}}} (u \cdot \sigma)  \Im\Phi_\eta(u)
  |u|_\eta ^{-d-\alpha}  \chi(u) \dd u
  &= \frac{c_{\alpha,\beta}}{\eta |\ln \eta|} \int_{\vert
    v \vert \leq 1}
    (v \cdot \sigma)  \Im\phi_\eta (v)  \cM(v) \dd v\\
  &\lesssim \frac{\Theta(\eta)^{\frac12}}{\eta |\ln \eta|}
    \lesssim
    \frac{1}{|\ln(\eta)|^{\frac12}}.
\end{align*}
We approximate, using
Hypothesis~\ref{hyp:scalinginfinite}-(ii)-(a), 
\begin{equation*}
  \frac{1}{|\ln(\eta)|} \left\vert 
    \int_{\vert u \vert \geq \eta^{\frac{1}{1+\beta}}} (u \cdot
    \sigma) \Big[ \Im\Phi_\eta(u) - \Im\Phi(u) \Big]
    |u|_\eta ^{-d-\alpha}  \chi(u) \dd u  \right\vert 
  \lesssim \mathfrak{a}(\eta).
\end{equation*}

Define 
\begin{equation*}
  N(\eta) := \int_{|u| \ge \eta^{\frac{1}{1+\beta}}} (u \cdot \sigma)
  \Im \Phi (u) |u|^{-d-\alpha} \chi(u) \dd u.
\end{equation*}  
Observe that since
$\left\vert \Im\Phi(u) \right\vert \lesssim |u|^{1+\beta}$, and
$\alpha = 2 + \beta$,
\begin{align*}
  \left| N(\eta) - \int_{\vert u \vert \geq \eta^{\frac{1}{1+\beta}}}
  (u \cdot \sigma)  \Im\Phi(u)  |u|_\eta ^{-d-\alpha}  \chi(u)
  \dd u \right|
  &\leq  \int_{2 \geq \vert u \vert \geq \eta^{\frac{1}{1+\beta}}}
    |u|^{2+\beta}  \left\vert  |u|_\eta ^{-d-\alpha}
    -  |u|^{-d-\alpha}  \right\vert \dd u \\
  &\leq  \int_{1 \leq \vert v \vert \leq 2\eta^{-\frac{1}{1+\beta}} }
    |v|^{-d } \left\vert  |v|^{d+\alpha} \wv^{-d-\alpha}
    - 1   \right\vert \dd v\\
  &\lesssim 1,
\end{align*} 
since
$\left\vert |v|^{d+\alpha} \wv^{-d-\alpha} - 1 \right\vert
\sim_{v \to \infty} \frac{d+\alpha}{2} \frac{1}{1+|v|^2}$.  We
get, using Hypothesis~\ref{hyp:scalinginfinite}-(ii)-(a),
\begin{align*}
  - \eta N'(\eta) \sim \frac{1}{1+\beta} \int_{\sigma'\in
  \mathbb{S}^{d-1}} \left( \sigma \cdot \sigma' \right)
  \frac{\Im \Phi\left(  
  \eta^{\frac{1}{1+\beta}} \sigma' \right)}{\eta}
  \dd \sigma' \sim
  \frac{1}{1+\beta}  \int_{\sigma'\in
  \mathbb{S}^{d-1}} \left( \sigma \cdot \sigma' \right)
  \Omega(\sigma') \dd \sigma'.
\end{align*}
Apply then L'Hôpital's rule to deduce
\begin{equation*}
  \lim_{\eta \to 0} \frac{N(\eta)}{|\ln
    \eta|}  = \frac{1}{1+\beta}  \int_{\sigma'\in
    \mathbb{S}^{d-1}} \left( \sigma \cdot \sigma' \right)
  \Omega(\sigma') d \sigma'.
\end{equation*}
We conclude by taking $\eta \to 0$ in \eqref{eq:musurtheta}.

\section{Proof of Lemma \ref{lem:diffcoeff} (the diffusion
  coefficient)} \label{sec:diffcoeff}
\label{sec:diff-coef}

We assume $\alpha \ge 0$. Lemma~\ref{lem:diffcoeff} follows from Lemma~\ref{lem:ratespectral}, the definition~\eqref{eq:scaling-function} of $\theta$, and the 
\begin{lemma}\label{lem:zeromoment}
  Assume  Hypotheses~\ref{hyp:functional}--\ref{hyp:coercivity}--\ref{hyp:large-v}--\ref{hyp:scalinginfinite}. Then, the convergence
  \begin{equation*}
    \left\langle 1, \phi_{\eta}\right\rangle \sim_{\eta
      \to 0}
    \left\lbrace  
      \begin{array}{ll}
        \ds
        \| \cM \|_{L^1(\R^d)}
        &\text{ when } \alpha > 0 \\[3.5mm]
        \ds
        \frac{|\mathbb{S}^{d-1}|}{1+\beta} \vert \ln(\eta)\vert
        &\text{ when } \alpha =0 
      \end{array}
    \right.
  \end{equation*}
  holds, with explicit convergence rate.
\end{lemma}
\begin{proof}
  When $\alpha >0$, the integral
  \begin{align*}
    \frac{c_{\alpha,\beta}}{c_{\alpha,0}} = \left\langle 1, 1
    \right\rangle
    = \int_{\R^d} \cM \dd v < +\infty
  \end{align*}
  is well defined and, choosing $\ell \in (0,\alpha)$,
  \begin{align*}
    \left| \left\langle 1, \phi_{\eta} \right\rangle -
    \left\langle 1, 1 \right\rangle \right|
    &\leq \left| \left\langle 1, \phi_{\eta} - 1 \right\rangle
      \right| \leq
      \left\| 1 \right\|_{\min(\ell,\beta)}
      \left\| \phi_{\eta} - 1 \right\|_{-\min(\ell,\beta)} \\
    & \lesssim \left\| \phi_{\eta} - 1 \right\|_{-\beta}^a \left\|
      \phi_{\eta} - 1 \right\|_0 ^{1-a} \lesssim
      \mu(\eta)^{\frac{a}{2}}
  \end{align*}
  with $a = \min(\frac{\ell}{\beta},1) \in (0,1]$, which shows
  (with explicit rate)
  \begin{equation*}
    \left\langle 1, \phi_{\eta}\right\rangle \xrightarrow[]{\eta
    \to 0} \langle 1, 1 \rangle = 
    \frac{c_{\alpha,\beta}}{c_{\alpha,0}} = \| \cM
    \|_{L^1(\R^d)} \quad \text{ when } \alpha >0.
  \end{equation*}

In the case $\alpha=0$,
\begin{align*}
  \int_{\R^d} \phi_\eta(v) \cM(v) \dd v =
  \int_{|v| \le \eta^{-\frac{1}{1+\beta}}} \phi_\eta(v)
  \cM(v) \dd v +
  \int_{|v|\ge \eta^{-\frac{1}{1+\beta}}} \phi_\eta(v) \cM(v) \dd v.  
\end{align*}
The second term is estimated by
\begin{align*}
  \int_{|v|\ge \eta^{-\frac{1}{1+\beta}}} \phi_\eta(v) \cM(v) \dd
  v
  &= c_{0,\beta} \int_{|u|\ge 1} \Phi_\eta(u) |u|_\eta ^{-d} \dd u\\
  &= c_{0,\beta} \left( \int_{|u|\ge 1} \vert \Phi_\eta(u)\vert^2
    |u|_\eta ^{-d+\beta} \dd u \right)^\frac12 \left(
    \int_{|u|\ge 1} |u|_\eta ^{-d-\beta} \dd u \right)^\frac12
  \lesssim 1,
\end{align*}
using the moment bounds~\eqref{eq:mmt-fract-bis}. The first term
is decomposed into
\begin{equation*}
  \int_{|v| \le \eta^{-\frac{1}{1+\beta}}} \phi_\eta(v) \cM(v)
  \dd v
  =  \int_{|v| \le \eta^{-\frac{1}{1+\beta}}}
    \left(\phi_\eta(v) - 1\right) \cM(v) \dd v
    + \int_{|v| \le \eta^{-\frac{1}{1+\beta}}} \cM(v) \dd v.
\end{equation*}
Since 
\begin{equation*}
  \left\vert \int_{|v| \le \eta^{-\frac{1}{1+\beta}}}
  \left(\phi_\eta(v) - 1\right) \cM(v) \dd v \right\vert
   \leq \left\| \phi_\eta(v) - 1 \right\|_{-\beta} \left\|
    {\bf 1}_{|\cdot| \le \eta^{-\frac{1}{1+\beta}}} \right\|_\beta
   \lesssim \mu(\eta)^\frac12 \eta^{-\frac{\beta}{2(1+\beta)}}
    \lesssim 1,
\end{equation*}
we deduce
\begin{align*}
  \int_{\R^d} \phi_\eta(v) \cM(v) \dd v \sim
  \int_{|v| \le \eta^{-\frac{1}{1+\beta}}} \cM(v) \dd v
  \sim c_{0,\beta} \frac{\left|\mathbb{S}^{d-1}\right|}{1+\beta}
  |\ln \eta| 
\end{align*}
with explicit error term. 
\end{proof}

\section{Proof of the hypotheses for scattering
  equations}\label{sec:scatt}

Let us consider $\alpha \ge 0$ and the scattering operator,
written both on ``$f$'' or ``$h=\frac{f}{\cM}$'':
\begin{equation*}
  \begin{cases}
    \ds \mathcal{L} f(v) = \int_{\R^d} b(v,v')\,\left[ f(v')\,
      \cM(v)-f(v)\, \cM(v')\right] \dd v', \\[3mm] \ds
    Lh(v) = \int_{\R^d} b(v,v')\,\cM(v') \left[
      h(v')\,-h(v) \right] \dd v'.
  \end{cases}
\end{equation*}
We assume that $b$ is $C^1$, that the operator conserves
the local mass
\begin{align*}
  \int_{\R^d}\big[\mathrm b(v,v')
  -\mathrm b(v',v)\big]\,\cM(v')\dd  v'=0
\end{align*}
and that the \emph{collision kernel} $b$ and \emph{collision
  frequency} $\nu(v) := \int_{\R^d} b(v,v') \, \cM(v') \dd v'$
satisfy, for some constant $\nu_0 >0$,
\begin{align*}
  \wv^{-\beta} \lesssim \nu(v) \lesssim
  \wv^{-\beta}, \qquad
  \lambda^\beta \nu (\lambda u) \sim_{\lambda \to \infty} \nu_0
  |u|^{-\beta} \quad \text{ and } \quad 
  \Vert b(v,\cdot) \Vert_\beta + \Vert b(\cdot,v) \Vert_\beta
  \lesssim \wv^{-\beta}.
\end{align*}
This includes $b(v,v') = \wv^{-\beta}\wvp^{-\beta}$ for any
$\alpha + \beta >0$, $b(v,v') = \lfloor v -v' \rceil^{-\beta}$
when $\beta \ge 0$ and $\alpha > 3 \beta$, and even
$b(v,v') = \vert v -v' \vert^{-\beta}$ when $\beta <0$ and
$\alpha + \beta >0$.

\subsection{Proof of Hypothesis~\ref{hyp:coercivity}}

Hypothesis~\ref{hyp:coercivity} is standard and proved for
instance in~\cite{MR1803225}.

\subsection{Proof of Hypothesis~\ref{hyp:large-v}}

We perform the following calculations (the case of $\tilde \chi_R$ is similar):
\begin{align*}
  & \Vert L(\chi_R)\Vert_\beta^2
  = \int_{\R^d} \wv^\beta \vert L( \chi_R) \vert^2 \cM(v) \dd v \\
  &\leq \int_{\R^d} \wv^\beta\nu(v) \int_{\R^d} \vert  \chi_R(v)
    - \chi_R(v') \vert^2 \, \mathrm b(v,v')\, \cM(v') \,
    \cM(v) \dd v' \dd v \\
  &\lesssim
    \iint_{\R^d \times \R^d} \vert  \chi_R(v)  - \chi_R(v')
    \vert^2 \, \mathrm b(v,v') \, \cM(v) \, \cM(v') \dd v \dd v'\\
  &\lesssim \iint_{\{ \vert v \vert < R \} \times \R^d} \vert  \chi_R(v)  - \chi_R(v')
    \vert^2 \, \mathrm b(v,v') \, \cM(v) \, \cM(v') \dd v \dd v' \\
  & \qquad + \iint_{\{ \vert v \vert> R \}\times \R^d} \vert  \chi_R(v)  - \chi_R(v')
    \vert^2 \, \mathrm b(v,v') \, \cM(v) \, \cM(v') \dd v \dd v' \\
  &\lesssim \iint_{\R^d \times \{ \vert v' \vert> R \}}
    \mathrm b(v,v') \, \cM(v) \, \cM(v') \dd v \dd v' 
    +\iint_{\{ \vert v \vert> R \} \times \R^d}
    \mathrm b(v,v') \, \cM(v) \, \cM(v') \, \dd v' \dd v\\
  &\lesssim \left\| \chi_R ^c \cM \right\|_{-\beta} \lesssim
    R^{-\frac{(\alpha+\beta)}{2}}  \xrightarrow[R \to \infty]{} 0.
\end{align*}

\subsection{Proof of Hypothesis~\ref{hyp:scalinginfinite}}

The eigenvalue problem can be written 
\begin{align*}
   \left( L^{*,+}\phi_\eta \right) (v) := \int_{\R^d} b(v',v)
    \cM(v')\phi_\eta(v') \dd v' = \left(\nu(v) - i \eta  (v \cdot
    \sigma) -\mu(\eta)\wv^{-\beta} \right) \phi_\eta(v) 
\end{align*}
with the normalization
$\int_{\R^d} \cM_\beta(v')\phi_\eta(v') \dd v' = 1$. Observe
first that Hypothesis~\ref{hyp:coercivity} implies
\begin{align*}
  \left\| \phi_\eta - 1 \right\|_{-\beta}^2 \le \mu(\eta) \left\|
  \phi_\eta \right\|_{-\beta}^2
\end{align*}
and thus, for $\eta$ small enough
\begin{align*}
  \left\| \phi_\eta \right\|_{-\beta}^2 \le \frac{\lambda}{\lambda -
  \mu(\eta)}
\end{align*}
is uniformly bounded as $\eta \to 0$. Observe second that
\begin{equation*}
  \left\vert L^{*,+} (\phi_\eta)(v)  \right\vert \leq
  \left\|b(\cdot,v) \right\|_\beta
  \left\| \phi_\eta \right\|_{-\beta} \lesssim \wv^{-\beta}
 \end{equation*}
which yields, for $\eta$ small enough,  
\begin{equation*}
  \vert \phi_\eta(v) \vert \lesssim
  \frac{\wv^{-\beta}}{\left[
      \left(\nu(v)-\mu(\eta)\wv^{-\beta}\right)^2  + \eta^2  (v
      \cdot \sigma)^2 \right]^\frac12} \lesssim
  \frac{1}{ \wv^{\beta} \nu(v)-\mu(\eta)} \lesssim 1,
\end{equation*}
i.e. $\phi_\eta$ is uniformly bounded in $L^\infty(\R^d)$ as
$\eta \to 0$, and Hypothesis~\ref{hyp:scalinginfinite}-(i) when
$\alpha > 2 + \beta$ follows.

The rescaled eigenvector $\Phi_\eta$ satisfies 
\begin{equation*}
  \Phi_\eta(u) :=\phi_\eta\left(\eta^{-\frac{1}{1+\beta}} u
  \right) =
  \frac{\eta^{\frac{\beta}{1+\beta}}
    L^{*,+}\phi_\eta \left(\eta^{-\frac{1}{1+\beta}}
      u\right)}{\eta^{\frac{\beta}{1+\beta}}
    \nu\left(\eta^{-\frac{1}{1+\beta}}u \right) - i (u \cdot
    \sigma) -\mu(\eta) \vert u \vert_\eta^{-\beta}}.
\end{equation*}

We turn to the case $\alpha \le
2+\beta$. Estimate~\eqref{eq:mmt-fract} in
Hypothesis~\ref{hyp:scalinginfinite}-(ii) follows from
$\Phi_{\eta}$ begin uniformly bounded and for $\eta$ small and
$|u| \le 1$ (using
$|L^{*,+} (\phi_\eta)(v)| \lesssim \wv^{-\beta}$),
\begin{equation*}
  \left\vert \Im \Phi_{\eta}(u) \right\vert
  = \left\vert \frac{ (u \cdot \sigma)
  \left[\eta^{\frac{\beta}{1+\beta}} L^{*,+} \phi_\eta \left(
  \eta^{-\frac{1}{1+\beta}} u \right) \right]
  }{\left(\eta^{\frac{\beta}{1+\beta}} \nu \left(
  \eta^{-\frac{1}{1+\beta}} u \right) -\mu(\eta) \vert u
  \vert_\eta^{-\beta}\right)^2 + (u \cdot \sigma)^2} \right\vert
  \lesssim \frac{\vert u \cdot \sigma
  \vert}{\eta^{\frac{\beta}{1+\beta}}
  \nu\left(\eta^{-\frac{1}{1+\beta}} u\right)}
  \lesssim \vert u \vert_\eta^{1+\beta}.
\end{equation*}

When $\alpha \le \beta$, the integral moment
bound~\eqref{eq:mmt-fract-bis} in
Hypothesis~\ref{hyp:scalinginfinite}-(ii)-(b) follows from (for
small $\eta$ and large $u$ and using again
$|L^{*,+} (\phi_\eta)(v)| \lesssim \wv^{-\beta}$)
\begin{align*}
  \left| \Phi_\eta (u) \right| \lesssim \frac{1}{1+|u|^\beta|u
  \cdot \sigma|}
\end{align*}
which implies 
\begin{align*}
  \left\| \Phi_\eta \right\|_\beta ^2 \lesssim
  \int_0  ^\pi \int_1
  ^{+\infty} \frac{r^{-1-\alpha+\beta}}{1+r^{2+2\beta} \cos
  \theta^2} \dd r \dd \theta < +\infty.
\end{align*}

To prove the remaining points we use $L^* 1 = 0$ to write
\begin{align*}
  \Phi_\eta(u) - \frac{\eta^{\frac{\beta}{1+\beta}} \nu \left(
  \eta^{-\frac{1}{1+\beta}} u
  \right)}{\eta^{\frac{\beta}{1+\beta}} \nu \left(
  \eta^{-\frac{1}{1+\beta}}u \right) - i (u \cdot \sigma)
  -\mu(\eta) \vert u \vert_\eta^{-\beta}}
  &= \frac{\eta^{\frac{\beta}{1+\beta}} L^{*,+} \phi_\eta
    \left(\eta^{-\frac{1}{1+\beta}} u \right)-
    \eta^{\frac{\beta}{1+\beta}} \nu \left(
    \eta^{-\frac{1}{1+\beta}} u
    \right)}{\eta^{\frac{\beta}{1+\beta}} \nu \left(
    \eta^{-\frac{1}{1+\beta}} u \right) - i (u \cdot \sigma) -
    \mu(\eta) \vert u \vert_\eta^{-\beta}}\\
  &= \frac{\eta^{\frac{\beta}{1+\beta}}
    L^{*,+}\left(\phi_\eta-1\right)\left(\eta^{-\frac{1}{1+\beta}} u
    \right)}{\eta^{\frac{\beta}{1+\beta}} \nu \left(
    \eta^{-\frac{1}{1+\beta}}u \right) - i (u \cdot \sigma)
    -\mu(\eta) \vert u \vert_\eta^{-\beta}}.
\end{align*}
Since then 
\begin{equation*}
  \left\vert \eta^{\frac{\beta}{1+\beta}}
    L^{*,+}\left(\phi_\eta-1\right) \left(\eta^{-\frac{1}{1+\beta}}u
      \right) \right\vert \lesssim \eta^{\frac{\beta}{1+\beta}}
    \left\lfloor \eta^{-\frac{1}{1+\beta}} u \right\rceil^{-\beta} \Vert
    \phi_\eta-1 \Vert_{-\beta} \lesssim
    \sqrt{\mu(\eta)} \eta^{\frac{\beta}{1+\beta}}  \nu
    \left(\eta^{-\frac{1}{1+\beta}} u \right),
\end{equation*}
we deduce
\begin{align*}
  \left\| \frac{\Phi_\eta}{\Phi_{\eta,0}} - 1\right\|_{L^\infty(\R^d)}
  \lesssim \sqrt{\mu(\eta)}  \xrightarrow[\eta \to 0]{} 0
\end{align*}
with the simpler function 
\begin{equation*}
  \Phi_{\eta,0}(u) :=
  \frac{\eta^{\frac{\beta}{1+\beta}} \nu
    \left(\eta^{-\frac{1}{1+\beta}}
      u \right)}{\eta^{\frac{\beta}{1+\beta}}
      \nu\left(\eta^{-\frac{1}{1+\beta}}u \right) - i (u \cdot
      \sigma) -\mu(\eta) \vert u \vert_\eta^{-\beta}}.
\end{equation*}
To prove the convergence of $\Phi_\eta$ it is thus enough to
check the convergence of $\Phi_{\eta,0}$:
\begin{equation*}
  \lim_{\eta \to 0} \Phi_{\eta}(u) = \lim_{\eta \to 0} \Phi_{\eta,0}(u)
  = \frac{\nu_0}{\nu_0 - i \vert u \vert^{\beta} (u \cdot \sigma)}
  =: \Phi(u)
\end{equation*}
and in the case $\alpha=2+\beta$ we also have
\begin{equation*}
  \Omega(u) = \lim_{\lambda \to
  0, \ \lambda \not =0} \lambda^{-(1+\beta)}
  \frac{\nu_0 \vert \lambda u \vert^{\beta}
  ( \lambda u \cdot \sigma)}{\nu_0^2 + \vert
  \lambda u \vert^{2\beta} ( \lambda u \cdot \sigma)^2}
  = \nu_0^{-1} \vert u \vert^{\beta} (u \cdot \sigma),
\end{equation*}
and the corresponding diffusion coefficients are given in the
statement of Corollary~\ref{cor:scat}.

Moreover, since
\begin{equation*}
\Im \Phi_{\eta,0}(u) :=
\frac{ \vert u \vert_\eta^{\beta} \nu_\eta (u)
  \vert u \vert_\eta^{\beta} (u \cdot
  \sigma)}{\left( \vert u \vert_\eta^{\beta} \nu_\eta (u)  -
    \mu(\eta) \right)^2 + \vert u \vert_\eta^{2\beta}(u \cdot
      \sigma)^2},
\end{equation*}
and
\begin{equation*}
\Im \Phi(u) :=
  \frac{ \nu_0 \vert u \vert^{\beta}  (u \cdot
      \sigma)}{ \nu_0^2  + \vert u \vert^{2\beta}(u \cdot
      \sigma)^2},
\end{equation*}
we deduce
\begin{equation*}
  \frac{\Im \Phi (u)}{ \Im \Phi_{\eta,0} (u)} 
  = \frac{ \nu_0 }{\vert u \vert_\eta^{\beta} \nu_\eta (u)}
    \frac{\vert u \vert^{\beta}}{\vert u \vert_\eta^{\beta}}     
    \frac{\left( \vert u \vert_\eta^{\beta} \nu_\eta (u)
    - \mu(\eta) \right)^2 + \vert u \vert_\eta^{2\beta}(u \cdot
    \sigma)^2}{  \nu_0^2  + \vert u \vert^{2\beta}(u \cdot
    \sigma)^2} 
  = (1+ o(1))\frac{\vert u \vert^{\beta}}
    {\vert u \vert_\eta^{\beta}}.
\end{equation*}
Since
$\left|\Im \Phi_{\eta}-\Im \Phi_{\eta,0}\right| \leq
\sqrt{\mu(\eta)} \left|\Im \Phi_{\eta,0}\right|$,
\begin{equation*}
  \frac{\Im \Phi (u)}{ \Im \Phi_{\eta} (u)} 
  =\frac{\Im \Phi (u)}{ \Im \Phi_{\eta,0} (u)}
  \frac{ \Im \Phi_{\eta,0} (u)}{ \Im \Phi(u)}
  = (1+ o(1))\frac{\vert u \vert^{\beta}}
  {\vert u \vert_\eta^{\beta}}
\end{equation*}
which completes the proof of
Hypothesis~\ref{hyp:scalinginfinite}-(ii)-(a).

\begin{remark}
  Note that when $\alpha <0$ (assuming as always $\alpha+\beta>0$)
  the \emph{rescaled mass} is positive:
  \begin{equation*}
    \int_{\R^d} \Phi(u) |u|^{-d-\alpha} \dd u
    = \int_{\R^d} \Re \Phi(u) |u|^{-d-\alpha} \dd u
    = \int_{\R^d} \frac{\nu_0^2}{\nu_0^2 + |u|^{2\beta} (u \cdot
      \sigma)^2} \frac{1}{|u|^{d+\alpha}} \dd u >0
  \end{equation*}
  by inspection and we also prove that it is finite:
  \begin{align*}
    \int_{\R^d} \Phi(u) |u|^{-d-\alpha} \dd u
    & = \int_{r \ge 0} \left( \int_{\theta \in
      (-\frac{\pi}{2},\frac{\pi}{2})} \frac{\nu_0^2 |\sin^{d-2}\theta|}{\nu_0^2 +
      r^{2(1+\beta)} \sin ^2 \theta} \dd \theta \right) \frac{1}{r^{1+\alpha}} \dd
      r \\
    & = \int_{r=0} ^1  \dots + \int_{r \ge 1} \int_{|\theta| \in
      (\frac{\pi}{4},\frac{\pi}{2})} \dots + \int_{r \ge 1}
      \int_{|\theta| \le \frac{\pi}{4}} \dots =: I_1 + I_2 + I_3.
  \end{align*}
  Then $I_1 \lesssim \int_{r \sim 0} r^{-1-\alpha} \dd r <+\infty$
  and $I_2 \lesssim \int_{r \ge 1} r^{-3-\alpha-2\beta} \dd r
  <+\infty$ and finally
  \begin{align*}
    I_3 & \lesssim \int_{r \ge 1} \left( \int_{|z| \le
          \frac{1}{\sqrt 2}} \frac{\nu_0^2}{\nu_0^2 +
   r^{2(1+\beta)} z^2} \dd z \right) \frac{1}{r^{1+\alpha}} \dd
          r \\
        & \lesssim \int_{r \ge 1} \left( \int_{|w| \le
   \frac{r^{1+\beta}}{\sqrt 2}} \frac{\nu_0^2}{\nu_0^2 +
   w^2} \dd w \right) \frac{1}{r^{2+\alpha+\beta}} \dd r < + \infty.
  \end{align*}
  This implies the existence of the limit defining the diffusion
  coefficient in Lemma~\ref{lem:diffcoeff}. However for such
  $\alpha$, no initial data seem to allow for a fractional
  diffusive limit, see also Subsection~\ref{ss:negative}.
\end{remark}

\section{Proof of the hypotheses for kinetic Fokker-Planck equations}
\label{sec:kfp}

Let us consider $\beta=2$, $\alpha \ge 0$, $\cM$ given by
Hypothesis~\ref{hyp:functional}, and let us consider the operators
\begin{equation*}
  \mathcal{L}(f) := \nabla_v \cdot \left( \cM \nabla_v
    \left( \frac{f}{\cM} \right) \right) \quad \text{ and } \quad
  L h := \cM^{-1} \nabla_v \cdot \left( \cM \nabla_v h \right),
\end{equation*}
which are self-adjoint in, respectively, $L^2_v(\cM^{-1})$ and
$L^2_v(\cM)$.

\subsection{Proof of Hypothesis~\ref{hyp:coercivity}}

This hypothesis is:
\begin{align*}
  \int_{\R^d} \left| \nabla_v h \right|^2 \cM(v) \dd v \ge \lambda
  \left\| h - \cP h \right\|_{-2} \text{ with } \cP h :=
  \int_{\R^d} h(v') \wvp^{-2} \cM(v') \dd v'
\end{align*}
for some $\lambda >0$ (recall that
$\int \wdot^{-2} \cM =1$ as per
Hypothesis~\ref{hyp:functional}). It is a form of the so-called
\emph{Hardy-Poincaré inequality}, see for
instance~\cite{BBDGV} where references are
collected for proving it for $d \ge 3$ and $\alpha >-2$,
\cite[Corollary 1]{DolbeaultToscani} and~\cite[Appendix
A]{blanchet_asymptotics_2009} where it is proved in all
dimensions $d \ge 1$ under the condition $d+\alpha >0$ (for
instance the ``$\alpha$'' in~\cite[Corollary 1]{DolbeaultToscani}
corresponds to our ``$-(d+\alpha)$''). Note that the case when
$d \ge 3$ and $\alpha \in (-d,-2)$ would correspond
in~\cite{DolbeaultToscani,blanchet_asymptotics_2009} to
situations where the Hardy-Poincaré inequality holds without the
need of the zero-average condition. These cases are however
excluded by our assumption $\alpha \geq 0$. 

\subsection{Proof of Hypothesis~\ref{hyp:large-v}}
We perform the following computation (the case of $\tilde \chi_R$ is similar)
\begin{align*}
  \Vert  L(\chi_R)\Vert_2^2
  &= \int_{\R^d} \left| \nabla \cdot \left( \cM\nabla_v \chi_R \right)
    \right|^2 \wdot^2 \,  \frac{{\rm d}v}{\cM} \\
  & = \int_{\R^d} \left| \Delta \chi_R  + \frac{\nabla_v
    \cM}{\cM} \cdot \nabla \chi_R \right|^2 \wdot^2  \,
    \cM \dd v\\
  &= \int_{B_{2R}\backslash B_R} \left| \Delta \chi_R + \frac{\nabla_v
    \cM}{\cM} \cdot \nabla \chi_R \right|^2 \wdot^2 \, \cM \dd v \\
  &= \int_{B_{2R}\backslash B_R} \wdot^{-2} \, \cM \dd v 
    \lesssim_\chi  R^{-(2 + \alpha)} = R^{-(\beta + \alpha)}.
\end{align*}

\subsection{Proof of Hypothesis~\ref{hyp:scalinginfinite}}

The equation satisfied by $\Phi_\eta$ is
\begin{align}
  \label{eq:recall-rescaled}
  - |u|^2_\eta \Delta_u \Phi_\eta + (d+\alpha) u \cdot \nabla_u
  \Phi_\eta - i (u \cdot \sigma) |u|_\eta ^2 \Phi_\eta = \mu(\eta)
  \Phi_\eta. 
\end{align}

We now prove~\eqref{eq:mmt-fract} in
Hypothesis~\ref{hyp:scalinginfinite}, but estimate first the
non-rescaled eigenfunction.
\begin{lemma}
  \label{lem:ponctuel-phi}
  The unique solution to 
  \begin{align*}
    - L \phi_{\eta}
      - i \eta (v \cdot \sigma) \phi_{\eta}= \mu(\eta)
      \wv^{-2} \phi_{\eta} \quad \text{ with } \quad 
      \int_{\R^d} \phi_{\eta}(v) \, \wv^{-2} \, \cM(v) \dd v = 1
  \end{align*}
  satisfies for any $R \ge 1$
  \begin{align*}
    \| \phi_\eta \|_{L^\infty(B(0,R))} \lesssim_R 1 \quad \text{ and }
    \quad \| \Im \phi_\eta \|_{L^\infty(B(0,R))}
    \lesssim_R \max(\eta,\mu(\eta))
  \end{align*}
  with constants depending only on $R$ but uniform in
  $\eta \to 0$.
\end{lemma}

\begin{proof}[Proof of Lemma~\ref{lem:ponctuel-phi}]
  As for the scattering equation, Hypothesis~\ref{hyp:coercivity}
  implies, for $\eta$ small enough
  \begin{align*}
    \lambda \left\| \phi_\eta - 1 \right\|_{-2}^2 \le \mu(\eta) \left\|
    \phi_\eta \right\|_{-2} \quad \Rightarrow \quad 
    \left\| \phi_\eta \right\|_{-2}^2 \le \frac{\lambda}{\lambda -
    \mu(\eta)} \lesssim 1.
  \end{align*}
  The elliptic regularity of the operator
  $L = \Delta - (d+\alpha) \langle v \rangle^{-2} v \cdot
  \nabla_v$, with uniform ellipticity constant, then classically
  implies that
  \begin{align*}
    \left\| \phi_\eta \right\|_{L^\infty(B(0,R))} \lesssim_R 1.
  \end{align*}

  Since $\cP \phi_\eta =1$ in the decomposition
  $\phi_\eta = \cP \phi_\eta + \cP^\bot \phi_\eta$, one deduces
  \begin{align*}
    \left\| \Im \phi_\eta \right\|_{-2} \le \big\| \cP^\bot \phi_\eta 
    \big\|_{-2} \lesssim \mu(\eta)
  \end{align*}
  and the imaginary part satisfies the equation
  \begin{align*}
    - L (\Im \phi_\eta) -\mu(\eta) \wv^{-2} \Im \phi_\eta =
    \eta (v \cdot \sigma) \Re \phi_\eta.
  \end{align*}
  Therefore the elliptic regularity combined with the integral
  bound on $\Im \phi_\eta$ and the bound
  \begin{equation*}
    \| \eta (v \cdot \sigma) \Re \phi_\eta \|_{L^2(B(0,R))}
    \lesssim \eta
  \end{equation*}
  on the right hand side implies that
  \begin{align*}
    \left\| \Im \phi_\eta \right\|_{L^\infty(B(0,R))} \lesssim_R
    \max(\eta,\mu(\eta))
  \end{align*}
  which concludes the proof.
\end{proof}

The following lemma proves~\eqref{eq:mmt-fract}.

\begin{lemma}
  \label{lem:ponctuel-Phi}
  There are $\eta_1 \in (0,\eta_0)$ small enough and
  $\delta \in (0,\min(4-\alpha,3))$ and $C$ large enough
  so that
  \begin{equation*}
    \forall \, \eta \in (0,\eta_1),
    \ \forall \,  u \in \R^d, \qquad
    \vert \Phi_\eta(u) \vert \lesssim
    \vert u \vert_\eta^{C \mu(\eta)} \quad \text{and} \quad
    \vert \Im \Phi_\eta (u) \vert \lesssim \vert u
    \vert_\eta^{\min(2+\alpha,3)-\delta}.
  \end{equation*}
\end{lemma}

\begin{proof}[Proof of Lemma~\ref{lem:ponctuel-Phi}]
  Multiply~\eqref{eq:recall-rescaled} by
  $\frac{\overline{\Phi_\eta}}{\vert \Phi_\eta \vert}$ and take
  the real part:
  \begin{equation*}
    - |u|_\eta ^2 \Re \left( \frac{\overline{\Phi_\eta}}{\vert
        \Phi_\eta \vert} \Delta_u \Phi_\eta \right)
    + (d+\alpha)u \cdot
    \Re\left(\frac{\overline{\Phi_\eta}}{\vert \Phi_\eta \vert}
      \nabla_u \Phi_\eta \right) = \mu(\eta) \vert \Phi_\eta \vert.
  \end{equation*}
  Since 
  \begin{equation*}
    \nabla_u \vert \Phi_\eta \vert =
    \Re\left(\frac{\overline{\Phi_\eta}}{\vert \Phi_\eta \vert}
      \nabla_u \Phi_\eta \right), \qquad
    \Delta_u \vert \Phi_\eta \vert \geq \Re
    \left( \frac{\overline{\Phi_\eta}}{\vert \Phi_\eta \vert}
      \Delta_u \Phi_\eta \right),
  \end{equation*}
  one gets
  \begin{align*}
    - |u|_\eta^2 \Delta_u \vert \Phi_\eta \vert + (d+\alpha)u \cdot
    \nabla_u \vert \Phi_\eta \vert - \mu(\eta) \vert \Phi_\eta \vert
    \leq 0.
  \end{align*}
  Then observe that the real function
  $F(u) = |u|_\eta^{C\mu(\eta)}$ satisfies for
  $|u| \ge A \eta^{\frac{1}{3}}$ with $A$ large:
  \begin{align*}
    & - |u|_\eta^2 \Delta_u F + (d+\alpha)u \cdot \nabla_u F
      - \mu(\eta) F \\
    & \quad \ge \mu(\eta) F \left[ - C d \frac{|u|_\eta^2}{|u|^2}
      - C (C
      \mu(\eta) -2) \frac{|u|_\eta^2}{|u|^2} 
      + C (d+\alpha) -1 \right] \\
    & \quad \ge \mu(\eta) F \left[ - C d (1+\varepsilon) - C (C
      \mu(\eta) -2) (1+\varepsilon) + C (d+\alpha) -1 \right] \\
    & \quad \ge \mu(\eta) F \left[ C (2+\alpha) - 1 - \varepsilon C (d-2) -
      C^2\mu(\eta) (1+\varepsilon) \right] 
 \end{align*}
 where we have used that
 $\frac{|u|_\eta^2}{|u|^2} \le 1 +\varepsilon$ with $\varepsilon$
 small for $|u| \ge A \eta^{\frac13}$ when $A$ large enough. The
 right hand side is thus positive for $C$ large enough and
 $\varepsilon$ and $\eta$ small enough, since $2+\alpha >0$:
 \begin{equation*}
   \forall \, |u| \ge A \eta^{\frac13}, \quad
   - |u|_\eta^2 \Delta_u F + (d+\alpha)u \cdot \nabla_u F
   - \mu(\eta) F
   \ge 0
 \end{equation*}
 i.e. $F$ is a super-solution in this region. Moreover, Lemma~\ref{lem:ponctuel-phi} shows that
 \begin{equation*}
   \sup_{|u|\leq A \eta^{\frac13}} \left| \Phi_\eta (u) \right| \le
   \left\| \phi_\eta \right\|_{L^\infty(B(0,A))} \lesssim_A 1
 \end{equation*}
 and we can therefore compare $\Phi_\eta$ and $F$ on the ball
 $|u| \leq A \eta^{\frac13}$ with a bound uniform in $\eta$. The
 maximum principle thus implies that
 $|\Phi_\eta| \lesssim |u|_\eta ^{C \mu(\eta)}$ for all
 $|u| \ge A \eta^{\frac13}$ with a bound uniform in
 $\eta$. Finally, since $\eta^{C \mu(\eta)} \sim 1$ as
 $\eta \to 0$, this bound extends to any $u\in \R^d$ up to
 enlarging the comparison constant (independently of
 $\eta \to 0$).

 Take then the imaginary part of~\eqref{eq:recall-rescaled}
 \begin{equation*}
   - |u|^2_\eta \Delta_u \Im \Phi_\eta + (d+\alpha) u \cdot \nabla_u
   \Im \Phi_\eta  - \mu(\eta)
   \Im \Phi_\eta = (u \cdot \sigma) |u|_\eta ^2 \Re \Phi_\eta,
 \end{equation*}
 multiply by $\frac{\Im \Phi_\eta}{|\Im \Phi_\eta|}$ and use the estimate $|\Phi_\eta(u)| \lesssim |u|^{C \mu(\eta)}$ when $|u| \ge A \eta^{\frac13}$ to get for $|u| \ge A \eta^{\frac13}$
 \begin{equation*}
   - |u|_\eta ^2 \Delta_u \vert \Im \Phi_\eta
   \vert + (d+\alpha) u \cdot \nabla_u \vert \Im \Phi_\eta \vert -
   \mu(\eta) \vert \Im \Phi_\eta \vert \leq |u|_\eta ^{3+C\mu(\eta)}.
 \end{equation*}
 Define $G(u) := |u|_\eta ^{\mathtt{e}}$ with
 $\mathtt{e}:=2+\min(\alpha,1)-\delta$ and compute for
 $|u| \in [A \eta^{\frac13},1]$:
 \begin{equation*}
   - |u|_\eta ^2 \Delta_u G + (d+\alpha) u \cdot \nabla_u G -
   \mu(\eta) G 
   \quad \ge G \left[ \mathtt{e} \delta - \mu(\eta)
     - \mathcal{O}\left(A^{-2}\right) \right] \gtrsim G 
   \gtrsim |u|_\eta ^{3+C\mu(\eta)} 
 \end{equation*}
 for $A$ large enough and $\eta$ small enough. The maximum principle then shows that  $|\Im \Phi_\eta| \lesssim |u|_\eta ^{\mathtt{e}}$ on  $|u| \in [A \eta^{\frac13},1]$ by comparing $\Im \Phi_\eta$ and $G$ on $|u|=A \eta^{\frac13}$ thanks to the second inequality in Lemma~\ref{lem:ponctuel-phi}. Again the bound extends to any $|u| \le A \eta^{\frac13}$ using the second inequality in Lemma~\ref{lem:ponctuel-phi}, since  $\max(\eta,\mu(\eta)) \lesssim \eta^{\frac{\mathtt{e}}{3}}$ uniformly as $\eta \to 0$ (examining separately the cases $\alpha \in [0,1]$ and $\alpha \in (1,4)$).
\end{proof}

The next lemma allows us to prove the integral moment estimate~\eqref{eq:mmt-fract-bis} in Hypothesis~\ref{hyp:scalinginfinite}-(ii)-(b). 
\begin{lemma}
  \label{lem:gain-mmt-rescaled}
  There is $\mathtt{g}>0$ such that for any $q \ge -2$ and
  $G, H \in L^2(\wu^{q-d-\alpha})$ such that
  \begin{align}
    \label{eq:Phi-F}
    - |u|_\eta ^2 \Delta_u H +(d+\alpha) u \cdot \nabla_u
    H - i (u \cdot \sigma)
    |u|_\eta ^2 H = G,
  \end{align}
  the following gain of decay at infinity holds
  \begin{align*}
    \int_{\R^d} |H(u)|^2 \wu^{q+\mathtt{g}-d-\alpha} \dd u 
    \lesssim_{q,\zeta} \int_{\R^d} |G(u)|^2 \wu^{q-d-\alpha} \dd
    u + \int_{\R^d} |H(u)|^2 \wu^{q-d-\alpha} \dd u.
  \end{align*}
\end{lemma}

\begin{proof}[Proof of Lemma~\ref{lem:gain-mmt-rescaled}]
  Consider a real-valued smooth function $\chi_0(u)$ that is zero
  on $|u| \le \frac12$ and equal to $1$ on $|u| \ge 1$, and
  integrate~\eqref{eq:Phi-F} against
  $\overline{H} \chi_0^2 |u|_\eta^{q-d-\alpha}$ and take the real
  part:
  \begin{align*}
    & \int_{|u| \ge 1}  |u|^2_\eta \left|
    \nabla H(u) \right|^2 |u|_\eta ^{q-d-\alpha} \dd u \\
    & \lesssim \int_{\R^d} \left( \chi_0^2  \left| H
      \right|^2 +  \left| \Delta \left( \chi_0^2 \right) \right|
      |u|_\eta ^2 \left| H \right|^2 + \left| \nabla \left(
      \chi_0 \right) \right| |u|_\eta |H|^2 + \chi_0^2\left| G
      \right|^2 \right)  |u|_\eta ^{q-d-\alpha} \dd u \\
    & \lesssim \int_{|u| \ge \frac12}  \left( \left| H
      \right|^2 +  \left| G \right|^2 \right) |u|_\eta
      ^{q-d-\alpha} \dd u. 
  \end{align*}
  Integrate then ~\eqref{eq:Phi-F} against
  $\overline{H} (u \cdot \sigma) \chi_1^2
  |u|_\eta^{q-d-\alpha-1}$ where $\chi_1$ is a real-valued smooth
  function that is zero on $|u| \le 1$ and equal to $1$ on
  $|u| \ge 2$, and take the imaginary part:
  \begin{align*}
    & \int_{\R^d} (u \cdot \sigma)^2 \chi_1(u)^2 \left|
      H(u) \right|^2 |u|_\eta ^{1+q-d-\alpha} \dd u \\
    & \lesssim
      \int_{|u| \ge 1}  |u|^2_\eta \left|
      \nabla H(u) \right|^2 |u|_\eta ^{q-d-\alpha} \dd u
      + \int_{|u| \ge 1} \left( \left| H
      \right|^2 +  \left| G \right|^2 \right) |u|_\eta ^{q-d-\alpha} \dd
      u \\
    & \lesssim \int_{|u| \ge \frac12} \left( \left| H
      \right|^2 +  \left| G \right|^2 \right) |u|_\eta ^{q-d-\alpha} \dd
      u 
  \end{align*}
  where we have used the real part estimate in the last line. This yields
  \begin{align*}
    \int_{\R^d} (u \cdot \sigma)^2|u| \left|
    H(u) \right|^2 \wu ^{q-d-\alpha} \dd u
    \lesssim
    \int_{\R^d} \left( \left| H
    \right|^2 +  \left| G \right|^2 \right) \wu^{q-d-\alpha} \dd
    u. 
  \end{align*}
  This first estimate improves the decay at infinity outside a
  cone around $u \bot \sigma$. We now use the ellipticity of the
  equation to control this region. The operator is written
  $\mathsf L_\eta = - |u|_\eta^{d+\alpha} \nabla_u \cdot [
  |u|_\eta ^{-d-\alpha} \nabla_u ]$ and we deduce by simple
  commutator estimates that
  \begin{align*}
    \int_{|u| \ge 2} \left| \nabla_u \left( 
    H(u) |u|^{\frac{q-d-\alpha+2}{2}} \right) \right|^2 \dd u
    \lesssim \int_{|u| \ge 1} \left( \left| H
    \right|^2 +  \left| G \right|^2 \right) |u|_\eta ^{q-d-\alpha} \dd
    u. 
  \end{align*}

  Consider first the case $d>2$. The Caffarelli-Kohn-Nirenberg
  inequality yields
  \begin{equation*}
    \left\| H(\cdot) | \cdot |^{\frac{q-d-\alpha+2}{2}} 1_{|\cdot| \ge 2}
    \right\|^2 _{L^{\frac{2d}{d-2}}}
    \lesssim \int_{|u| \ge 1} \left( \left| H
      \right|^2 +  \left| G \right|^2 \right)
    |u|_\eta ^{q-d-\alpha} \dd u.
  \end{equation*}
  Consider the cone
  $\mathcal C := \{ | \frac{u}{\vert u \vert} \cdot \sigma | \le
  \wu^{-\frac{\delta}{2}} , \vert u \vert \geq 2\}$ for some
  $\delta>0$, and a gain of weight $\wu^{\mathtt{g}}$ for some
  $\mathtt{g} >0$ to be chosen later. The Hölder inequality then
  yields
  \begin{equation*}
    \int_{\mathcal C} \left| H(u) \right|^2
    |u|^{q-d-\alpha+\mathtt{g}} \dd u \le
    \left\| H(u) |u|^{\frac{q-d-\alpha+2}{2}}
      1_{|u| \ge 2} \right\|^2 _{L^{\frac{2d}{d-2}}(\R^d)}
    \left( \int_{\mathcal C} |u|^{\frac{d(\mathtt{g}-2)}{2}} \dd u
    \right)^{\frac{2}{d}}.
  \end{equation*}
  The extra volume integral may be estimated using spherical coordinates as
  \begin{equation*}
    \int_{\mathcal C} |u|^{\frac{d(\mathtt{g}-2)}{2}} \dd u
    \lesssim \int_1^{+\infty} r^{d-1} \left\langle r
    \right\rangle^{\frac{d(\mathtt{g}-2)}{2}} \left\langle r
    \right\rangle^{-\frac{\delta}{2}} \dd r \lesssim \int_1^{+\infty}
    r^{-1+\frac{d\mathtt{g}}{2}-\frac{\delta}{2}} \,\dd r
  \end{equation*}
  which is finite as soon as $\mathtt{g} < \frac{\delta}{d}$ (which
  defines and restricts $\delta$). Outside the cone we use the
  first estimate:
  \begin{align*}
    \int_{\mathcal C^c \cap \{ |\cdot| \ge 2\}} \left| H(u) \right|^2
    |u|^{q-d-\alpha+\mathtt{g}} \dd u 
    \lesssim \int_{\R^d} \left| H(u) \right|^2   (u \cdot
    \sigma)^2
    |u|^{-2+\delta+q-d-\alpha+\mathtt{g}}  \dd u 
  \end{align*}
  which is controlled as soon as $\mathtt{g} \le 3-\delta$. The
  constraints are compatible for $\mathtt{g} \in (0,\frac{3}{d+1})$.

  In the case $d=1$, the gain of decay is immediate from the
  first estimate alone. In the case $d=2$, we follow a similar
  argument but replace the Caffarelli-Kohn-Nirenberg inequality
  with the Onofri inequality:
  \begin{equation*}
    \left\| H(u) |u|^{\frac{q-d-\alpha+2}{2}} 1_{|u| \ge 2}
    \right\|^2 _{L^{2p}}
    \lesssim_p  \int_{|u| \ge 1} \left( \left| H
      \right|^2 +  \left| G \right|^2 \right)
    |u|_\eta ^{q-d-\alpha} \dd u
  \end{equation*}
  for any $p<\infty$. The Hölder inequality then gives
  \begin{equation*}
    \int_{\mathcal C} \left| H(u) \right|^2
    |u|^{q-d-\alpha+\mathtt{g}} \dd u \le
    \left\| H(u) |u|^{\frac{q-d-\alpha+2}{2}}
      1_{|u| \ge 2} \right\|^2 _{L^{2p}(\R^d)}
    \left( \int_{\mathcal C} |u|^{q(\mathtt{g}-2)} \dd u
    \right)^{\frac{1}{q}}
  \end{equation*}
  where $q=\frac{1}{1-1/p}$ is the exponent conjugate to $p$. The
  conclusion follows as before by taking $p$ large enough.
\end{proof}

Let us deduce the moment bound~\eqref{eq:mmt-fract-bis} from
Lemma~\ref{lem:gain-mmt-rescaled}. Observe first that the
pointwise bound $|\Phi_\eta(u)| \lesssim |u|^{C \mu(\eta)}$
proved in Lemma~\ref{lem:ponctuel-Phi} implies that
\begin{align*}
  \int_{|u| \ge 1} |\Phi_\eta(u)|^2 |u|^{q-d-\alpha} \dd u < +\infty
\end{align*}
for $q=-2$ and $\eta$ small enough with bound uniform in $\eta$,
since $\alpha +2 >0$. We then repeatedly apply
Lemma~\ref{lem:gain-mmt-rescaled} with $H=\Phi_\eta$ and
$G = \mu(\eta) \Phi_\eta$ to obtain that $\Phi_\eta$ decays
faster than any polynomial at infinity, with constants uniform in
$\eta$ (note that $\mathtt{g}$ is independent of $q$ in the
lemma).

To prove the asymptotic convergence we consider the solution
$\Phi : \R^d \to \C$ to
\begin{equation}
  \label{eq:pb-resc}
  - |u|^2 \Delta \Phi + (d+\alpha) u \cdot \nabla \Phi + i (u
  \cdot \sigma) \Phi =0, \quad \Phi(0) = 1, \quad \Phi \in C^0,
  \quad \Phi \in L^2(\langle u \rangle^\infty).
\end{equation}
This solution exists by weak limit $\Phi_\eta \to \Phi$. The
regularity of $\Phi$ follows from the equation outside of zero,
and the fact that $\Phi$ goes to $1$ as $u$ goes to $0$ comes
from the fact that the normalisation of $\phi_\eta$ yields
$\int_{\R^d} \vert \Phi(u) -1 \vert^2 \vert u \vert^{-d-2-\alpha}
\dd u < \infty$ (a different conclusion near zero would make this
integral infinite). The solution $\Phi$ is unique since
integrating the equation for the difference of two solutions
$\Phi:=\Phi_1-\Phi_2$ against $\overline{\Phi} |u|^{-d}$ yields
$\int_{\R^d} |\nabla \Phi|^2 |u|^{2-d} \dd u =0$ with
all boundary terms near zero behaving like $|\Phi|^2(u)$,
$u \sim 0$. Note that this solution can computed semi-explicitly
as an entire series
$\Phi(u) = \sum_{k,l \ge 0} a_{k,l} |u|^{2k} (u \cdot \sigma)^l$
by solving the two-parameters induction (see the discussion later
and Figure~\ref{fig:Recurrence}).

With the solution $\Phi$ above at hand, the convergence $\Phi_\eta \to \Phi$ in $L^2_{\text{\tiny loc}}(\R^d \backslash{\{0\}})$ follows easily from the bounds established and the convergence of the coefficients of the equation satisfied by $\Phi_\eta$: one can prove that $\eta \mapsto \Phi_\eta$ is Cauchy in $L^2$ on any such compact set as $\eta \to 0$, and such convergence has a polynomial rate and is uniform on any compact set in $\R^d \setminus \{0\}$.

We now prove Hypothesis~\ref{hyp:scalinginfinite}-(ii)-(a). The
equation for $W_\eta := \Phi_\eta - \Phi$ is
\begin{align*}
- |u|^2_\eta \Delta_u  W_\eta + (d+\alpha) u \cdot
    \nabla_u  W_\eta - i(u\cdot \sigma)|u|^2_\eta W_\eta
    - \mu(\eta) \ W_\eta 
  &=  \eta^{\frac23} (d+\alpha) \frac{u}{|u|^2} \cdot \nabla_u 
    \Phi + \mu(\eta) \Phi.
\end{align*}
We derive a bound on $ \nabla_u \Phi$. For this, differentiate
the limit equation for $\Phi$,
\begin{align*}
  - |u|^2 \Delta \left( \nabla \Phi \right)
  + (d+ \alpha) u \cdot \nabla\left( \nabla \Phi \right)
  + (d+\alpha) \nabla \Phi - 2 ( \Delta \Phi ) u = i (u \cdot
  \sigma) |u|^2 \nabla \Phi + i \nabla \left( ( u \cdot \sigma)
  |u|^2 \right)  \Phi.
\end{align*}
Test against
$\frac{\overline{\nabla \Phi}}{\vert \nabla \Phi \vert}$, use
\begin{align*}
  & \nabla \left\vert \nabla \Phi\right\vert =
    \Re\left(\frac{\overline{\nabla \Phi}}{\vert \nabla \Phi
    \vert} \nabla \nabla \Phi \right), \qquad 
    \Re \left( \frac{\overline{\nabla \Phi}}{\vert \nabla \Phi \vert}
    \Delta \left( \nabla \Phi \right) \right) \leq
    \Delta \left(\left\vert \nabla \Phi\right\vert\right),\\
  & u \cdot \Re \left( \Delta \Phi  \frac{\overline{\nabla \Phi}}{\vert
    \nabla \Phi \vert} \right) =
    u \cdot \nabla \left( \vert \nabla \Phi \vert \right),
    \qquad \vert \nabla \left( ( u \cdot \sigma) |u|^2 \right)
    \vert \leq 3 |u|^2,
\end{align*}
and take the real part to get
\begin{equation*}
  - |u|^2 \Delta \left( \left\vert \nabla \Phi\right\vert \right)
  + (d+ \alpha-2) u \cdot \nabla\left( \left\vert \nabla \Phi
    \right\vert \right)
  + (d+\alpha) \left\vert \nabla \Phi \right\vert
  \leq 3 |u|^2 \Vert \Phi \Vert_\infty.
\end{equation*}
The comparison principle (with the same sort of computations as
for Lemma~\ref{lem:ponctuel-Phi}) then yields that
$\left\vert \nabla \Phi\right\vert \lesssim |u|^2$ since we know
from the equation that $\nabla \Phi (0) = 0$. As a consequence,
\begin{equation*}
    - |u|^2_\eta \Delta_u  \vert W_\eta \vert + (d+\alpha) u \cdot
    \nabla_u  \vert W_\eta \vert 
    - \mu(\eta) \vert W_\eta \vert \lesssim \eta^{\frac23}
    \vert u \vert + \mu(\eta)
    \lesssim \eta^{\frac23}  \vert u \vert_\eta,
\end{equation*}
since $\Phi$ is uniformly bounded. From this, one deduces
$\vert W_\eta \vert \lesssim \eta^{\frac23}\vert u \vert_\eta$
(since $W_\eta$ is of order $\eta$ near zero). This implies the
hypothesis since then, recalling $\alpha = 4$ and $\beta = 2$,
\begin{align*}
  &\left\vert 
    \int_{1 \geq \vert u \vert \geq \eta^{\frac{1}{3}}} (u \cdot
    \sigma)  \Big[ \Im\Phi_\eta(u) - \Im\Phi(u) \Big]
    |u|_\eta ^{-d-4} \dd
    u \right\vert \\
  &\leq \int_{1 \geq \vert u \vert \geq \eta^{\frac{1}{3}}}
    \eta^{\frac23}|u|_\eta ^{-d-2} \dd u
    = \int_{\eta^{-\frac{1}{3}} \geq \vert u \vert \geq 1}
    \frac{\eta^{\frac{d+2}{3}}\dd u}{\eta^{\frac{d+2}{3}} \vert 1
    + |w|^2\vert^{\frac{d+2}{2}} } \lesssim 1.
\end{align*}
We now prove the second part of
Hypothesis~\ref{hyp:scalinginfinite}-(ii)-(a), regarding the
existence of a scaling limit of $\Phi$. The rescaled limit $\Phi$
satisfies $\Phi(0)=1$ and the Schrödinger equation
\begin{equation*}
  -|u|^2 \Delta_u \Phi + (d+\alpha)u \cdot  \nabla_u \Phi
  - i (u \cdot \sigma) |u|^2 \Phi =0.
\end{equation*}
Therefore, $\Omega_\lambda(u):=\frac{\Im \Phi(\lambda
  u)}{\lambda^3}$ satisfies (using that $\Phi$ is continuous at
$\Phi(0)=1$)
\begin{equation*}
  -|u|^2 \Delta_u \Omega_\lambda(u) + (d+\alpha)u \cdot  \nabla_u
  \Omega_\lambda(u) = (u \cdot \sigma) |u|^2 \Re \Phi(\lambda u)
  \xrightarrow[\lambda \to 0]{}  (u \cdot \sigma) |u|^2.
\end{equation*}
Since the limit equation
$-|u|^2 \Delta \Omega + (d+\alpha)u \cdot \nabla \Omega = (u
\cdot \sigma) |u|^2$ has unique continuous solution satisfying
$\Omega(0)=0$ given by
$\Omega(u)=\frac{(u \cdot \sigma)|u|^2}{d+8}$, elliptic estimates
imply $\Omega_\lambda \to \Omega$ in $L^1(\mathbb{S}^{d-1})$.

\begin{remark}
  Note that interestingly, and in contrast with scattering
  operators, the rescaled mass vanishes for the Fokker-Planck
  operator when $\alpha \in (-\beta,0)=(-2,0)$:
  \begin{equation*}
    \int_{\R^d} \Phi(u) |u|^{-d-\alpha} \dd u
    = \int_{\R^d} \Re \Phi(u) |u|^{-d-\alpha} \dd u = 0
  \end{equation*}
  Indeed, it follows from constructing
  $\mathfrak G : \R^d \to \C$ that is $C^2$ so that
  $\mathfrak G \Phi$ and $ \Phi\nabla \mathfrak G$ decay faster
  than any polynomials at infinity, that $\mathfrak G(0)$ and
  $\nabla \mathfrak G(0)=0$, and
  \begin{equation}
    \label{eq:induction-G}
    - |u|^{d+\alpha} \nabla \cdot \left[ |u|^{-(d+\alpha)} \nabla
      \left( |u|^2 \mathfrak G \right) \right] - i (u \cdot \sigma)
    |u|^2 \mathfrak G = 1.
  \end{equation}
  Then integration by parts are justifed near zero and infinity
  and imply
  \begin{align*}
    \int_{\R^d} \Phi(u)|u|^{-d-\alpha} \dd u
    & = \int_{\R^d}
      \Phi(u)\left\{ - |u|^{d+\alpha} \nabla \cdot \left[
      |u|^{-(d+\alpha)} \nabla \left( |u|^2 \mathfrak G \right)
      \right] - i (u \cdot \sigma) |u|^2 \mathfrak G
      \right\} |u|^{-d-\alpha} \dd u \\
    & = \int_{\R^d} \left\{ -|u|^2 \Delta_u \Phi + (d+\alpha)u
      \cdot  \nabla_u \Phi - i (u \cdot \sigma) |u|^2 \Phi \right\}
      \mathfrak G(u) |u|^{-d-\alpha} \dd u =0.
  \end{align*}
  To construct such $\mathfrak G$, plug the ansatz
  $\mathfrak G(u) = \sum_{k,l \ge 0} a_{k,l} B_{k,l}(u)$ with
  $B_{k,l}(u) := |u|^{2k} (u \cdot \sigma)^l$ and $a_{0,l}=0$
  for all $l \in \mathbb N$, $a_{k \neq 1,0} =0$ and $a_{1,0}=1$,
  into~\eqref{eq:induction-G} to get the sufficient condition
  \begin{equation*}
    \left[ 2(k+1) (\alpha - 2l - 2k) + (d+\alpha) l \right] a_{k,l}
    = i a_{k-1,l-1} + l (l-1) a_{k-1,l+2}.
  \end{equation*}
  This relation forms a triangle in the $(l,k)$ quadrant with
  horizontal base, and given the conditions above one can solve
  by induction on $l$ for each $k$ and then on $k$, see
  Figure~\ref{fig:Recurrence}. Moreover one can prove by
  induction (only large $k,l$ matter) that
  $\vert a_{k.l} \vert \lesssim \frac{1}{k! l!}$ so that
  $\mathfrak G$ is well-defined.
  \begin{figure}[!ht]
    \includegraphics[scale=1.3]{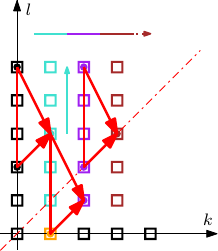}
    \caption{The black squares are zero values given to start the
      induction. The orange square has value $1$. The black
      values for $k=0$ allow to compute the turquoise blue
      squares (that are all zero then). The turquoise blue
      squares together with the (non-zero) orange square give the
      values held by the purple squares, etc. On each column,
      values are computed from bottom to top, sliding the
      triangular red planer on the previous column.}
    \label{fig:Recurrence}
  \end{figure}
  Finally $\mathfrak H(u) := |u|^2 \mathfrak G(u)$ satisfies,
  arguing as before,
  $- |u|^{d+\alpha} \nabla \cdot [ |u|^{-(d+\alpha)} \nabla ( |
  \mathfrak H | ) ] \le 1$. Then $|\cdot|^{2+\alpha+0}$ provides
  a super-solution at infinity, and $\mathfrak G$ bounded for $u$
  near zero, so this proves that $\mathfrak G \Phi$ decays faster
  than any polynomial (see the discussion following Lemma
  \ref{lem:gain-mmt-rescaled} giving the decay of
  $\Phi$). Similar arguments can be performed on
  $|\nabla\mathfrak H|$ and $(\nabla \mathfrak G)\Phi$ by
  differentiating \eqref{eq:induction-G}.
\end{remark}

\section{Proof of the hypotheses for kinetic
  L\'evy-Fokker-Planck equations}
\label{sec:lkfp}

Consider $s \in (\frac12,1)$, $\alpha >s$ and $\cM$ is given by
Hypothesis~\ref{hyp:functional}, and the operator
\begin{equation*}
  \mathcal{L}(f) :=
  \Delta_v ^s f + \nabla_v\cdot\left(U\,f\right).
\end{equation*}
The fractional Laplacian is defined as in~\eqref{eq:fract-Lap}
but we use the equivalent definition (see for example
\cite{daoud_fractional_2022} and the references therein)
\begin{equation*}
  \Delta_v ^s f(v) := - C_{d,s}
  \int_{\R^d} \frac{\big[ f(v)-f(v') \big]}{\vert v - v'
    \vert^{d+2s}}\dd v'
  \quad \text{ with } \quad C_{d,s} := \frac{4^s
    \Gamma\left(\frac{d}2 +s \right)}{\pi^{\frac{d}2} |\Gamma(-s)|}. 
\end{equation*}
The drift force $U$ solves
$\Delta_v ^s \cM+ \nabla_v\cdot\left(U\,\cM\right) =
0$. \cite[\textcolor{red}{Proposition 7}]{BDLS} shows that the
explicit radial solution $U$ to this equation satisfies
$U(v) = \mathbb{U}(v) \wv^{-\beta} v$ with $\beta:=2s-\alpha$ and
$\mathbb{U}$ a uniformly positive function bounded from
above. The operator $L$ is
\begin{align*}
  L h := \cM^{-1} \Delta_v ^s \left(\cM h\right)
    + \cM^{-1}\nabla_v\cdot\left( U \,\cM h\right)
  = \cM^{-1}\Big[ \Delta_v ^s\left(\cM h \right) -
    \left( \Delta_v^s \cM \right) h \Big] + U \cdot
   \nabla_v h.
\end{align*}

\subsection{Proof of Hypothesis~\ref{hyp:coercivity}}

This hypothesis is implied by the fractional Hardy-Poincaré
inequalities proved in \cite{BDLS} and earlier
in~\cite{wang-2015}:
\begin{proposition}[\cite{BDLS,wang-2015}]
  \label{prop:coer-bdls}
  Let $d\ge1$, $s\in(0,1)$, $\alpha>s$ and
  $\beta:=2s-\alpha$. Then there is $\lambda >0$ (depending on
  $s$) such that
  \begin{align*}
    - \Re \, \big\langle L h,h 
    \big\rangle \ge \lambda \Vert h - \mathcal{P}h \Vert_{-\beta}^2.
  \end{align*}
\end{proposition}
\begin{proof}
  Compute
  \begin{align*}
    - \Re \, \big\langle L h,h \big\rangle 
    &= - \Re \, \int_{\R^d} \Big[ \Delta_v^s(\cM h)
      + \nabla_v\cdot\left( U \,\cM h\right) \Big]
      \overline{h} \dd v'\\
    &= - \Re \, \int_{\R^d} \left(\Delta_v^s \overline{h}\right)
      h \, \cM \dd v + \Re \, \int_{\R^d} \frac12 U \cdot \nabla_v
      \left(\vert h \vert^2\right)\cM \dd v\\
    &= - \Re \, \int_{\R^d} \left(\Delta_v^s\overline{h}\right)
      h \, \cM \dd v - \Re \, \int_{\R^d} \frac12
      \nabla_v\cdot\left(U\,\cM \right) \vert h \vert^2  \dd v\\
    &= - \Re \, \int_{\R^d} \left(\Delta_v^s\overline{h}\right)
      h \, \cM \dd v + \Re \, \int_{\R^d} \frac12
      \left(\Delta_v^s \cM\right) \vert h \vert^2  \dd v\\
    &= \frac{C_{d,s}}2  \int_{\R^d \times \R^d}
      \frac{|h-h'|^2}{|v-v'|^{d+2s}} \,\cM \dd v \dd v' 
  \end{align*}
  and thus
  \begin{equation*}
    - \Re \, \big\langle L h,h 
    \big\rangle =\frac{C_{d,s}}{4} \int_{\R^d}
    \frac{|h-h'|^2}{|v-v'|^{d+2s}} \,(\cM+\cM')\dd v \dd v'.
  \end{equation*}
  Note that there is $\kappa >0$ such that
  \begin{equation*}
    \forall \, (v,v') \in \R^d \times \R^d,\quad
    \wv^{-\beta}\,\wvp^{-\beta}\,\cM\,\cM' \leq
    \kappa\,\frac{\cM+\cM'}{|v-v'|^{d+2s}},
  \end{equation*}
  by matching the asymptotics at large $v$ and $v'$. Hence we get
  that
  \begin{equation*}
    - \Re \, \big\langle L h,h 
    \big\rangle 
    \gtrsim \int_{\R^d} |h-h'|^2\wv^{-\beta}\,\wvp^{-\beta} \, \cM
    \,\cM' \dd v \dd v' \gtrsim \Vert h
    - \mathcal{P}h \Vert_{-\beta}^2,
  \end{equation*}    
  where we used in the last line the classical coercivity for
  scattering operators discussed above.
\end{proof}

Note that in the coercivity inequality of Proposition~\ref{prop:coer-bdls}, 
$\lambda(s) \to 0$ as $s \to 1$ since $C_{d,s} \to 0$ as
$s \to 1$. This explains why the coercivity weight $\beta=2$ of
the Fokker-Planck operator differs from the coercivity weight
$\beta=2-\alpha$ of the Lévy-Fokker-Planck operator when
$s \to 1$. In fact when $\alpha =2s$ and $s \to 1$ the correct
formal limit is the Fokker-Planck operator with Gaussian
equilibrium, in view of the general theory of Lévy processes, for
which $\beta=0$ is indeed the limit of $\beta=2-\alpha=2-2s$ as
$s \to 1$.

\subsection{Proof of Hypothesis~\ref{hyp:large-v}}

We estimate (the case of $\tilde \chi_R$ is similar)
\begin{equation*}
  \Vert  L(\chi_R)\Vert_\beta
  = \left\| \cM^{-1} \Big[ \Delta_v^s\left(\cM\chi_R\right)
    - \left(\Delta_v^s \cM \right) \chi_R \Big]
    + U \cdot  \nabla_v \chi_R \right\|_\beta
\end{equation*}
in several steps. Write first
\begin{align*}
  \left\| U \cdot \nabla_v\chi_R \right\|_\beta^2
  &= \int_{\R^d} \left| U \cdot \nabla_v\chi_R \right|^2
    \wv^{\beta} \cM(v) \dd v
    = \int_{\R^d} \left|  \mathbb{U}(v) \wv^{-\beta} v \cdot
    \nabla_v\chi_R \right|^2 \wv^{\beta} \cM(v) \dd v\\
  &\leq \Vert  \mathbb{U} \Vert_\infty
    \int_{\R^d} \left| v \cdot \nabla_v
    \chi_R \right|^2 \cM_\beta(v) \dd v \lesssim R^{-\alpha-\beta}.
\end{align*}
Then split the other term into
\begin{align*}
  &\left\| \cM^{-1} \Big[ \Delta_v^s \left(\cM\chi_R
    \right) - \left(\Delta_v^s \cM \right) \chi_R
    \Big] \right\|_\beta^2 \\
  &= \left\| \cM^{-1} \Big[ \Delta_v^s
    \left( \cM\chi_R \right) - \left(\Delta_v^s \cM
    \right) \chi_R \Big] \textbf{1}_{\vert v \vert \leq R}
    \right\|_\beta^2
    + \left\| \cM^{-1} \Big[ \Delta_v^s \left(
    \cM\chi_R \right) - \left( \Delta_v^s \cM \right)
    \chi_R \Big] \textbf{1}_{\vert v \vert \geq R} \right\|_\beta^2.
\end{align*}
When $|v|\le R$, write $v=Rw$ with $|w| \le 1$ and observe that
$\chi(w) = \chi(w')$ when $\vert w' \vert \leq 1$ to get,
\begin{align*}
  \left| \left[\Delta_v^s \left(\cM \chi_R \right) -
  \left( \Delta_v^s \cM \right) \chi_R \right](Rw)
  \right|
  &= C_{d,s} \left| \int_{\vert w' \vert \geq 1}
    \frac{\chi(w)-\chi(w')}{R^{d+a}
    \left|w - w'\right|^{d+2s}} \cM(Rw') \, R^d \dd w' \right| \\
  &\lesssim \int_{\vert w' \vert \geq 1}
    \frac{\left|\chi(w)-\chi(w')\right|}{R^{d+a+\alpha}
    \left|w - w'\right|^{d+2s}} \,
    \frac{{\rm d}w'}{\left|w'\right|^{d+\alpha}}
    \lesssim R^{-(d+2s+\alpha)},
\end{align*}
which yields (using $\beta=2s-\alpha$ and $\alpha>0$)
\begin{equation*}
  \left\| \cM^{-1} \Big[ \Delta_v^s \left( \cM\chi_R
    \right) - \left( \Delta_v^s \cM \right)
    \chi_R \Big] \textbf{1}_{\vert v \vert \leq R}
  \right\|_\beta^2
  \lesssim R^{-4s+\beta-\alpha}  \lesssim R^{-\alpha-\beta}.
\end{equation*} 
When $\vert v \vert \geq R$, we write
\begin{align*}
  & \Big[ \Delta_v^s \left( \cM \chi_R \right) -
  \left( \Delta_v^s \cM \right) \chi_R \Big](v) \\
  & = \int_{\R^d} \frac{\big[ \chi_R(v)-\chi_R(v') \big]}{\vert v - v'
  \vert^{d+2s}} \cM(v') \dd v' \\ 
  & = \int_{\vert v - v' \vert \leq \frac{|v|}{2}} \frac{\big[ \chi_R(v)-\chi_R(v') \big]}{\vert v - v'
  \vert^{d+2s}} \cM(v') \dd v' 
  + \int_{\vert v - v' \vert \geq \frac{|v|}{2}} \frac{\big[ \chi_R(v)-\chi_R(v') \big]}{\vert v - v'
  \vert^{d+2s}} \cM(v') \dd v'.
\end{align*}
Start with the first integral in the right hand side:
\begin{align*}
  & \left\vert \int_{\vert v - v' \vert \leq \frac{|v|}{2}}
    \frac{\big[ \chi_R(v)-\chi_R(v') \big]}{\left| v - v'\right|^{d+2s}}
    \cM(v') \dd v' \right\vert \\
  & \hspace{1cm} \lesssim \int_{\vert v - v' \vert \leq \frac{|v|}{2}}
    \frac{\sup_{B\left(v,\frac{|v|}{2}\right)} \left\vert
    \nabla_{v'}^2 \left[ \left(\chi_R(v)-\chi_R(v')\right)
    \cM(v') \right] \right\vert}{\left| v - v' \right|^{d+2s-2}}
    \dd v'\\
  & \hspace{1cm} \lesssim |v|^{2-2s}
    \sup_{B\left(v,\frac{|v|}{2}\right)} \left| \nabla_{v'}^2\left[
    \left(\chi_R(v)-\chi_R(v')\right) \cM(v') \right] \right|.
\end{align*}
One has
\begin{align*}
  &\sup_{B\left(v,\frac{|v|}{2}\right)}
    \left\vert D_{v'}^2\left( (\chi_R(v)-\chi_R(v')) \cM(v')
    \right) \right\vert \\
  &\lesssim \frac{|v|^{-d-\alpha}}{R^2} \sup_{v' \in
    B(v,\frac{|v|}{2})} \left\vert
    \chi''\left(\frac{v'}{R}\right) \right\vert + 
    \frac{|v|^{-d-\alpha-1}}{R} \sup_{v' \in B(v,\frac{|v|}{2})} \left\vert
    \chi'\left(\frac{v'}{R}\right) \right\vert + |v|^{-d-\alpha-2}.
\end{align*}
Consequently,
\begin{align*}
  &\left\vert \int_{\vert v - v' \vert \leq \frac{|v|}{2}}
    \frac{\big[ \chi_R(v)-\chi_R(v') \big]}{\vert v - v' \vert^{d+2s}}
    \cM(v') \, dv' \right\vert \\
  &\lesssim |v|^{2-2s} \left[
    \frac{|v|^{-d-\alpha}}{R^2} \sup_{v' \in B(v,\frac{|v|}{2})}
    \left\vert  \chi''\left(\frac{v'}{R}\right) \right\vert +
    \frac{|v|^{-d-\alpha-1}}{R} \sup_{v' \in B(v,\frac{|v|}{2})}
    \left\vert \chi'\left(\frac{v'}{R}\right) \right\vert
    + |v|^{-d-\alpha-2}\right]\\
  &\lesssim |v|^{-d-\alpha-2s} \left[ \frac{|v|^{2}}{R^{2}}
    \sup_{v' \in B(v,\frac{|v|}{2})} \left\vert  \chi''
    \left(\frac{v'}{R}\right) \right\vert+ \frac{|v|}{R}\sup_{v'
    \in B(v,\frac{|v|}{2})} \left\vert  \chi'
    \left(\frac{v'}{R}\right) \right\vert + 1 \right]
    \lesssim |v|^{-d-\alpha-2s},\\
\end{align*}
where we have used that $\chi'$ ans $\chi''$ have compact support
and $|v| \le 2 |v'|$ in this region.

Focus now on the second integral (using $\alpha > 0$)
\begin{equation*}
  \left\vert \int_{\vert v - v' \vert \geq  \frac{|v|}{2}}
    \frac{\big[ \chi_R(v)-\chi_R(v') \big]}{\vert v - v' \vert^{d+2s}} \cM(v')
    \dd v' \right\vert
  \leq \int_{|v - v'| \geq  \frac{|v|}{2}} \frac{\left\vert
      \chi_R(v) - \chi_R(v') \right\vert }{\left| v - v'
    \right|^{d+2s}} \cM(v') \dd v' \lesssim |v|^{-d-2s}.
\end{equation*}

As a conclusion,
\begin{align*}
  \left\| \cM^{-1} \Big[ \Delta_v^s \left(\cM\chi_R
  \right) - \left( \Delta_v^s \cM \right) \chi_R
  \Big] \textbf{1}_{\vert v \vert \geq R} \right\|_\beta^2
  \lesssim \int_{|v'| > R} |v'|^{2\alpha -4s}
  \wvp^{\beta-d-\alpha} \dd v' \lesssim R^{-\alpha-\beta},
\end{align*}
since $\beta = 2s - \alpha$. This concludes the proof.

\subsection{Proof of Hypothesis~\ref{hyp:scalinginfinite}}

The adjoint of $L$ is $L^* = \Delta_v^s - U \cdot \nabla_v$ and
following exactly the same arguments as in the proof of
Lemma~\ref{lem:ponctuel-phi} for the Fokker-Planck operator
yields
\begin{lemma}
  \label{lem:ponctuel-phi-lfp}
  The unique solution to the eigenvalue equation
  \begin{align*}
    - L^* \phi_{\eta}
      - i \eta (v \cdot \sigma) \phi_{\eta}= \mu(\eta)
      \wv^{-\beta} \phi_{\eta} \quad \text{ with } \quad 
      \int_{\R^d} \phi_{\eta}(v) \, \wv^{-\beta} \, \cM(v) \dd v = 1
  \end{align*}
  satisfies for any $R \ge 1$,
  \begin{align*}
    \| \phi_\eta \|_{L^\infty(B(0,R))} \lesssim_R 1 \quad \text{ and }
    \quad \| \Im \phi_\eta \|_{L^\infty(B(0,R))} \lesssim_R
    \max(\eta,\mu(\eta)),
  \end{align*}
  with constants depending only on $R$ and uniform in
  $\eta \to 0$.
\end{lemma}

Note that local fractional ellipticity results are present in \cite{serra_regularity_2016}.
We now come to the pointwise estimates on the rescaled
eigenvector. This is when $\alpha \leq 2 + \beta$, that is
$\alpha \leq 1 + s$. Observe indeed that when $\alpha > 1 + s$,
the scaling is diffusive, and the diffusion coefficient is
obtained by solving
\begin{equation*}
  \Delta_v^s(\cM F) +
  \nabla_v\cdot\left( U\,\cM F\right) = - (v \cdot \sigma) \cM(v),
\end{equation*}
with $\int_{\R^d} F(v)\, \cM_\beta(v) \dd v =0$.
\begin{lemma}
  \label{lem:ponctuel-Phi-levy}
  Assume $s \in (\frac12,1)$. There are
  $\eta_1 \in (0,\eta_0)$ small enough and $A$ and $C$ large
  enough so that
  \begin{align*}
    \forall \, \eta \in (0,\eta_1), \ \forall \, u \in \R^d, \quad
    \vert \Phi_\eta(u) \vert \lesssim \vert u \vert_\eta^{C \mu(\eta)}
    \quad \text{ and } \quad
    \vert \Im \Phi_\eta (u) \vert \lesssim \vert u
    \vert_\eta^{\min(1,\alpha)+\beta - C \mu(\eta)}.
  \end{align*}
\end{lemma}

\begin{proof}
  The rescaled equation for $\Phi_{\eta}$ is, using
  $U_\eta(u) :=\eta^{\frac{1-\beta}{1+\beta}}
  U(u\eta^{-\frac{1}{1+\beta}})$ and $2s-\beta = \alpha$,
  \begin{align*}
    - \eta^{\frac{\alpha}{1+\beta}} \Delta_u^s
    \Phi_{\eta}
    + U_\eta(u) \cdot \nabla_u \Phi_{\eta}
    - i (u\cdot \sigma) \Phi_{\eta}=
    \mu(\eta)|u|_\eta^{-\beta} \Phi_{\eta}.
  \end{align*}

  Multiply this equation by
  $\frac{\overline{\Phi_\eta}}{\vert \Phi_\eta \vert}$ and take
  the real part:
  \begin{equation*}
    - \eta^{\frac{\alpha}{1+\beta}} \Re
    \left( \frac{\overline{\Phi_\eta}}{\vert \Phi_\eta \vert}
      \Delta_u^s \Phi_{\eta}  \right)
    + U_\eta(u) \cdot \Re\left(\frac{\overline{\Phi_\eta}}{\vert
        \Phi_\eta \vert} \nabla_u \Phi_\eta \right) =
    \mu(\eta)|u|_\eta^{-\beta}  \vert \Phi_\eta \vert.
  \end{equation*}
  Using the Kato inequality
  $\Delta_u^s \vert \Phi_\eta \vert \geq \Re \left(
    \frac{\overline{\Phi_\eta}}{\vert \Phi_\eta \vert} \Delta_u^s
    \Phi_{\eta} \right)$ (see~\cite{BrPo} for the Laplacian
  and~\cite{ChVe} for the fractional Laplacian), one gets
  \begin{align*}
    - \eta^{\frac{\alpha}{1+\beta}} |u|_\eta^{\beta}
    \Delta_u^s \vert \Phi_\eta \vert + |u|_\eta^{\beta}
    U_\eta(u) \cdot \nabla_u \vert \Phi_\eta \vert
    - \mu(\eta) \vert \Phi_\eta \vert \leq 0.
  \end{align*}
  Then observe that the real function
  $F(u) = |u|_\eta^{C\mu(\eta)}$ satisfies for
  $|u| \ge A \eta^{\frac{1}{3}}$:
  \begin{align*}
    &- \eta^{\frac{\alpha}{1+\beta}} |u|_\eta^{\beta}
      \Delta_u^s F
      + |u|^\beta_\eta U_\eta(u) \cdot \nabla_u F - \mu(\eta) F\\
    &=  - \eta^{\frac{\alpha}{1+\beta}} |u|_\eta^{\beta}
      \Delta_u^s  F
      + C \mu(\eta) |u|_\eta^{C\mu(\eta)-2}
      \mathbb{U}_\eta(u) \, \vert u \vert^2
      - \mu(\eta) |u|_\eta^{C\mu(\eta)}
  \end{align*}
  where we have used that
  $U_\eta(u) = |u|_\eta ^{-\beta} \mathbb{U}_\eta(u) u $ with
  some $\mathbb{U}_\eta$ positive bounded from below
  (independently of $\eta$). We now estimate
  $\Delta_u^s \left( |\cdot|_\eta^{C\mu(\eta)} \right)(u)$. By
  scaling:
  \begin{equation*}
    \forall \, u \in \R^d, \quad
    \Delta_u^s \left( |\cdot|_\eta^{C\mu(\eta)} \right)(u) =
    \eta^{\frac{C\mu(\eta)-2s}{1+\beta}} \Delta_v^s \left(
      \wdot^{C\mu(\eta)} \right)\left(u \eta^{- \frac{1}{1+\beta}}\right).
  \end{equation*}
  We then estimate $\Delta_v^s \left( \wdot^{C\mu(\eta)} \right)$
  using
  \begin{align*}
    & \Delta_v^s \left( \wdot^{C\mu(\eta)} \right)(v)
      = C_{d,s} \int_{\R^d}
      \frac{\wvp^{C\mu(\eta)}-\wv^{C\mu(\eta)}}{\vert v' - v
      \vert^{d+2s}} \dd v',\\
    & \quad = C_{d,s}
      \int_{\vert v - v' \vert < \frac{|v|}{2}}
      \frac{\wvp^{C\mu(\eta)}-\wv^{C\mu(\eta)}}{\vert v' - v
      \vert^{d+2s}} \dd v' + C_{d,s}
      \int_{\vert v - v' \vert > \frac{|v|}{2}}
      \frac{\wvp^{C\mu(\eta)}-\wv^{C\mu(\eta)}}{\vert v' - v
      \vert^{d+2s}} \dd v'.
  \end{align*}  
  To control the first term in the right hand side, use that
  \begin{equation*}
    \wvp^{C\mu(\eta)}-\wv^{C\mu(\eta)} - \nabla_v\left(
      \wdot^{C\mu(\eta)} \right) \cdot (v'-v) \lesssim C \mu(\eta)
    \wv^{C\mu(\eta)-2}\vert v' - v \vert^{2}
  \end{equation*} 
  to get
  \begin{multline*}
    \int_{\vert v - v' \vert < \frac{|v|}{2}}
    \frac{\wvp^{C\mu(\eta)}-\wv^{C\mu(\eta)}}{\vert v' - v
    \vert^{d+2s}}\dd v'
    \lesssim \int_{\vert v - v' \vert < \frac{|v|}{2}}
    \frac{C \mu(\eta) \wv^{C\mu(\eta)-2}\vert v' - v
    \vert^{2}}{\vert v' - v \vert^{d+2s}}\dd v'\\
    \lesssim C \mu(\eta) \wv^{C\mu(\eta)-2} \int_{\vert v - v'
    \vert < \frac{|v|}{2}} \vert v' - v \vert^{2-2s-d} \dd v' 
    \lesssim C \mu(\eta) \wv^{C\mu(\eta)-2s}.
  \end{multline*}
  To control the second term, use that
  \begin{equation*}
    \left( \wvp^{C\mu(\eta)}-\wv^{C\mu(\eta)} \right) \lesssim C \mu(\eta)
    \wv^{C\mu(\eta)-1} \vert v' - v \vert
  \end{equation*} 
  to get (using here $s>\frac12$)
  \begin{align*}
    \int_{\vert v - v' \vert > \frac{|v|}{2}}
    \frac{\wvp^{C\mu(\eta)}-\wv^{C\mu(\eta)}}{\vert v' - v
    \vert^{d+2s}}\dd v'
    &\lesssim \int_{\vert v - v' \vert > \frac{|v|}{2}}
      \frac{C \mu(\eta) \wv^{C\mu(\eta)-1}\vert v' - v \vert}{\vert
      v' - v \vert^{d+2s}}\dd v'\\
    &\lesssim C \mu(\eta) \wv^{C\mu(\eta)-1} \int_{\vert v - v' \vert
      > \frac{|v|}{2}} \frac{\dd v'}{\vert v' - v \vert^{d+2s-1}}
    \\
    & \lesssim C \mu(\eta) \wv^{C\mu(\eta)-2s}.
  \end{align*}
  We therefore have (using the scaling)
  \begin{align*}
    \Delta_v^s \left( \wdot^{C\mu(\eta)} \right)(v)
    \lesssim C \mu(\eta) \wv^{C\mu(\eta)-2s} \quad
    \Longrightarrow \quad
    \Delta_u^s \left( |\cdot|_\eta^{C\mu(\eta)} \right)(u)
    \lesssim C \mu(\eta) |u|_\eta ^{C\mu(\eta)-2s}.
  \end{align*}

  This estimate implies, for some absolute constant $C_0>0$,
  \begin{align*}
    \eta^{\frac{\alpha}{1+\beta}} |u|_\eta^{\beta}\Delta_u^s F 
    & \le C_0 C \mu(\eta)\eta^{\frac{\alpha}{1+\beta}}
      |u|_\eta^{C\mu(\eta)+\beta-2s} \\
    & \le C_0 C \mu(\eta)|u|_\eta^{C\mu(\eta)}
      \eta^{\frac{\alpha}{1+\beta}}
      |u|_\eta^{-\alpha} 
      \lesssim C_0 C \mu(\eta)|u|_\eta^{C\mu(\eta)}
      \left( 1 + A^2 \right)^{-\frac{\alpha}2}
  \end{align*}
  in the region $\vert u \vert \geq
  A\eta^{\frac{1}{1+\beta}}$. As a consequence
  \begin{align*}
    &- \eta^{\frac{\alpha}{1+\beta}} |u|_\eta^{\beta}
      \Delta_u^s F + 
      \mathbb{U}_\eta(u) u \cdot \nabla_u F 
      - \mu(\eta) F \\
    & = - \eta^{\frac{\alpha}{1+\beta}} |u|_\eta^{\beta}
      \Delta_u^s F + C \mu(\eta)
      |u|_\eta^{C\mu(\eta)-2}
      \mathbb{U}_\eta(u) \,
      \vert u \vert^2 - \mu(\eta) |u|_\eta^{C\mu(\eta)}\\
    &\geq - C_0 C\mu(\eta)|u|_\eta^{C\mu(\eta)} \left(1 +
      A^2\right)^{-\frac{\alpha}2}
      + C \mu(\eta) |u|_\eta^{C\mu(\eta)-2} \left( \inf \mathbb{U}_\eta
      \right) \,
      \vert u \vert^2 - \mu(\eta) |u|_\eta^{C\mu(\eta)}\\
    &\geq C\mu(\eta) |u|_\eta^{C\mu(\eta)}
      \left[ -C_0 \left(1 + A^2 \right)^{-\frac{\alpha}2}+  |u|_\eta^{-2}
      \left( \inf \mathbb{U}_\eta \right)\, \vert u \vert^2 -
      C^{-1}\right] \\
    &\geq C\mu(\eta) |u|_\eta^{C\mu(\eta)}
      \left[ -C_0 \left(1 + A^2 \right)^{-\frac{\alpha}2} +
      \left(1+A^{-2} \right)^{-1}
      \left( \inf \mathbb{U}_\eta \right)\, - C^{-1}\right] \geq 0
  \end{align*}
  for $A$ and $C$ sufficiently large, and we deduce
  $|\Phi_\eta| \lesssim F$ on
  $|u| \ge A \eta^{\frac{1}{1+\beta}}$ and, for the same reasons
  as for the Fokker-Planck operator, the bound extends to any
  $u \in \R^d$.

  Taking now the imaginary part of the equation, one gets 
  \begin{align*}
    - \eta^{\frac{\alpha}{1+\beta}} |u|_\eta^{\beta}
    \Delta_u^s \vert \Im \Phi_\eta \vert
    + \mathbb{U}_\eta(u) u \cdot \nabla_u \vert
    \Im \Phi_\eta \vert - \mu(\eta) \vert \Im\Phi_\eta \vert
    \lesssim  |u|_\eta^{1+\beta + C\mu(\eta)}.
  \end{align*}
  Define then
  $\gamma := \min(\alpha,1)+\beta - C\mu(\eta) \in \, (0,2s)$ and
  the real function $G(u) := |u|_\eta^\gamma$. Note that
  $\gamma \in (0,2s)$ for $\eta$ small enough, which implies that
  $\Delta_u^s G$ makes sense. Write for
  $|u| \ge A \eta^{\frac{1}{3}}$
  \begin{align*}
    - \eta^{\frac{\alpha}{1+\beta}} |u|_\eta^{\beta}
    \Delta_u^s
    G + |u|^\beta_\eta U_\eta(u) \cdot \nabla_u G - \mu(\eta) G
    =  - \eta^{\frac{\alpha}{1+\beta}} |u|_\eta^{\beta}
    \Delta_u^s G
    + \gamma |u|_\eta^{\gamma-2} \mathbb{U}_\eta(u) \, \vert u
    \vert^2 - \mu(\eta) G.
  \end{align*}
  Let us estimate
  $\Delta_u^s \left( |\cdot|_\eta^{\gamma} \right)(u)$. Note that
  by scaling
  \begin{align*}
    \forall \, u \in \R^d, \quad
    \Delta_u^s \left( |\cdot|_\eta^{\gamma} \right)(u) =
    \eta^{\frac{\gamma-2s}{1+\beta}} \Delta_v^s \left(
    \wdot^{\gamma} \right)(u \eta^{- \frac{1}{1+\beta}} ).
  \end{align*}
  One then estimates $\Delta_v^s \left( \wdot^{\gamma} \right)$
  using
  \begin{align*}
    \Delta_v^s \left( \wdot^{\gamma} \right)(v)
    &= C_{d,s} \int_{\R^d}
      \frac{\left( \wvp^{\gamma}-\wv^{\gamma} \right)}{\vert v' - v
      \vert^{d+2s}} \dd v',\\
    &= C_{d,s} \int_{\vert v - v' \vert < \frac{|v|}{2}}
      \frac{\left( \wvp^{\gamma}-\wv^{\gamma} \right)}{\vert v' - v
      \vert^{d+2s}}\dd v'
      + C_{d,s}  \int_{\vert v - v' \vert > \frac{|v|}{2}}
      \frac{\left( \wvp^{\gamma}-\wv^{\gamma} \right)}{\vert v' - v
      \vert^{d+2s}} \dd v'.  
  \end{align*}
  Small $v$'s are fine since
  $\Delta_v^s \left( \wdot^{\gamma} \right)$ is locally
  bounded. Continue with large $v$. In the first integral,
  \begin{align*}
    \int_{\vert v - v' \vert < \frac{|v|}{2}}
    \frac{\left( \wvp^{\gamma}-\wv^{\gamma}
    \right)}{\vert v' - v \vert^{d+2s}}\dd v'
    & = \int_{\vert v - v' \vert < \frac{|v|}{2}}
      \frac{\wvp^{\gamma}-\wv^{\gamma} - \nabla_v\left(
      \wdot^\gamma \right)(v)(v'-v) }{\vert v' - v
      \vert^{d+2s}}\dd v' \\
    &\lesssim \int_{\vert v - v' \vert < \frac{|v|}{2}}
      \frac{\sup_{z \in B(v,\frac{|v|}{2})} \left\vert \nabla^2 \left(
      \wdot^\gamma \right)(z) \right\vert \vert v' - v
      \vert^{2}}{\vert v' - v \vert^{d+2s}}\dd v'
    \lesssim \vert v \vert^{\gamma - 2s}.
  \end{align*}  
  The second integral may be estimated from above using that
  $\vert v - v' \vert > \frac{\vert v \vert}{2}$ implies
  $\vert v - v' \vert > \frac{\vert v' \vert}{3}$,
  \begin{align*}
    \int_{\vert v - v' \vert > \frac{|v|}{2}}
    \frac{\wvp^{\gamma}-\wv^{\gamma}}{\vert v' - v \vert^{d+2s}}
    \dd v'
    & \lesssim \int_{\vert v - v' \vert > \frac{|v|}{2}}
      \frac{\lfloor v - v' \rceil^\gamma}{\vert v' - v
      \vert^{d+2s}}\dd v' \lesssim \vert v \vert^{\gamma - 2s}.
  \end{align*}  
  From this, we deduce
  $\Delta_u^s ( |\cdot|_\eta^{\gamma} )(u) \lesssim
  \eta^{\frac{\gamma-2s}{1+\beta}} \lfloor u \eta^{-
    \frac{1}{1+\beta}} \rceil^{\gamma-2s} = |u|_\eta^{\gamma-2s}$
  which implies
  \begin{align*}
    \eta^{\frac{\alpha}{1+\beta}} |u|_\eta^{\beta}\Delta_u^s G
    \lesssim \eta^{\frac{\alpha}{1+\beta}} |u|_\eta^{\gamma +
    \beta -2s}
    \lesssim \left( 1 + A^2 \right)^{-\frac{\alpha}2}
    |u|_\eta^{\gamma}
  \end{align*}
  in the region $\vert u \vert \geq
  A\eta^{\frac{1}{1+\beta}}$. As a consequence, as previously,
  \begin{align*}
    - \eta^{\frac{\alpha}{1+\beta}} |u|_\eta^{\beta}
    \Delta_u^s
    G + |u|^\beta_\eta U_\eta(u) \cdot \nabla_u G - \mu(\eta) G
    \gtrsim |u|_\eta ^{\gamma}
  \end{align*}
    for $A$ sufficiently large and we deduce
  $|\Im \Phi_\eta| \lesssim G$ on
  $|u| \ge A \eta^{\frac{1}{1+\beta}}$ and, for the same reasons
  as for the Fokker-Planck operator, the bound extends to any
  $u \in \R^d$.
\end{proof}

\subsection{Rescaled drift force and limit equation}

We formally discuss the behaviour of the force $U_\eta$ when
$\eta$ goes to $0$: setting $v= u\eta^{- \frac{1}{1+\beta}}$
gives the equation
\begin{equation*}
  \eta^{\frac{a -\beta}{1+\beta}}\Delta_v^s \cM_\eta
  + \nabla_v\cdot\left(U_\eta\,\cM_\eta\right) = 0.
\end{equation*} 
Observe that when $u \neq 0$,
\begin{align*}
  & \eta^{\frac{\alpha}{1+\beta}}\Delta_v^s \cM_\eta (u)
  = -c_{\alpha,\beta} C_{d,s}  \eta^{\frac{\alpha}{1+\beta}} \int_{\R^d}
    \frac{\left( \vert u \vert_\eta^{-d-\alpha} - \vert u'
    \vert_\eta^{-d-\alpha}\right)}{\vert u - u' \vert^{d+2s}} \dd u'\\
  &= -c_{\alpha,\beta} C_{d,s} \eta^{\frac{\alpha}{1+\beta}}
    \int_{B(0,\varepsilon)} \frac{\left( \vert u \vert_\eta^{-d-\alpha} - \vert u'
    \vert_\eta^{-d-\alpha}\right)}{\vert u - u' \vert^{d+2s}} \dd u' 
   -c_{\alpha,\beta} C_{d,s} \eta^{\frac{\alpha}{1+\beta}}
    \int_{B(0,\varepsilon)^c} \frac{\left(\vert u \vert_\eta^{-d-\alpha} - \vert
    u' \vert_\eta^{-d-\alpha}\right)}{\vert u - u' \vert^{d+2s}} \dd u'.
\end{align*}
The second term in the right hand side goes to zero as
$\eta \to 0$ since the singularity around zero has been removed
from the integration domain. To deal with the first term,
decompose
\begin{equation*}
  \eta^{\frac{\alpha}{1+\beta}}
  \int_{B(0,\varepsilon)} \frac{\left(\vert u \vert_\eta^{-d-\alpha} - \vert u'
  \vert_\eta^{-d-\alpha}\right)}{\vert u - u' \vert^{d+2s}} \dd u' 
  = \eta^{\frac{\alpha}{1+\beta}}
  \int_{B(0,\varepsilon)} \frac{\vert u \vert_\eta^{-d-\alpha} }
  {\vert u - u' \vert^{d+2s}} \dd u' -
  \eta^{\frac{\alpha}{1+\beta}}
  \int_{B(0,\varepsilon)} \frac{\vert u'
  \vert_\eta^{-d-\alpha}}{\vert u - u' \vert^{d+2s}} \dd u'.
\end{equation*}
The first part goes to zero if $\varepsilon < |u|$. The second part
is written as
\begin{align*}
  \eta^{\frac{\alpha}{1+\beta}}
  \int_{B(0,\varepsilon)} \frac{\vert u'
  \vert_\eta^{-d-\alpha}}{\vert u - u' \vert^{d+2s}} \dd u' 
  & \sim_\varepsilon \vert u \vert^{-d-2s} 
    \eta^{\frac{\alpha}{1+\beta}}
    \int_{B(0,\varepsilon)}  \vert u'
    \vert_\eta^{-d-\alpha} \dd u' \\
  & = \vert u \vert^{-d-2s}
    \int_{B\left(0,\varepsilon \eta^{-\frac{\alpha}{1+\beta}}\right)}
    (1 + \vert v'\vert^2)^{-d-\alpha} \dd v'.
\end{align*}
Taking $\eta$ small then $\varepsilon$ small yields
\begin{equation*}
  \lim_{\eta \to 0, \varepsilon \to 0} \left( -c_{\alpha,\beta}
  \eta^{\frac{\alpha}{1+\beta}}
  \int_{B(0,\varepsilon)} \frac{\vert u \vert_\eta^{-d-\alpha} - \vert u'
  \vert_\eta^{-d-\alpha}}{\vert u - u' \vert^{d+2s}} \dd u'
  \right) 
  =  \frac{c_{\alpha,\beta}}{c_{\alpha,0}} \cdot
  \frac{1}{\vert u \vert^{d+2s}}.
\end{equation*}
Since
$\nabla_v( \vert v \vert^{-d-2s} v) = - 2s \vert v
\vert^{-d-2s}$, we deduce that
\begin{equation*}
  \lim_{\eta \to 0} \eta^{\frac{1-\beta}{1+\beta}}
  U\left(u\eta^{-\frac{1}{1+\beta}}\right) = U_\infty(u)  =
  \frac{c_{\alpha,\beta}}{2s c_{\alpha,0}} \frac{u}{\vert u
    \vert^{d+2s}} \vert u \vert^{d+\alpha} =
  \frac{c_{\alpha,\beta}}{2s c_{\alpha,0}} \vert u \vert^{-\beta} u.
\end{equation*}
This proves the scaling limit of the drift force.

From the rescaled equation for $\Phi_{\eta}$, we deduce that
$\Phi_\eta$ goes to $\Phi$, where $\Phi$ solves,
\begin{equation*}
  \frac{c_{\alpha,\beta}}{2 as c_{\alpha,0}}
  \frac{u }{\vert u \vert^{\beta}} \cdot \nabla_u \Phi
  - i (u\cdot \sigma) \Phi = 0 \quad \text{with}\quad
  \Phi(0) = 1 \quad \Longrightarrow \quad 
  \Phi(u) := \exp\left(i\frac{2s c_{\alpha,0}}{c_{\alpha,\beta}}
  \frac{|u|^\beta (u \cdot \sigma)}{1+\beta}\right). 
\end{equation*}

Thus, in the limit case $\alpha=1+s=2+\beta$,
$\Omega(u) = \lim_{\lambda \to 0, \ \lambda \not = 0} \frac{\Im
  \Phi \left( \lambda u \right)}{\lambda^{1+\beta}}$ satisfies,
\begin{equation*}
  \frac{c_{\alpha,\beta}}{2a c_{\alpha,0}} u \cdot \nabla_u \Omega
  =  (u\cdot \sigma)\vert u \vert^{\beta} \quad
  \Longrightarrow \quad \Omega(u) :=
  \frac{2s c_{\alpha,0}} {c_{\alpha,\beta}}
  \frac{|u|^\beta (u \cdot \sigma)}{1+\beta}.
\end{equation*}

\subsection{The particular case $\alpha = 2s$}

More explicit calculations are available when $\alpha = 2s$. In
this case $\beta=0$, $U(v) = c_0 v$ for some constant $c_0>0$,
and the eigenproblem is
\begin{align*}
  - \Delta_v^s \phi + c_0 v \cdot \nabla_v
  \phi - i \eta (v\cdot \sigma) \phi = \mu(\eta) \phi.
\end{align*}
Taking the Fourier transform (in the dual of the Schwarz space)
gives
\begin{align*}
  - \vert \xi \vert^{2s} \hat\phi - c_0 \xi \cdot \nabla_\xi
  \hat\phi + \eta \sigma \cdot \nabla_\xi \hat\phi  =
  (\mu(\eta) +c_0)\hat\phi
\end{align*}
or equivalently
\begin{align*}
  \left( \eta \sigma - c_0 \xi \right) \cdot \nabla_\xi
  \hat\phi = \left( \mu(\eta)+ c_0 + \vert \xi \vert^{2s} \right)
  \hat\phi. 
\end{align*}
The solution to this equation is given by
$\hat\phi = \delta_{c_0^{-1} \eta \sigma}$ and
$\mu(\eta) = - \vert c_0^{-1} \eta \sigma \vert^{2s} = c_0^{-2s}
\eta^{2s}$, which yields by inverse Fourier transform
$\phi_\eta(v) := \exp\left( i c_0^{-1} \eta (v \cdot \sigma)
\right)$. This agrees with the expression of $\Phi$ given above,
and allows us to compute $c_0 = \frac{1}{2s}$.

\section{Remarks and extensions}
\label{sec:appendix}

In Hypothesis~\ref{hyp:functional}, the equilibrium $\cM$ is an
explicit power law, and in particular is centered and even. We
discuss in this section the changes required for our proofs to
deal with more general $\cM$ that are (i) characterised by
\emph{asymptotic} power-law estimates rather than exact formulae,
and (ii) not necessarily even or centered. This means replacing
Hypothesis~\ref{hyp:functional} with:
\begin{namedhyp}{Hypothesis 1'}[Equilibria]
  \label{hyp:functional2}
  The equilibrium distribution satisfies
  \begin{equation}
    \label{eq:MpolyS}
    \cM = \wdot^{-(d+\alpha)} \mathcal{S}(v),
  \end{equation}
  where $\mathcal{S}$ is a slowly varying function, and the
  \textbf{generalised mass condition}~\eqref{eq:gen-mass}.
\end{namedhyp}

Slowly varying functions are non-vanishing measurable functions
that satisfy $\mathcal{S}(ax) \sim \mathcal{S}(x)$ as $x$ goes to
infinity, for any $a >0$. Some examples of slowly varying
functions are positive constants, functions that converge to
positive constants, logarithms and iterated logarithms.

\subsection{Equilibria characterised only asymptotically}
\label{subsec:asymp}

If one considers a \textit{centered} equilibrium
$\cM$ that satisfies Hypothesis~1', the proof of Theorem~\ref{theo:main} in Section~\ref{sec:duality} and the proof of
Lemma \ref{lem:existencespectral} in Section \ref{sec:branch} are
essentially unchanged. The formulas for $\mu_0$ and $\kappa$ in
Lemmas~\ref{lem:ratespectral} and \ref{lem:diffcoeff} are
slightly modified, and rely on the existence of a scaling limit
of
$\eta^{-\frac{d+\alpha}{1+\beta}} \cM ( u
\eta^{-\frac{1}{1+\beta}})$ as $\eta \to 0$, which follows from
Hypothesis~1'. Everything else remains unchanged and the
structures of the proofs in Sections \ref{sec:ratespectral} and
\ref{sec:diffcoeff} are the same. Rates of convergence will
depend on the form of $\mathcal S$.

\subsection{Non-centered equilibria}
\label{subsec:non-cent}

When the microscopic equilibrium $\cM(v)$ is not centered, it
results in a drift in the macroscopic equation. Our approach
however allows us to tackle such a situation, with the following
changes depending on whether this macroscopic drift is of higher,
comparable or smaller order than the resulting (fractional)
macroscopic diffusion. In view of Theorem~\ref{theo:main} in the
centered situation, we expect a macroscopic diffusion of order
$\zeta(\alpha,\beta) =
\min\left(2,\frac{\alpha_++\beta}{1+\beta}\right)$, and therefore
we expect the drift to be dominant when $\alpha >1$ and dominated
when $\alpha \in [0,1)$, with a borderline case at
$\alpha=1$. Observe that $\alpha=1$ is also the threshold for the
absolute convergence of the integral
$\int_{\R^d} (v\cdot \sigma) \cM(v) \dd v$ defining the
macroscopic drift.

Consider a solution $f$ in
$L^\infty([0,+\infty); L^2_{x,v}(\cM^{-1}))$ to~\eqref{eq:kinetic} and denote
\begin{equation*}
  f_\varepsilon(t,x,v) :=
  f\left( \frac{t}{\theta(\varepsilon)}, \frac{x}{\varepsilon}
    + \frac{\bar v_\varepsilon  t}{\theta(\varepsilon)}, v
  \right) \in L^\infty_t([0,+\infty); L^2_{x,v}(\cM^{-1}))
\end{equation*}
where $\varepsilon >0$ and $\theta(\varepsilon)$ is defined
in~\eqref{eq:scaling-function}, and where the \emph{velocity
  corrector} $\bar v_\varepsilon$ is defined by
\begin{equation}
  \label{eq:coeffdrift}
  \bar v_\varepsilon :=
  \begin{cases}
    \ds \frac{\ds \int_{\R^d} v \cM(v) \dd v}{\ds
    \int_{\R^d}\cM(v) \dd v}
    &\text{ when } \alpha > 1, \\[10mm]
    \ds  
    \left( \lim_{R \to \infty} \frac{1}{\ln(R)} \frac{\ds
    \int_{\R^d} v \chi_R(v)
    \cM(v) \dd v}{\ds \int_{\R^d} \chi_R(v)
    \cM(v) \dd v} \right) \frac{\vert\ln(\varepsilon)\vert}{1+\beta}
    &\text{ when } \alpha =1, \\[10mm]
    \ds 0
    & \text{ when } \alpha \in [0,1).
  \end{cases}
\end{equation}
The equation satisfied by $f_\varepsilon$ is
\begin{equation}
  \label{eq:gvarbiais}
  \theta(\varepsilon) \partial_t f_\varepsilon +
  \varepsilon (v - \bar v_\varepsilon) \cdot \nabla_x f_\varepsilon
  = \cL f_\varepsilon.
\end{equation}
With this definition of $f_\varepsilon$, Theorem~\ref{theo:main}
holds and yields the (fractional) diffusive limit of
$f_\varepsilon$. The changes in the proofs are as follows. The
arguments presented in Section \ref{sec:duality} are essentially
unchanged with a few modifications to obtain the scaling of the
eigenvalue resulting from \eqref{eq:gvarbiais}. We chose
$\bar v_\varepsilon$ in such a way that the dominant eigenmode
has the scaling obtained in Lemmas \ref{lem:ratespectral} and
\ref{lem:diffcoeff}. The new spectral problem to be considered in
the modified Lemma~\ref{lem:existencespectral} is
\begin{equation*}
  - L^* \phi_{\eta}
  - i \eta \left[ (v-\bar v_\varepsilon) \cdot \sigma \right]
  \phi_{\eta}= \mu(\eta)
  \wv^{-\beta} \phi_{\eta} \quad \text{ with } \quad 
  \int_{\R^d} \phi_{\eta}(v) \, \cM_\beta(v) \dd v = 1.
\end{equation*}
Line-by-line technical modifications are needed in the proof of
Lemmas~\ref{lem:ratespectral} and~\ref{lem:diffcoeff} due to the
additional drift but the procedure and method are preserved and
we do not repeat the arguments. Let us just explain why we define
the correction velocity $\bar{v}_\varepsilon$ in this way. The
spectral projector estimate of Section~\ref{sec:branch} follows
the same procedure, with \eqref{eq:F} replaced by
\begin{equation}
  \label{eq:Fbiais}
  - L^* F - i \eta \left[ (v-\bar v_\varepsilon) \cdot \sigma \right] F
  - z \wv^{-\beta} F = (v-\bar v_\varepsilon) \cdot \sigma.
\end{equation}
The $L^2$ estimate is unchanged and the crucial estimate
\eqref{eq:estim-m} remains true as long as 
\begin{equation*}
  q(R) := \int_{\R^d}
  \left[ v-\bar v_\varepsilon \right] \chi_R(v) \cM(v) \dd v \quad
  \text{ at } \quad R:=\eta^{-\frac{1}{1+\beta}}
\end{equation*}
is small compared with $\textsc{r}_1\eta^{-1}\Theta(\eta)$, when
$\textsc{r}_1$ is large enough. This implies that the influence
of the drift is smaller than the size of the fluid mode, which is
of order $\eta^{-1}\Theta(\eta)$. Recall that
\begin{equation}
  \label{eq:cases-theta}
  \frac{\Theta(\eta)}{\eta} :=
  \begin{cases}
    \eta
    &\text{when } \alpha > 2 + \beta,\\[2mm]
    \eta \vert \ln(\eta) \vert
    &\text{when } \alpha = 2 + \beta,\\[2mm]
    \eta^{\frac{\alpha - 1}{1+\beta}} &\text{when } 0 \le \alpha
    < 2 + \beta.
  \end{cases}   
\end{equation}

One can then prove that for all $\alpha \ge 0$, one has
$q(\eta^{-\frac{1}{1+\beta}}) \lesssim \eta^{\frac{\alpha -
    1}{1+\beta}}$, which proves that
$q(\eta^{-\frac{1}{1+\beta}})$ is small compared with
$\textsc{r}_1\eta^{-1}\Theta(\eta)$ when $\textsc{r}_1$ is large
enough.

\subsection{More general velocity fields}

One could replace the transport operator $v \cdot \nabla_x$ by a
more general $a(v) \cdot \nabla_x$, where $a$ is odd. All our
results and proofs can be extended, even though the scalings
found may be changed since the scaling of $\ell(R)$
in~\eqref{eq:lR} will be different. If $a(v)$ scales like
$\vert v \vert^\delta$, redoing the computations as in Section
\ref{sec:duality} and Section \ref{sec:branch} then one would
find
\begin{equation}
  \label{eq:formula-theta}
  \Theta(\eta) :=
  \begin{cases}
    \eta^2
    &\text{when } \alpha > 2\delta + \beta,\\[2mm]
    \eta^2 \vert \ln(\eta) \vert
    &\text{when } \alpha = 2\delta + \beta,\\[2mm]
    \eta^{\frac{\alpha+\beta}{\delta+\beta}} &\text{when }
    0 \le \alpha < 2\delta + \beta.
  \end{cases}
\end{equation}

An example is given by relativistic particles, for which
$a(v) := c \frac{v}{\sqrt{c^2+ v^2}}$, where $c$ is the speed of
light. Such transport operators are relevant to special
relativity, see for instance~\cite{Stewart} in physics
and~\cite{Glassey} in mathematics. There, $\Theta$ is given
by~\eqref{eq:formula-theta} with $\delta=0$.

\subsection{Kinetic Fokker-Planck equation with non gradient
  confining force}

All the results we obtain for the Fokker-Planck equation with
gradient force can be extended to Fokker-Planck operators with
non-gradient confining force at little expense. We chose not to
present this more general setting in the core of the paper to
stay consistent with the clean and simple
Hypothesis~\ref{hyp:functional} and to help with readability. It
is however possible to consider
\begin{equation*}
  \mathcal{L}(f)=
  \Delta_vf + \nabla_v\cdot\left(U\,f\right) \quad \text{
    where $U$ satisfies } \quad 
  \Delta_v \cM+ \nabla_v\cdot\left(U\,\cM\right)
  = 0,
\end{equation*}
provided that quantitative bounds are available on $U$ to ensure
it is comparable to the drift in the Fokker-Planck operator. The
analysis is then similar.

\subsection{The case $\alpha < 0$}
\label{ss:negative}

Assume in this subsection that $\beta>0$. Observe that the natural
condition for Hypothesis~\ref{hyp:coercivity} (weighted
coercivity inequality) to hold is $\alpha + \beta > 0$ and
include the cases of negative values of $\alpha$, and indeed the
construction of the fluid mode in
Lemmas~\ref{lem:existencespectral} and~\ref{lem:ratespectral} is
valid for $\alpha \in (-\beta,0)$. However, our main result,
Theorem~\ref{theo:main}, assumes that $\alpha \geq 0$ and this
restriction comes from the convergence estimates in
Section~\ref{sec:duality}: in the case $\alpha < 0$, it is not
possible to find initial conditions such that both error terms
$\partial_t E_1$ and $E_2$ vanish in the limit.

This obstacle is in fact structural. Let us consider the simplest
case of a scattering equation
\begin{equation*}
  \theta(\varepsilon) \partial_t h_\varepsilon + \varepsilon v
  \cdot \nabla_x h_\varepsilon = \nu(v) \left(  
    r_\varepsilon  - h_\varepsilon \right),
\end{equation*}
with $\nu(v) = \nu_0 \wv^{-\beta}$,
$\theta(\varepsilon) = \varepsilon^{\frac{\beta}{1+\beta}}$ and
$h_\varepsilon(0,x,v)$ radially symmetric in $v$ and satisfying
\begin{equation*}
  \hat h_\varepsilon(0,\xi,\varepsilon^{- \frac{1}{1+ \beta}} u)
  \to \hat H(\xi,u).
\end{equation*}
In the spirit of~\cite{M3} we compute the equation satisfied by
the Laplace-Fourier transform $\tilde r_\varepsilon(p,\xi)$
(Laplace in $t$ and Fourier in $x$):
\begin{equation*}
  \frac{1}{\theta(\varepsilon)} \left[ \int_{\R^d} \left( 1 -
      \frac{\nu(v)}{\theta(\varepsilon)p + i \varepsilon v \cdot
        \xi + \nu(v)} \right) \, \cM(v) \wv^{-\beta}\dd v \right]
  \tilde r_\varepsilon(p,\xi) = \int_{\R^d} \frac{\hat
    h_\varepsilon(0,\xi,v)\cM(v) \wv^{-\beta}
  }{\theta(\varepsilon)p + i \varepsilon v \cdot \xi + \nu(v)} \,
  \dd v
\end{equation*} 
Observe that (using that $\cM$ is even):
\begin{align*}
  & \frac{1}{\theta(\varepsilon)} \int_{\R^d} \left( 1 -
    \frac{\nu(v)}{\theta(\varepsilon)p + i \varepsilon v \cdot
    \xi + \nu(v)} \right) \, \cM(v) \wv^{-\beta}\dd v \\
  & \qquad \qquad = \int_{\R^d} \left( \frac{p + i  \varepsilon
    \theta(\varepsilon)^{-1}v \cdot \xi}{\theta(\varepsilon)p + i
    \varepsilon v \cdot \xi + \nu(v)} \right) \, \cM(v)
    \wv^{-\beta}\dd v \\
  & \qquad \qquad = \int_{\R^d} \left( \frac{p \left(  \theta(\varepsilon)p +
    \nu(v) \right) + \varepsilon^2 \theta(\varepsilon)^{-1} (v
    \cdot \xi)^2 }{\left(\theta(\varepsilon)p + \nu(v)\right)^2 +
    \left(\varepsilon v \cdot \xi\right)^2 } \right) \, \cM(v)
    \wv^{-\beta}\dd v
\end{align*}
and (using that $\hat h_\varepsilon(0,\xi,v)$ is radially
symmetric in $v$)
\begin{equation*}
  \int_{\R^d} \frac{\hat h_\varepsilon(0,\xi,v)\cM(v)
    \wv^{-\beta} }{\theta(\varepsilon)p + i \varepsilon v \cdot
    \xi + \nu(v)} \, \dd v 
  = \int_{\R^d} \frac{\hat h_\varepsilon(0,\xi,v)\cM(v)
    \wv^{-\beta} \left( \theta(\varepsilon)p  + \nu(v)
    \right)}{\left(\theta(\varepsilon)p + \nu(v)\right)^2 +
    \left(\varepsilon v \cdot \xi\right)^2} \, \dd v.
\end{equation*}
We then change variable
$v = \varepsilon^{- \frac{1}{1+ \beta}} u$ in these two integrals. Since then
$\nu(v) = \varepsilon^{\frac{\beta}{1+\beta}}\vert u
\vert_\varepsilon^{-\beta}$,
\begin{align*}
  & \int_{\R^d} \left( \frac{p \left(  \theta(\varepsilon)p + \nu(v)
    \right) + \varepsilon^2 \theta(\varepsilon)^{-1}
    (v \cdot \xi)^2 }{\left(\theta(\varepsilon)p +
    \nu(v)\right)^2 + \left(\varepsilon v \cdot
    \xi\right)^2 } \right) \, \cM(v) \wv^{-\beta}\dd
    v \\
  & = \int_{\R^d} \left( \frac{p \left(  \theta(\varepsilon)p +
    \varepsilon^{\frac{\beta}{1+\beta}}\vert u
    \vert_\varepsilon^{-\beta} \right) + \varepsilon^2
    \theta(\varepsilon)^{-1} (\varepsilon^{- \frac{1}{1+ \beta}}
    u \cdot \xi)^2 }{\left(\theta(\varepsilon)p +
    \varepsilon^{\frac{\beta}{1+\beta}}\vert u
    \vert_\varepsilon^{-\beta}\right)^2 + \left(\varepsilon
    \varepsilon^{- \frac{1}{1+ \beta}} u \cdot \xi\right)^2 }
    \right) \, \cM \left(\varepsilon^{- \frac{1}{1+ \beta}} u \right)
    \varepsilon^{\frac{\beta}{1+\beta}}\vert u
    \vert_\varepsilon^{-\beta} \varepsilon^{- \frac{d}{1+ \beta}}
    \dd u \\
  & = c_{\alpha,\beta} \, \varepsilon^{\frac{\alpha}{1+ \beta}}
    \int_{\R^d} \left( \frac{p \left( p + \vert
    u \vert_\varepsilon^{-\beta} \right) +  ( u \cdot \xi)^2
    }{\left( p +  \vert u \vert_\varepsilon^{-\beta}\right)^2 +
    \left( u \cdot \xi\right)^2 } \right)  \vert u
    \vert_\varepsilon^{-d-\alpha} \vert u
    \vert_\varepsilon^{-\beta}  \dd u\\
  & \underset{\varepsilon \to 0}{\sim} c_{\alpha,\beta} \,
    \varepsilon^{ \frac{\alpha}{1+ \beta}}   \int_{\R^d} \left(
    \frac{p \left( p + \vert u \vert^{-\beta} \right) +  ( u
    \cdot \xi)^2 }{\left( p +  \vert u \vert^{-\beta}\right)^2 +
    \left( u \cdot \xi\right)^2 } \right) \, \frac{{\rm d}
    u}{\vert u \vert^{d+\alpha+\beta}}
\end{align*}
and
\begin{align*}
  & \int_{\R^d} \frac{\hat h_\varepsilon(0,\xi,v)\cM(v)
    \wv^{-\beta} \left( \theta(\varepsilon)p  +
    \nu(v) \right)}{\left(\theta(\varepsilon)p +
    \nu(v)\right)^2 + \left(\varepsilon v \cdot
    \xi\right)^2} \, \dd v \\
  & = \int_{\R^d} \frac{\hat h_\varepsilon\left(0,\xi,\varepsilon^{-
    \frac{1}{1+ \beta}} u \right) \cM \left(\varepsilon^{- \frac{1}{1+
    \beta}}u \right) \varepsilon^{\frac{\beta}{1+\beta}}\vert u
    \vert_\varepsilon^{-\beta} \left( \theta(\varepsilon)p  +
    \varepsilon^{\frac{\beta}{1+\beta}}\vert u
    \vert_\varepsilon^{-\beta}
    \right)}{\left(\theta(\varepsilon)p +
    \varepsilon^{\frac{\beta}{1+\beta}}\vert u
    \vert_\varepsilon^{-\beta}\right)^2 + \left(
    \varepsilon^{\frac{\beta}{1+ \beta}} u \cdot \xi\right)^2} \,
    \varepsilon^{- \frac{d}{1+ \beta}} \dd u \\
  & = \int_{\R^d} \frac{\hat h_\varepsilon \left(0,\xi,\varepsilon^{-
    \frac{1}{1+ \beta}} u \right)  \varepsilon^{- \frac{d}{1+ \beta}}
    \cM \left( \varepsilon^{- \frac{1}{1+ \beta}} u \right) \vert u
    \vert_\varepsilon^{-\beta} \left(  p  + \vert u
    \vert_\varepsilon^{-\beta} \right)}{\left( p + \vert u
    \vert_\varepsilon^{-\beta}\right)^2 + \left(u \cdot
    \xi\right)^2} \,  \dd u \\
  & \underset{\varepsilon \to 0}{\sim} c_{\alpha,\beta} \,
    \varepsilon^{ \frac{\alpha}{1+ \beta}}   \int_{\R^d} \hat
    H(\xi,u)   \frac{ \vert u \vert ^{-\beta} \left(  p  +  \vert
    u \vert^{-\beta} \right)}{\left( p +  \vert u \vert
    ^{-\beta}\right)^2 + \left( u \cdot \xi\right)^2} \,
    \frac{{\rm d} u}{\vert u \vert^{d+\alpha}}.
\end{align*}
All in all, we deduce that $\tilde r_\varepsilon(p,\xi)$
converges towards $\tilde r(p,\xi)$ defined by
\begin{equation*}
  \left[ \int_{\R^d} \left( \frac{p \left( p + \vert u
          \vert^{-\beta} \right) +  ( u \cdot \xi)^2 }{\left( p +
          \vert u \vert^{-\beta}\right)^2 + \left( u \cdot
          \xi\right)^2 } \right) \,
    \frac{{\rm d} u}{\vert u \vert^{d+\alpha+\beta}} \right] \tilde r(p,\xi) =
  \int_{\R^d} \hat H(\xi,u)   \frac{ 
    \left(  p  +  \vert u \vert^{-\beta} \right)}{\left( p +
      \vert u \vert ^{-\beta}\right)^2 + \left( u \cdot
      \xi\right)^2} \, \frac{{\rm d} u}{\vert u
    \vert^{d+\alpha+\beta}}.
\end{equation*}
This is, written in Laplace-Fourier variables, an equation with
fractional derivative in space but which is non local in time,
instead of a fractional diffusion equation.
 
\bibliographystyle{abbrv}

\bibliography{biblio}
\bigskip

\end{document}